\documentclass[11pt]{amsart}
\usepackage{amssymb,bm,mathrsfs,xcolor}
\usepackage[all]{xy}
\usepackage{hyperref}
\hypersetup{colorlinks=true}
\newcommand{\map}[1]{\xrightarrow{#1}}

\newcommand{\dlim}{\varinjlim}
\newcommand{\iso}{\cong}
\newcommand{\define}{\stackrel{\mathrm{def}}{=}}

\DeclareMathOperator{\Hom}{Hom}

\DeclareMathOperator{\Aut}{Aut}

\DeclareMathOperator{\GL}{GL}
\DeclareMathOperator{\Sym}{Sym}

\newcommand{\Q}{\mathbb Q}
\newcommand{\Z}{\mathbb Z}
\newcommand{\R}{\mathbb R}
\newcommand{\C}{\mathbb C}

\newcommand{\A}{\mathbb A}
\newcommand{\co}{\mathcal O}

\newcommand{\chern}{\mathrm{ch}}

\newcommand{\DD}{\mathcal{D}}
\DeclareMathOperator{\mynull}{null}

\begin{document}
\author{Benjamin Howard}
\title[Pullback formulas for arithmetic cycles]{Pullback formulas for arithmetic cycles on orthogonal Shimura varieties}
\date{}
\address{Department of Mathematics\\Boston College\\ 140 Commonwealth Ave. \\Chestnut Hill, MA 02467, USA}
\email{howardbe@bc.edu}

\thanks{This research was supported in part by NSF grants  DMS-2101636 and DMS-1801905.}

\begin{abstract}
On an orthogonal Shimura variety, one has a collection of  arithmetic special cycles in the Gillet-Soul\'e arithmetic Chow group.
We describe how these cycles behave under pullback to an embedded orthogonal Shimura variety of lower dimension.
The bulk of the paper is devoted to cases in which the special cycles intersect the embedded Shimura variety improperly, which requires that we analyze  logarithmic  expansions of  Green currents on the deformation to the normal bundle of the embedding.
\end{abstract}

\maketitle

\theoremstyle{plain}
\newtheorem{theorem}{Theorem}[subsection]
\newtheorem{bigtheorem}{Theorem}[section]
\newtheorem{proposition}[theorem]{Proposition}
\newtheorem{lemma}[theorem]{Lemma}
\newtheorem{corollary}[theorem]{Corollary}
\newtheorem{problem}[theorem]{Problem}

\theoremstyle{definition}
\newtheorem{definition}[theorem]{Definition}
\newtheorem{hypothesis}[theorem]{Hypothesis}

\theoremstyle{remark}
\newtheorem{remark}[theorem]{Remark}
\newtheorem{example}[theorem]{\mathcal{L}xample}
\newtheorem{question}[theorem]{Question}

\numberwithin{equation}{subsection}

\numberwithin{equation}{subsection}
\renewcommand{\thebigtheorem}{\Alph{bigtheorem}}


\section{Introduction}



On an orthogonal Shimura variety $M$, one has a systematic supply of special cycles coming from embeddings of smaller orthogonal Shimura varieties.  These cycles are the subject of a series of conjectures of Kudla \cite{kudlaMSRI}, whose central theme is that they should be geometric analogues of the coefficients of  Siegel  theta functions.

In order to do arithmetic intersection theory with these cycles, one must endow them with Green currents. 
 One  construction of such Green currents was done by Garcia-Sankaran \cite{GS},  using ideas of Bismut \cite{Bis} and Bismut-Gillet-Soul\'e \cite{BiGSa}.
 The special cycles endowed with these currents define  arithmetic special cycles in the Gillet-Soul\'e arithmetic Chow group of $M$.  The goal of this paper is to show that these arithmetic special cycles behave nicely under pullbacks via embeddings $M_0 \to M$ of smaller orthogonal Shimura varieties, in the sense that the pullback of an  arithmetic special  cycle on $M$ is a prescribed linear combination of arithmetic special  cycles on $M_0$.

When an arithmetic  cycle intersects $M_0$ properly, its pullback to  can be defined in a naive way, and is easy to compute directly from the definitions.  Unfortunately, the intersections that arise in our setting are very rarely proper.  
For improper intersections, Gillet and Soul\'e define pullbacks using the moving lemma, which is poorly suited to any kind of explicit calculation.  One doesn't have any natural choice of rationally trivial cycle with which to move the special cycle, and even if one did, replacing an arithmetic special cycle by a rationally equivalent one would destroy all the nice properties of the special cycle and Green current that one started with.

Our approach to treating improper intersections is to use the work of Hu \cite{Hu}, which gives an alternate definition of arithmetic pullbacks based on Fulton's deformation to the normal cone approach to intersection theory.  
One can always specialize a cycle on $M$ to a cycle on the normal cone to $M_0 \subset M$.
As $M_0$ is smooth, the normal cone is canonically identified with  (the total space of) the normal bundle $N_{M_0/M} \to M_0$.  
Hu shows that there is an analogous specialization of Green currents.  
The core of this paper is the calculation of the specializations   of  Garcia-Sankaran Green currents to $N_{M_0/M}$, or at least the calculation of enough of them to deduce  the pullback formula.

\subsection{Statement of the main result}

Fix a  quadratic space $V$ of dimension $n+2 \ge 3$ over  a totally real number field $F$.  Assume that $V$ has signature $(n,2)$ at one embedding $\sigma : F \to \R$, but  is positive definite at every other  embedding.

From $V$ one can construct a Shimura datum $(G,\DD)$ in which $G$ is the restriction of scalars  to $\Q$ of either  $\mathrm{SO}(V)$ or $\mathrm{GSpin}(V)$, and $\DD$ is a hermitian symmetric domain of dimension $n$.
Fixing a sufficiently small compact open subgroup $K\subset G(\A_f)$, we obtain a quasi-projective Shimura variety $M$  over the reflex field $\sigma(F) \subset \C$ with complex points
\[
M(\C) = G(\Q) \backslash \DD\times G(\A_f) / K.
\]
For the rest of the introduction we assume that $V$ is anisotropic, so that $M$ is projective.

Fix a positive integer $d\le n+1$.
Given the data of a nonsingular symmetric matrix $T \in \Sym_d(F)$ and a $K$-fixed $\Z$-valued Schwartz function 
\[
\varphi \in S(\widehat{V}^d),
\]
one can define a special cycle $Z(T,\varphi)$ on $M$ of codimension $d$,   as in \cite{kudlaMSRI}.
After fixing a positive definite $v\in \mathrm{Sym}_d(\R)$, 
Garcia-Sankaran \cite{GS} construct a smooth form $\mathfrak{g}^\circ(T,v,\varphi)$  of type $(d-1,d-1)$ on the complex fiber of $M \smallsetminus Z(T,\varphi)$.  This form is locally integrable on $M(\C)$, and its associated current satisfies the Green equation 
\[
dd^c [ \mathfrak{g}^\circ(T,v,\varphi) ] + \delta_{Z(T,\varphi)} = [ \omega^\circ(T,v,\varphi) ] 
\]
for a smooth form $\omega^\circ(T,v,\varphi)$.  In particular, it  determines an arithmetic cycle class
\begin{equation}\label{intro cycle}
\widehat{C}_M(T,v,\varphi) =
( Z(T,\varphi) , \mathfrak{g}^\circ(T,v,\varphi) ) \in \widehat{\mathrm{CH}}^d (M) 
\end{equation}
in the Gillet-Soul\'e arithmetic Chow group.

Because our main focus is on the Green currents,  in this paper we work exclusively with the arithmetic Chow group of the canonical model over the reflex field.  No integral models will appear.  

Following \cite{GS}, in \S \ref{ss:canonical cycle construction}  we extend the definition of \eqref{intro cycle} to all $T\in \Sym_d(F)$, including the case $\det(T)=0$.
For the purposes of this introduction, we  say only that this extension makes use of a special hermitian line bundle
\[
\widehat{\omega} \in \widehat{\mathrm{Pic}}(M) \iso  \widehat{\mathrm{CH}}^1(M) .
\]
For example, in the degenerate case of the zero matrix $0_d\in \Sym_d(F)$, the definition is
\[
\widehat{C}(0_d ,v,\varphi) =   \varphi(0)\cdot   \big(
 \underbrace {\,  \widehat{\omega}^{-1} \cdots \widehat{\omega}^{-1} } _{d} 
 +   \big( 0 , -\log(\det(v)) \cdot  \Omega^{d-1} \big) \big)
\]
where the $\cdots$ on the right hand side is iterated arithmetic intersection, and $\Omega^{d-1}$ is the $d-1$ exterior power of the Chern form of $\widehat{\omega}^{-1}$.

Now suppose that our quadratic space is presented as an orthogonal direct sum $V=V_0\oplus W$ with $W$ totally positive definite and $\dim(V_0)\ge 3$.  In particular, $V_0$ has signature $(n_0,2)$ at the real embedding $\sigma:F\to \R$, and is positive definite at all other embeddings.    As such, it has its own Shimura datum $(G_0,\DD_0)$, and a choice of compact open $K_0 \subset G_0(\A_f)$ determines a Shimura variety $M_0$ over $\sigma(F) \subset \C$  with its own family of arithmetic special cycles
\[
\widehat{C}_{M_0}(T,v,\varphi_0)  \in \widehat{\mathrm{CH}}^d (M_0) .
\]

The inclusion $V_0 \to V$ induces an embedding of Shimura data $(G_0,\DD_0) \to (G,\DD)$.
We  choose $K_0$ and $K$ in such a way that $K_0 \subset G_0(\A_f) \cap K$, and so that the   induced  
$
i_0 : M_0 \to M
$
 is a closed immersion.
Our main result, stated in the text as Corollary \ref{cor:pullback}, is the following pullback formula for arithmetic special cycles.

\begin{bigtheorem}\label{BigThmA}
Fix a  $K$-fixed Schwartz function  
\[
\varphi =  \varphi_{0} \otimes \psi \in S( \widehat{V}_0^d ) \otimes S(\widehat{W}^d) \subset  S( \widehat{V}^d)
\]
with $\varphi_0$ fixed by $K_0$,  and both $\varphi_0$ and $\psi$ valued in $\Z$.
The pullback 
\[
i_0^* : \widehat{\mathrm{CH}}^d(M) \to \widehat{\mathrm{CH}}^d( M_0) 
\]
satisfies
\[
i_0^* \widehat{C}_M(T,v,\varphi) =  
 \sum_{ \substack{ T_0 \in \Sym_d( F ) \\ y\in W^d  \\  T_0+T(y) = T   }  }   \psi(y)   \cdot    \widehat{C}_{M_0}(T_0,v,\varphi_{0})
\]
for all $T\in \Sym_d(F)$ and all  positive definite $v\in \Sym_d(\R)$.
Here $T(y) \in \Sym_d(F)$ is the moment matrix of the tuple $y\in W^d$,  as in \eqref{moment matrix}.
 \end{bigtheorem}

\begin{remark}
If one forgets Green currents and works with the usual Chow groups of $M$ and $M_0$,  the above pullback formula appears  \cite{kudla21}.
\end{remark}

\begin{remark}
The constructions of Green currents in \cite{GS} are carried out on the Shimura varieties for special orthogonal groups of signature $((n,2) , (n+2,0), \ldots, (n+2,0) )$ as above, and also on  the Shimura varieties for unitary groups of signature $(( n,1) , ( n+1,0) , \ldots, (n+1,0) )$.  
There is no difficulty at all in proving the analogue of Theorem \ref{BigThmA} also in the unitary case, using exactly the same argument.  We have  restricted attention to the orthogonal case only for concreteness, and to avoid excessively cluttering the exposition. 
\end{remark}

We have not attempted to exhaust the methods, which can presumably be pushed farther.
For example, one would like a similar statement for noncompact Shimura varieties and integral models, 
as well as a formula expressing the intersection  of two arithmetic special cycles as a linear combination of other arithmetic special cycles.  There should  be similar results for the Shimura varieties associated to quadratic spaces with signature $(n,2)$ at several archimedean places.  
One could also try to prove similar formulas for other Green currents, such as those of Funke-Hofmann \cite{FH}.
The author hopes to address some of these questions in future work.


\subsection{Outline of the paper}


In \S \ref{s:arithmetic intersection theory} we recall what we need from Hu's thesis \cite{Hu}.
Suppose  $X_0 \to X$ is a closed immersion of complex manifolds.  
If $G$ is a Green current for an analytic cycle $Z$ on $X$,  one would like to construct a Green current  $\sigma_{X_0/X}(G)$ for the specialization $\sigma_{X_0/X}(Z)$ of $Z$ to the normal bundle $N_{X_0/X}$.

To see how this works, denote by  $\tilde{X}$ the deformation to the normal bundle $N_{X_0/X}$.  It comes with a  holomorphic function 
 $\tau : \tilde{X} \to \C$ whose fiber over $t\in \C$ we denote by $\tilde{X}_t$.  
 The fiber at $t=0$ is $\tilde{X}_0 = N_{X_0/X}$. 
 If $t\neq 0$ there is a canonical identification $X\iso \tilde{X}_t$, and hence a closed immersion $j_t : X \iso \tilde{X}_t \hookrightarrow \tilde{X}$.  In this way we obtain a current
 $j_{t*}G$ on $\tilde{X}$.  Hu's idea is to look for  a logarithmic expansion
\[
j_{t*}G = \sum_{i \ge 0} G_i(t) \cdot (\log |t| )^i 
\]
whose coefficients  $G_i(t)$ are currents on $\tilde{X}$  with the property that each function  $t \mapsto G_i(t)$ extends continuously to $t=0$, and define $\sigma_{X_0/X}(G)$ in terms of the current $G_0(0)$.
In this generality such a logarithmic expansion need not exist.  If it does exist the logarithmic expansion will not be unique, but  $ \sigma_{X_0/X} (G)$  is independent of the choice.

In Theorems \ref{thm:hu} and \ref{thm:arithmetic specialization} we quote two results of Hu.
The first  guarantees the existence of logarithmic expansions (and hence specializations to the normal bundle) for certain currents on $X$.  The second shows that if $X$ is a projective variety over a number field, one can use the specialization of cycles and Green currents to define a morphism from the arithmetic Chow group of $X$ to the arithmetic Chow group of $N_{X_0/X}$.  This morphism of arithmetic Chow groups  agrees with  the one induced by  pullback through the composition $N_{X_0/X} \to X_0 \to X$.

Now suppose  $L$ a hermitian line bundle on $X$.
In  \S \ref{s:green} we recall from  \cite{GS} a construction that takes a tuple $s=(s_1,\ldots, s_d)$ of global sections of  $L^\vee$  and produces a Green current $\mathfrak{g}^\circ (s)$ for the analytic cycle  $Z(s)$ defined by $s_1=\cdots = s_d=0$.
The central problem is to understand the specialization $\sigma_{X_0/X}(\mathfrak{g}^\circ (s) )$.
 For our applications it is enough to assume that $s=(p,q)$ is the concatenation of tuples $p=(p_1,\ldots, p_k)$ and $q=(q_1,\ldots, q_\ell)$, arranged so that the cycle $Z(p)$ meets $X_0$ properly, while $X_0 \subset Z(q)$.

In essence, our  strategy is to show that the star product formula
 \begin{equation}\label{intro star}
 \mathfrak{g}^\circ(s) = \mathfrak{g}^\circ(p) \star \mathfrak{g}^\circ(q_1) \star\cdots \star \mathfrak{g}^\circ(q_\ell)
 \end{equation}
 of Garcia-Sankaran implies the analogous star product formula
 \begin{equation}\label{intro star 2}
 \sigma_{X_0/X}( \mathfrak{g}^\circ(s))  = \sigma_{X_0/X}(  \mathfrak{g}^\circ(p) )  \star \sigma_{X_0/X}( \mathfrak{g}^\circ (q_1)  ) \star \cdots \star  \sigma_{X_0/X}( \mathfrak{g}^\circ (q_\ell) )
 \end{equation}
for specializations,  and then compute each specialization on the right individually.
As $Z(p)$ intersects $X_0$ properly, the specialization $\sigma_{X_0/X}(  \mathfrak{g}^\circ(p) ) $ is easy to compute.
To compute the specialization of $\mathfrak{g}^\circ(q_i)$ one must  do more work, but the idea is imitate the construction of the current  with $X$ replaced by the deformation to the normal bundle $\tilde{X}$, and use the resulting current on $\tilde{X}$ to find an 
explicit logarithmic expansion for  $ \mathfrak{g} ^\circ(q_i)$ 

It is not obvious to the author that Hu's specialization to the normal bundle is compatible with $\star$ products in general;
that is to say, deducing \eqref{intro star 2} from \eqref{intro star} seems to require using particular properties of the Green currents $\mathfrak{g}^\circ(s)$.
\emph{After} passing to  arithmetic Chow groups the compatibility of specialization with star products follows from Theorem  \ref{thm:arithmetic specialization}, but on the level of arithmetic cycles (that is, before passing to their rational equivalence classes) things are more complicated.   When we apply the calculations  described above to the case of orthogonal Shimura varieties, the complex manifold $X$ is not the Shimura variety $M(\C)$, but rather the hermitian symmetric domain $\DD$ that uniformizes it.  As $\DD$ does not have an arithmetic Chow group in any useful sense, our calculations must be carried out before passing to rational equivalence classes of arithmetic cycles.

To prove the compatibility of specializations with star products we need to show that the Green currents in question admit logarithmic expansions of an especially nice form; this is essentially Lemma \ref{lem:nice expansions}, which is the core of the proof of Proposition \ref{prop:special star}.   While Hu's proof of Theorem  \ref{thm:arithmetic specialization} provides a general construction of logarithmic expansions, the expansions one gets from this method are quite unpleasant.
For example,  if one starts with a current $G$ that is represented by a locally integrable form, the currents $G_i(t)$ produced by Hu's construction will typically not have this form (Hu's construction of logarithmic expansions uses an inductive process, and each step of the induction introduces $\delta$-currents that are not represented by smooth forms). 
It is  essential to our method that we find logarithmic expansions for $\mathfrak{g}^\circ(s)$ that are better behaved than  those obtained by tracing through Hu's proof of  Theorem  \ref{thm:arithmetic specialization}.

We emphasize that all the  calculations in \S \ref{s:green} are carried out in the setting of an arbitrary complex manifold $X$, and don't involve orthogonal Shimura varieties (or any Shimura varieties) at all.

Finally, in \S \ref{s:ortho} we define the precise arithmetic cycle classes $\widehat{C}(T,v,\varphi)$ appearing in Theorem \ref{BigThmA}, and apply the general constructions of the preceding sections to the case of orthogonal Shimura varieties.
The strategy for proving Theorem \ref{BigThmA} is  to use the explicit calculation of specializations of cycles and Green currents to show that the arithmetic cycles appearing in the equality of that theorem become equal after pullback via the bundle map $N_{M_0/M} \to M_0$.  By Proposition \ref{prop:injective normal}, they must have been equal before the pullback as well.

Something like  specializations to the normal bundle of Green currents  were  computed in \cite{AGHMP}, but for Bruinier's  \cite{Br02} Green functions for  special divisors.
When working only with arithmetic divisors,  the situation is much simpler, and one doesn't really need specialization to the normal bundle at all.
The codimension one arithmetic Chow group can be identified with the arithmetic Picard group, and pullback then agrees with the  naive notion of pullback of hermitian line bundles.  
One can compute arithmetic pullbacks (even in cases of improper intersection) more directly using this interpretation.  
This is the approach taken in \cite{BHY}, which is the unitary Shimura variety analogue of \cite{AGHMP}. 
 To compute pullbacks for higher codimension arithmetic Chow groups, the author knows of no method other than the specialization to the normal bundle approach developed here.

As a final remark, we note  that the proof of Theorem 4.13 of \cite{BiGSb}, whose statement involves pullbacks of arithmetic cycle classes via closed immersions,  also makes use of the deformation to the normal bundle.  
The connection between the methods used in \cite{BiGSb} and the logarithmic expansions of \cite{Hu} are not obvious  to the author.


\subsection{Acknowledgements} 


The author thanks Jan Bruinier, Luis Garcia, Jos\'e Burgos Gil,  and Siddarth Sankaran for helpful discussions, some of which took place during the AIM workshop  \emph{Arithmetic intersection theory on Shimura varieties}, January 4-8, 2021.
The author thanks AIM for facilitating these conversations.  The author also thanks the anonymous referees for helpful comments and corrections.


\section{Arithmetic specialization to the normal bundle}
\label{s:arithmetic intersection theory}


In this section we recall some results from Hu's thesis \cite{Hu},  and restate them in the precise form they will be needed later.


\subsection{Logarithmic differential forms}


We recall some definitions and results from \cite{BurgosGreen}.
Let $X$ be a complex manifold of dimension $n=\dim(X)$.

\begin{definition}
If  $Z\subset X$ is any analytic subset (i.e., a reduced closed analytic subspace), 
a \emph{resolution of singularities} 
\begin{equation}\label{resolution def}
r:(X',Z') \to (X,Z)
\end{equation}
 is a complex manifold $X'$ together with a proper surjection $r:X' \to X$,  such that  $Z'=r^{-1}(Z)$  is a divisor with normal crossings, and  $r$ restricts to an isomorphism  $X'\smallsetminus Z' \iso X \smallsetminus Z$.
\end{definition}

\begin{remark}
A  resolution of singularities  always exists by Theorem 3.3.5  of \cite{Kollar}, extended to analytic spaces as in \S 3.4.4 of \emph{loc.~cit.}.   See also \cite{Wlod}.  For quasi-projective varieties, this is  Hironaka's theorem.
\end{remark}

Denote by 
\[
E^\bullet_X=\bigoplus_{k\ge 0} E^k_X
\]
 the graded $\C$-algebra of smooth  differential forms on $X$.
 For  $\phi \in E^\bullet_X$, let 
\[
\phi_{[k]}\in E^k_X
\]
be its component in degree $k$.  
Let ${}_cE^\bullet_X \subset E^\bullet_X$ be the graded subspace of compactly supported forms, and 
denote by  $D_X^k$  be the space  of currents dual  to ${}_cE_X^{2n-k}$.
There is a  canonical injection  $E_X^k \to D_X^k$,    denoted  $g\mapsto [g]$,  defined by 
\begin{equation}\label{currents}
[g] (\phi) = \int_X g\wedge \phi.
\end{equation}
When no confusion can arise, we sometimes omit the brackets, and write $g$ both for the form and its associated  current.

Given a divisor with normal crossings $Z \subset X$, let 
\begin{equation}\label{divisor logforms}
E^\bullet_X(\log Z ) \subset E^\bullet_{X\smallsetminus Z}
\end{equation}
be the graded subalgebra of forms with logarithmic growth along $Z$ as in \S 1.2 of  \cite{BurgosGreen}:
in local coordinates on  $X$ such that  $Z$ is given  by the equation $z_1 \cdots z_m=0$, the forms of logarithmic growth are generated, as an algebra over the ring of   smooth forms,  by 
\[
\log |z_i| , \frac{dz_i}{z_i} , \frac{d \bar{z}_i}{\bar{z}_i } \quad \mbox{for } 1\le i \le m.
\]

Now let $Z\subset X$ be any analytic subset of codimension $d>0$.
A  choice of resolution of singularities \eqref{resolution def} determines a subspace  
\begin{equation}\label{logforms}
E^\bullet_X(\log Z ) \subset E^\bullet_{X\smallsetminus Z},
\end{equation}
consisting of those forms whose pullback to $E^\bullet_{X' \smallsetminus Z'}$  has logarithmic growth along the normal crossing divisor $Z'$.
Although the notation does not indicate it, this subspace genuinely depends on the choice of resolution of singularities.

Denote by 
\[
E^\bullet_X(\mynull Z) \subset E^\bullet_X
\]
the graded subspace of forms whose pullback to (the smooth locus of)  $Z$ vanishes, and let ${}_cE^\bullet_X(\mynull Z)$ be the graded subspace of compactly supported such forms. The inclusion 
$
{}_cE^\bullet_X(\mynull Z) \subset {}_c E^\bullet_X
$
induces  a  canonical surjection
\[
D^\bullet_{X} \to D^\bullet_{X/Z},
\]
where  $D^k_{X/Z}$ is the space of currents dual to ${}_cE^{2n-k}_X(\mynull Z)$.

\begin{proposition}[Burgos]
 For any $g\in E^\bullet_X(\log Z )$ and  $\phi \in E_X^\bullet(\mynull Z)$, 
the form $g\wedge \phi$ is locally integrable on $X$.  The integral   \eqref{currents}  defines an   injection
\[
E^\bullet_X(\log Z ) \map{ g\mapsto [g] }  D_{X/Z}^\bullet
\]
satisfying $\partial [g] = [\partial g]$, and similarly for $\overline{\partial}$.
 \end{proposition}

\begin{proof}
See Corollaries 3.7 and 3.8 in \cite{BurgosGreen}.  
\end{proof}

\begin{remark}\label{rem:log currents}
 If $k < 2d$ then any form in $E^{2n-k}_X$ has trivial pullback to $Z$, and hence 
$
 D^k_{X/Z}  = D^k_{X} .
$
In particular, we obtain an injection 
\[
E^k_X(\log Z ) \map{ g\mapsto [g] }  D_{X}^k.
\]
For  $g\in E^k_X(\log Z)$ with $k<2  d-1$ we have  $\partial [g]=[\partial g]$ in $D^{k+1}_X$, and similarly for $\overline{\partial}$.
\end{remark}

\begin{definition}\label{def:burgos forms}
Suppose that $U$ is a smooth quasi-projective complex variety.
By a \emph{smooth compactification} of $U$  we mean a smooth projective variety $U^*$, a divisor with normal crossings $\partial U^* \subset U^*$, and an isomorphism $i :U \iso U^* \smallsetminus \partial U^*$.  The  smooth compactifications of $U$  form a cofiltered category in a natural way, allowing us to form the graded subalgebra
\[
E^\bullet_{\log} (U) = \dlim_{ ( U^* , \partial U^* , i) } E^\bullet_{U^*} ( \log \partial U^*) \subset
E^\bullet_U
\]
of \emph{forms with logarithmic singularities along $\infty$}; see Definition 1.2 of \cite{BurgosChow} and the discussion surrounding it.
\end{definition}

\begin{remark}\label{rem:log to log}
Of  special interest is the case in which $X$ is a smooth quasi-projective complex variety, $Z\subset X$ is a closed subvariety of codimension $d$, and we take $U=X\smallsetminus Z$.  In this case, for any 
\[
g \in E^k_{\log} ( X \smallsetminus Z)
\]
there is a resolution of singularities \eqref{resolution def} such that 
$
g \in E^k_X ( \log Z).
$
When $k < 2 d$,  Remark \ref{rem:log currents} therefore provides us with an injection 
\[
E^k_{\log}( X\smallsetminus  Z ) \map{ g\mapsto [g] }  D_{X}^k.
\]
\end{remark}

In the usual way, the complex structure on $X$ induces bigradings 
\[
E_X^k = \bigoplus_{p+q=k} E_X^{p,q} \quad \mbox{and}\quad  D_X^k = \bigoplus_{p+q=k} D_X^{p,q},
\]
and similarly for the  other spaces of forms and currents appearing above.


\subsection{Specialization to the normal bundle}
\label{ss:hu}


For a closed immersion of schemes  $X_0 \subset X$ one has the normal cone  $C_{X_0/X} \to X_0$.
If $X_0 \subset X$ is a regular immersion, the normal cone agrees with  (the total space of) the normal bundle $N_{X_0/X}\to X_0$.  
These constructions, as well as  the deformation to the normal cone, generalize in an obvious way to a closed immersion of complex analytic spaces;  the necessary technical details are in \cite{AM}.

Now suppose that $X_0 \subset X$ is a closed immersion of complex manifolds.
Denote by $\tilde{X}$ the deformation to   $N_{X_0/X}=C_{X_0/X}$.
By  construction, it comes with morphisms
\begin{equation}\label{deformation diagram}
\xymatrix{
{X} & { \tilde{X} } \ar[r]^\tau \ar[l]_\pi  & {\C}  
}
\end{equation}
such that $\pi$ identifies every fiber $\tilde{X}_t = \tau^{-1}(t)$  with
\[
\tilde{X}_t \iso \begin{cases}
X & \mbox{if }t\neq 0 \\
N_{X_0/X}  & \mbox{if }t=0.
\end{cases}
\]
When $t\neq 0$ we denote by 
\begin{equation}\label{j inc}
j_t : X \iso \tilde{X}_t \hookrightarrow \tilde{X}
\end{equation}
 the  inclusion, and similarly for 
 $
 j_0 : N_{X_0/X} \iso \tilde{X}_0 \hookrightarrow \tilde{X}.
 $

Let $Z \subset X$ be any  equidimensional analytic subset, endowed with its reduced complex analytic structure.  
The \emph{strict transform}
\begin{equation}\label{strict transform}
\tilde{Z} \subset \tilde{X}
\end{equation}
of $Z$   is defined as the deformation to the normal cone $C_{ (Z\times_X X_0) / Z} $.  It  is again reduced (although  $Z\times_X X_0$ and the normal cone  $C_{ (Z\times_X X_0) / Z} $ need not be), and can be  characterized as  the union of all irreducible components of $\pi^{-1}(Z)$  not contained in $N_{X_0/X}$.
 Equivalently, \eqref{strict transform} is the  closure of  $\pi^{-1}(Z)  \smallsetminus \tilde{X}_0$ in $\tilde{X}$.

By an \emph{analytic cycle} on a complex manifold we mean a locally finite formal $\Z$-linear combination of  irreducible analytic subsets, all of the same codimension.  
Being reduced and equidimensional, we may view $\tilde{Z}$  as an analytic cycle on $\tilde{X}$  and form, for every $t\in \C$,  the analytic cycle
\begin{equation}\label{draper intersection}
\tilde{Z}_t = \tilde{Z} \cdot \tilde{X}_t
\end{equation}
on $\tilde{X}$ supported on the fiber $\tilde{X}_t$.
Here the proper intersection of analytic cycles on the right  is taken in the sense of Draper \cite{Draper}.
Of special interest is the analytic cycle \eqref{draper intersection} at $t=0$.

\begin{definition}\label{def:cycle specialization}
The analytic cycle 
\begin{equation}\label{cycle specialization}
\sigma_{X_0/X}(Z) \define \tilde{Z}_0 
\end{equation}
 on $N_{X_0/X}=\tilde{X}_0$ is the \emph{specialization of $Z$ to the normal bundle $N_{X_0/X}$}.    
\end{definition}

Having defined $\tilde{Z}$,  $\tilde{Z}_t$, and $\sigma_{X_0/X}(Z)$ for a reduced analytic subset $Z\subset X$, extend the definitions linearly to all analytic cycles $Z$ on $X$.

\begin{remark}\label{rem:fiber integrals}
If   $t\neq 0$ then $\tilde{Z}_t$  is simply the pushforward of   $Z$ under the inclusion  \eqref{j inc}.
 The cycle  $\tilde{Z}_0$, which may be nonreduced,  is then uniquely determined by the continuity at $t=0$ of the function
\begin{equation}\label{king integrals}
t\mapsto   \delta_{\tilde{Z}_t}( \phi) \define \int_{ \tilde{Z}_t } \phi ,
\end{equation}
for every  $\phi \in {}_cE_{\tilde{X}}^{2 \dim(Z)}$.   Moreover, if we temporarily denote by $I_\phi$ the continuous compactly supported function on $\C$ defined by  \eqref{king integrals}, one has the Fubini-style integration formula
\[
\int_{ \tilde{Z} }  \phi \wedge \tau^* \omega  = \int_\C  I_\phi \wedge \omega 
\]
 for any smooth $2$-form  $\omega$ on $\C$.

The continuity of \eqref{king integrals} and the Fubini formula  \eqref{king integrals} are due to 
 King \cite{King}.    More precisely, Theorem 3.3.2 of \emph{loc.~cit.} constructs  a family of analytic cycles $t \mapsto \tilde{Z}_t$ on $\tilde{X}$ for which these properties hold;
 the equality of this family of cycles  with \eqref{draper intersection} is then a consequence of the results of \S 4.1 of \cite{King}, especially Proposition 4.1.6 and the remarks that follow it.
\end{remark}

  \begin{remark}
  In the case where $X$ and $Z$ are the complex analytic spaces associated to finite type schemes over $\C$, Draper's  analytic intersection \eqref{draper intersection} agrees with the proper intersection of cycles in the algebraic sense of  \cite{Fulton} or \cite{SouleBook}, and 
the cycle \eqref{cycle specialization} agrees with the specialization to the normal cone in the  sense \cite{Fulton}.
  \end{remark}

\begin{definition}\label{def:log expansion}
Fix a  $G \in D_X^k$, and note that every $t\in \C \smallsetminus \{0\}$ determines a  current
\[
 j_{t*} G  \in  D^{k+2}_{\tilde{X}  }.
\]
We say that $G$ \emph{admits a logarithmic expansion along $X_0$} if there is a sequence of  functions
$
 G_0, G_1,G_2,\ldots : \C  \to D^{k+2}_{ \tilde{X} }  
$
 with the following properties.
\begin{enumerate}
\item
For all $t\in \C \smallsetminus \{0\}$ we have 
 \[
 j_{t*} G    =    \sum_{i \ge 0}   G_i  (t)  \cdot   ( \log |t| )^i ,
 \]
and the sum is locally finite:  for every compact subset $K\subset \tilde{X}$  there is an integer $M_K$ such that 
 $
 G_i  (t)(  \phi)=0
 $
 for all  $i > M_K$, all  $t\in \C\smallsetminus \{0\}$, and all 
 $
 \phi \in E_{\tilde{X}}^{ 2n-k } 
 $
  with support contained in $K$.

 \item
 For every $i\ge 0$ and every  $\phi\in {}_cE_{\tilde{X}}^{2n-k}$, the function $ t \mapsto G_i(t)(\phi)$   is  continuous    at $t=0$, and is H\"older continuous at $t=0$ if $i>0$.

\item
Each  $G_i(0)$ lies in the image of 
$
j_{0*} : D^k _{  \tilde{X}_0  }  \to  D^{k+2} _{  \tilde{X}  } .
$
  \end{enumerate}
\end{definition}

The following result  is   slightly weaker  than Theorem 3.2.2 of \cite{Hu};  see  Remark \ref{rem:strong hu} below.   
It provides a general criterion for the existence of logarithmic expansions. 

\begin{theorem}[Hu]\label{thm:hu}
Suppose we have a form 
\[
g \in E_X^k(\log Z )
\]
 in the subspace \eqref{logforms}, for some  equidimensional analytic subset $Z\subset X$ of positive codimension, and some choice of resolution of singularities.
If $g$ is locally integrable on  $X$,  then the associated current    $[g]\in D^k_{X}$   admits a logarithmic expansion along $X_0$.
Moreover, if $X$ is compact there exists a logarithmic expansion with $G_i=0$ for $i\gg 0$.
\end{theorem}

\begin{remark}
Hu works on smooth quasi-projective complex varieties, but the same proof works for complex manifolds.  
The only difference is that in the quasi-projective case one can use the existence of smooth compactifications of $X$ and $\tilde{X}$ to prove the existence of finite (not just locally finite) logarithmic expansions.  See  Remark \ref{rem:strong hu} below.
\end{remark}

\begin{remark}\label{rem:log uniqueness}
Suppose we are given functions $f_0,\ldots, f_m: \C \to \C$ with $f_0$   continuous at $0$, and  $f_1,\ldots, f_m$  H\"older continuous at $0$.
An easy induction on $m$, as in Lemma 3.1.5 of \cite{Hu}, shows that if 
\[
\lim_{t \to 0}    \sum_{i=0}^m f_i (t)   \cdot   ( \log |t| )^i   =0, 
\]
then $f_i(0) =0$ for all $i$.
\end{remark}

The functions $G_i$ in Definition \ref{def:log expansion}, when they exist,  are not uniquely determined.
However, it follows from Remark \ref{rem:log uniqueness}  that the currents
$G_i(0)$ are independent of the choice of logarithmic expansion.  This  allows us to make the following definition.

\begin{definition}\label{def:current specialization}
If $G \in D_{X}^k$ admits a logarithmic expansion along $X_0$,  its   \emph{specialization  to the normal bundle} is the current 
\[
\sigma_{X_0/X}( G)   \in  D^k_{ N_{X_0/X}  } 
\]
on the normal bundle $N_{X_0/X}$ characterized by $j_{0*} \sigma_{X_0/X}( G)   = G_0(0)$.
\end{definition}

As a trivial example, if $g\in E_X^k$ then $\pi^* g$ is a smooth form on $\tilde{X}$, and 
\[
j_{t*}[g]  (\phi) =  \int_{\tilde{X}_t}  \pi^*g \wedge \phi
\]
for all $t \neq 0$.  
By Remark \ref{rem:fiber integrals}, the right hand side is a continuous function of $t\in \C$, 
and we obtain a logarithmic expansion of $[g]$ by setting $G_0(t) =  \pi^*g \wedge \delta_{\tilde{X}_t}$ and $G_i(t)=0$ for $i>0$.  In particular,
\[
\sigma_{X_0/X}(g) = [ j_0^* \pi^*g] = [ \pi_0^* i_0^* g],
\]
where $\pi_0 : N_{X_0/X} \to X_0$ is the bundle map and $i_0 : X_0 \to X$ is the inclusion.

\begin{remark}\label{rem:d specialization}
If $G \in D^k_X$ admits a logarithmic expansion along $X_0$, then so does $\partial G$, and
\[
\partial \sigma_{X_0/X}(G) =  \sigma_{X_0/X}(\partial G).
\]
The same holds with  $\partial$ replaced by $\overline{\partial}$.
 This is a formal consequence of the definitions; see  Theorem 3.1.6 of \cite{Hu}.
\end{remark}

The following proposition connects Definitions \ref{def:cycle specialization} and \ref{def:current specialization}.
The proof is extracted from the second proof of Theorem 3.2.3 in \cite{Hu}.

\begin{proposition}\label{prop:green specialization}
Suppose $X_0 \subset X$ is a closed complex submanifold, $Z$ is a codimension $d$ analytic cycle on $X$, and   $G\in D_X^{d-1,d-1}$ satisfies the Green equation 
\[
dd^c G + \delta_Z = [ \omega ] 
\]
for some  $\omega \in E_X^{d,d}$.
If  $G$  admits a logarithmic expansion along $X_0$, then  its specialization to the normal bundle
satisfies the Green equation
\[
dd^c \sigma_{X_0/X}(G) + \delta_{ \sigma_{X_0/X}(Z) } =  [ \pi_0^* i_0^*\omega  ]  .
\]
Here $\pi_0 :N_{X_0/X} \to X_0$ is the bundle map and $i_0 :X_0 \to X$ is the inclusion.
\end{proposition}

\begin{proof}
When $t\neq 0$, we may push forward the Green equation for $G$  via $j_t : X \to \tilde{X}$. 
This yields the equality
\[
dd^c j_{t*} G  + \delta_{\tilde{Z}_t} = \pi^*\omega \wedge \delta_{ \tilde{X}_t } 
\]
of currents on  $\tilde{X}$.  Replacing $j_{t*}G$  by a logarithmic expansion  results in 
\[
\left( dd^c G_0(t)  + \delta_{\tilde{Z}_t}  -   \pi^*\omega \wedge \delta_{ \tilde{X}_ t}  \right)  + 
 \sum_{i >  0}  dd^c G_i(t) \cdot   ( \log|t| )^i  = 0,
\]
and it follows from Remarks \ref{rem:fiber integrals} and \ref{rem:log uniqueness}     that 
\[
 dd^c G_0(0)  + \delta_{\tilde{Z}_0}  -   \pi^*\omega \wedge \delta_{ \tilde{X}_ 0} =0  .
\]
The claim now follows using
$
 \pi^*\omega \wedge \delta_{ \tilde{X}_ 0} = j_{0*} [ j_0^* \pi^* \omega]   =  j_{0*} [ \pi_0^* i_0^*\omega ] .
$
\end{proof}

\begin{remark}\label{rem:strong hu}
In Hu's version of Theorem \ref{thm:hu} it is assumed that  $X$  is a  quasi-projective variety, and that
$
g \in E^k_{\log}(X\smallsetminus Z) .
$
These extra assumptions are not used in the proof in any essential way. However, the first  guarantees the existence of a smooth compactification of $X$. Using this,  Hu proves a stronger result than what we have stated. 

  After choosing a smooth compactification  $X\subset X^*$, Hu constructs  a smooth compactification $\tilde{X} \subset \tilde{X}^*$ of the deformation to the normal bundle,  a diagram
 \[
\xymatrix{
{ X^* }   & { \tilde{X}^*} \ar[r]^\tau \ar[l]_\pi  & {\C}  
}
\]
extending \eqref{deformation diagram}, and a  finite expansion of currents
 \begin{equation}\label{strong expansion}
 \pi^* g \wedge \delta_{\tilde{X}^*}  = \sum_{i=0}^M  G_i(t) \cdot (\log |t| )^i 
 \end{equation}
in the space  $D^{k+2}_{\tilde{X}^* / \partial \tilde{X}^*}$.   
The inclusion 
$
{}_cE_{\tilde{X}}^\bullet  \to {}_cE_{\tilde{X}^*}^\bullet( \mynull \partial\tilde{X}^* )
$
induces a surjection
\[
D^\bullet_{\tilde{X}^* / \partial \tilde{X}^*} \to D^\bullet_{ \tilde{X}},
\]
and applying this map to  both sides of \eqref{strong expansion} yields a  finite logarithmic expansion of $[g]$.
The refined logarithmic expansion \eqref{strong expansion} contains more information than a logarithmic expansion in our sense.  Using it, Hu is able to construct a smooth compactification 
$N_{X_0/X} \subset N_{X_0/X}^*$ of the normal bundle, and a distinguished lift of $\sigma_{X_0/X}(g)$ under the surjection
\[
D^\bullet_{    N_{X_0/X}^* / \partial   N_{X_0/X}^*  } \to D^\bullet_{ N_{X_0/X}  }.
\]
  Although we will not need such a lift,  the benefits of having one are explained in Remark \ref{rem:strong hu pullback}.
\end{remark}


\subsection{Arithmetic Chow groups}


We  will use the arithmetic Chow groups defined in \S 3 of \cite{Gillet-Soule}, but only in the simple case of varieties over a field  $F$   with a chosen real embedding  $\sigma : F \to \R$.  If we let $c\in \Aut(\C/\R)$ be complex conjugation,
the triple $(F, \{ \sigma\} ,c)$  is an arithmetic ring,  and any
smooth quasi-projective variety $X$ over $F$ is an arithmetic variety over $(F, \{ \sigma\} ,c)$ in the sense of  \emph{loc.~cit.}.

Let $X_\R = X \otimes_{F,\sigma}\R$ be the base change of $X$ to $\R$,  and regard $X(\C) = X_\R(\C)$  as a complex manifold.
Define a real vector space
\[
E_{X}^{d,d} = 
  \{ \omega \in E_{X(\C)}^{d,d}  :  \omega \mbox{ is real and } c^* \omega = (-1)^{d} \omega  \} ,
\]
where now $c : X(\C) \to X(\C)$ is complex conjugation,  and similarly
\[
D_X^{d,d} = 
  \{ G \in D_{X(\C)}^{d,d}  :  G \mbox{ is real and } c^* G = (-1)^{d} G  \} .
\]

A  \emph{codimension $d$ arithmetic cycle} on $X$ is a pair $(Z,G)$ in which $Z$ is a codimension $d$ cycle in the usual sense, and 
\[
G \in \widetilde{D} _X^{d-1,d-1} \define \frac{D_X^{d-1,d-1}}{ \mathrm{Im}( \partial) + \mathrm{Im}(\overline{\partial})  },
\]
satisfies the Green equation
$
dd^c G + \delta_{Z(\C)} = [ \omega ] 
$
for some $\omega \in E_X^{d,d}$.
 Denote by $\widehat{Z}^d(X)$  the abelian group of all  such pairs. The arithmetic Chow group is the quotient
\[
\widehat{\mathrm{CH}}^d(X)  = \widehat{Z}^d(X) / (\mbox{rational equivalence}) .
\]

Now assume that $X$ is projective, and that $X_0 \subset X$ is a smooth closed subvariety.
 Let $(Z,G)$ be any codimension $d$ arithmetic cycle, and set $U=X\smallsetminus Z$.
 Recalling Definition \ref{def:burgos forms}, define a real vector space
\[
E_{\log}^{d,d} (U)  =   \{ g \in  E_{\log}^{d,d}  (U(\C) ) :  g \mbox{ is real and } F_\infty^* g = (-1)^{d} g  \} .
\]
By Remark \ref{rem:log to log} there is a canonical map 
\[
E_{\log}^{d-1,d-1} (U) \map{ g \mapsto [g] } D_X^{d-1,d-1},
\]
and Theorem 4.4 of \cite{BurgosGreen} implies the existence of a unique lift of $G$ to 
\[
g \in \widetilde{E}_{\log}^{d-1,d-1} ( U)  \define
 \frac{E_{\log}^{d-1,d-1} ( U) }{\mathrm{Im}( \partial) + \mathrm{Im}(\overline{\partial}) } .
\]
 Theorem \ref{thm:hu} therefore  implies that the current $G=[g]$ admits a logarithmic expansion along $X_0$.
 Combining this with Remark \ref{rem:d specialization} and Proposition \ref{prop:green specialization}, we obtain an 
arithmetic cycle 
\[
 (\sigma_{X_0/X}(Z) , \sigma_{X_0/X}( G )) \in \widehat{Z}^d(N_{X_0/X})  .
\] 
 This defines a homomorphism
\begin{equation}\label{arithmetic specialization}
\widehat{Z}^d(X) \to \widehat{Z}^d(N_{X_0/X}).
\end{equation}

The following is slightly weaker than what is proved in \S 4.1 of \cite{Hu}; see Remark \ref{rem:strong hu pullback} below.

\begin{theorem}[Hu]\label{thm:arithmetic specialization}
Still assuming that $X$ is projective, the homomorphism  \eqref{arithmetic specialization} descends to 
\[
\widehat{\mathrm{CH}}^d(X) \to \widehat{\mathrm{CH}}^d(N_{X_0/X}) ,
\]
and this map agrees with the composition
\[
\widehat{\mathrm{CH}}^d(X)  
 \map{i_0^* }   \widehat{\mathrm{CH}}^d(X_0) \map{\pi_0^*}  \widehat{\mathrm{CH}}^d(N_{X_0/X}).
\]
Here $i_0 :X_0 \to X$ is the inclusion,  $\pi_0 :N_{X_0/X} \to X_0$ is the bundle map, and $i_0^*$ and $\pi_0^*$ are the induced pullbacks on  arithmetic Chow groups.
\end{theorem}

 \begin{remark}\label{rem:strong hu pullback}
 Assuming only that $X$ is quasi-projective,  there are canonical maps 
 \[
  \widehat{Z}^d(X,\mathcal{D}_{\log}) \to  \widehat{Z}^d(X) 
\quad\mbox{and}\quad
  \widehat{\mathrm{CH}}^d(X,\mathcal{D}_{\log}) \to  \widehat{\mathrm{CH}}^d(X) ,
 \]
 where the domains are the $\mathcal{D}_{\log}$ arithmetic cycles and Chow groups  of \cite{BKK}.  
 These agree with those of \cite{BurgosChow}, and both maps  are isomorphisms if $X$ is projective.
 Hu proves the existence of a distinguished lift of \eqref{arithmetic specialization}  to 
\begin{equation}\label{strong arithmetic cycle specialization}
  \widehat{Z}^d(X,\mathcal{D}_{\log}) \to  \widehat{Z}^d(N_{X_0/X} , \mathcal{D}_{\log})  ,
  \end{equation}
 which then descends to a map on $\mathcal{D}_{\log}$ arithmetic Chow groups.
This descent agrees with the composition 
\[
\widehat{\mathrm{CH}}^d(X, \mathcal{D}_{\log})  
 \map{i_0^*}   \widehat{\mathrm{CH}}^d(X_0, \mathcal{D}_{\log}) \map{\pi_0^*}  \widehat{\mathrm{CH}}^d(N_{X_0/X}, \mathcal{D}_{\log} ).
\]
Even when $X$ is projective, this is stronger than Theorem \ref{thm:arithmetic specialization} (because $N_{X_0/X}$ is not projective).
The construction of the lift \eqref{strong arithmetic cycle specialization} is subtle, but the key ingredient is the lift of $\sigma_{X_0/X}(g)$ mentioned at the end of Remark \ref{rem:strong hu}. 
 \end{remark}

\begin{proposition}\label{prop:injective normal}
The pullback $\pi_0^*$ in Theorem \ref{thm:arithmetic specialization} 
is injective.
\end{proposition}
 
\begin{proof}
A similar statement  is found in \cite{BurgosChow}, but for the $\mathcal{D}_{\log}$ arithmetic Chow groups of Remark \ref{rem:strong hu pullback}.
The proof for   Gillet-Soul\'e arithmetic Chow groups is essentially the same:
 By  Theorem 3.3.5 of \cite{Gillet-Soule}  there is  commutative diagram with exact rows
\[
\xymatrix{
{ \mathrm{CH}^{d,d-1}(X_0) } \ar[r] \ar[d] & {   \widetilde{E}^{d-1,d-1}_{X_0}  }   \ar[r] \ar[d]& 
 {   \widehat{\mathrm{CH}}^d(X_0)  } \ar[r] \ar[d]^{\pi_0^*} &   {  \mathrm{CH}^d(X_0)  } \ar[d]  \\
 { \mathrm{CH}^{d,d-1}(N_{X_0/X}) } \ar[r]  & {   \widetilde{E}^{d-1,d-1}_{N_{X_0/X}}  }   \ar[r] & 
 {   \widehat{\mathrm{CH}}^d(N_{X_0/X})  } \ar[r]  &   {  \mathrm{CH}^d(N_{X_0/X})  . }  
}
\]
The first and last vertical arrows are isomorphisms by Theorem 8.3 of \cite{GilletRR}.
 The second vertical arrow is injective, and hence the third is as well.
\end{proof}


\section{Green currents of Garcia-Sankaran}
\label{s:green}


Given a closed immersion of complex manifolds $X_0\subset X$, 
the constructions of Garcia-Sankaran \cite{GS},   Bismut \cite{Bis}, and Bismut-Gillet-Soul\'e \cite{BiGSa}, provide a systematic way to produce Green currents for certain cycles on $X$.
Theorem \ref{thm:hu} can be applied to these currents to prove the existence of logarithmic expansions, but this abstract existence theorem is not sharp enough for our purposes.

The goal of this section is to  construct explicit logarithmic expansions for these currents, and so effectively compute their specializations  to the normal bundle $N_{X_0/X}$.


\subsection{Construction of Green forms}
\label{ss:GS basics}


Let $X$ be a complex manifold, and let $L$ be a holomorphic  line  bundle on $X$. We use the same symbol for both the total space $L\to X$, viewed as a complex manifold fibered over $X$, and for its sheaf of  holomorphic sections.

Let $h(-,-)$ be a hermitian metric on $L$.  If $s$ is any local holomorphic  section of $L$,  abbreviate $h(s) =h(s,s)$.  The \emph{Chern form} of $L$ is the $(1,1)$ form defined locally by 
\[
\chern(L) =   \frac{1}{2\pi i}  \partial \overline{\partial} \log h(s) .
\]
We denote again by $h$ the induced metric on the dual bundle $L^\vee$.  

Fix an integer $1\le d \le \dim(X)$ and a  tuple $s=(s_1,\ldots, s_d)$ with  $s_i\in H^0(X,L^\vee)$, and abbreviate
\[
h(s) = h(s_1) + \cdots + h(s_d).
\]
Denote by 
$
Z(s)   \subset X
$
the (possibly nonreduced) analytic subspace defined by   $s_1=\cdots=s_d=0$.

\begin{definition}
Fix a point $x\in Z(s)$,  trivialize $L$ in a neighborhood of  $x$, and use this to view $s_{1,x},\ldots, s_{d,x} \in \co_{X,x}$ as germs of holomorphic functions at $x$.   We say that $s$ is: 
\begin{itemize}
 \item
 \emph{regular at $x$} if $s_{1,x},\ldots, s_{d,x}\in \co_{X,x}$ is a regular sequence in the sense of commutative algebra;
 \item
 \emph{smooth at $x$} if $s_{1,x},\ldots, s_{d,x}$ are linearly independent in $\mathfrak{m}_{X,x}/ \mathfrak{m}_{X,x}^2$,
 where $\mathfrak{m}_{X,x} \subset \co_{X,x}$ is the maximal ideal.
 \end{itemize}
The tuple $s=(s_1,\ldots, s_d)$ is \emph{regular} or \emph{smooth}  if it has this property at every point   of $Z(s)$.
\end{definition}

\begin{remark}\label{rem:CM2}
Regularity of $s$ at $x$ is equivalent to all irreducible components of $Z(s)$ passing through $x$ having codimension   $d$ in $X$,  and both  are equivalent to $\co_{Z(s),x}$ being Cohen-Macaulay of dimension $\dim(X)-d$.
\end{remark}

\begin{remark}\label{rem:CM3}
Smoothness of $s$  at $x$ is equivalent to  $Z(s)$  being nonsingular (that is, a complex manifold) of  codimension  $d$  in some open neighborhood of $x$, as both  are equivalent to $\co_{Z(s),x}$ being regular of dimension $\dim(X)-d$.
\end{remark}

\begin{remark}\label{rem:smooth intersection}
If $s=(s_1,\ldots, s_d)$ is smooth, we have the equality of cycles
\[
Z(s) = Z(s_1)\cdots Z(s_d)
\]
on $X$, where the intersection on the right is  the proper analytic intersection of Draper \cite{Draper}.
In other words, in the smooth case the intersection
$
\mathrm{div}(s_1)\cdots \mathrm{div}(s_d)
$
 in Draper's sense is (of course) simply the reduced analytic subspace defined by $s_1= \cdots =  s_d=0$.
\end{remark}

\begin{remark}\label{rem:CM1}
If $s=(s_1,\ldots, s_d)$ is regular or smooth at a point $x$, the same is true of all tuples obtained by reordering the components of $s$, and of all tuples $(s_1,\ldots, s_r)$ with  $1 \le r \le d$.
\end{remark}

The claims of Remarks  \ref{rem:CM2},  \ref{rem:CM3}, and \ref{rem:CM1}, all follow from basic properties of regular sequences  and complex analytic spaces, as found in \cite{Matsumura} and \cite{Fischer}.  
Similarly, it is elementary to check that 
regularity of $s$ is equivalent to the corresponding  morphism of  vector bundles 
$
s:L^{\oplus d} \to \co_X
$
 being regular in the sense of \S 2.1.1 of \cite{GS}.  Therefore, if  $s$ is regular, the constructions  (2.5) and (2.12) of \emph{loc.~cit.} define forms
\[
\varphi^\circ(s) \in E^\bullet_X \quad \mbox{and}\quad   \nu^\circ(s) \in E^\bullet_X.
\]
  Both  have trivial components in odd degree, and their components in even degree $2p$ have type $(p,p)$.   Abbreviate
 \[
 \omega^\circ(s) =  (-2\pi i)^{-d} \cdot  \varphi^\circ(s)_{[2d]} \in E^{d,d}_X.
 \]

 We will not recall the detailed construction of the forms above, as we only need the degree $2d$ component of $\varphi^\circ(s)$ and the degree $2d-2$ component of $\nu^\circ(s)$.
 Explicit formulas for these can be found in  \cite{GS} and \cite{Garcia}.  
If $d=1$  then 
\begin{align}
\varphi^\circ(s)_{[2]} 
& =  2 \pi i  e^{-2\pi h(s)}  \left(     \chern(L)   -i     \frac{ \partial h(s) \wedge \overline{\partial} h(s) }{ h(s) }    \right)  \label{phi2}
\end{align}
and 
\begin{equation}\label{nu2}
\nu^\circ(s)_{[0]}  =    e^{-2\pi h(s)}.
\end{equation}
  If $d>1$ then
\begin{equation}\label{phid}
\varphi^\circ(s_1,\ldots,s_d)_{[2d]} = \varphi^\circ(s_1)_{[2]}  \wedge \cdots \wedge  \varphi^\circ(s_d)_{[2]} 
\end{equation}
and  
\begin{equation}\label{nud}
\nu^\circ(s_1,\ldots,s_d)_{[2d-2]} 
= \sum_{j=1}^d \nu^\circ(s_i)_{[0]} \wedge  \varphi^\circ(s_1,\ldots,\widehat{s}_j ,\ldots,s_d)_{[2d-2]}. 
\end{equation}
Strictly speaking, the above formulas are given in  \cite{GS} and \cite{Garcia} only for specific hermitian line bundles on hermitian symmetric domains associated to orthogonal and unitary groups, but the derivations of these formulas hold verbatim in our more general setting.

As explained in \cite{GS}, results of Bismut \cite{Bis} and Bismut-Gillet-Soul\'e \cite{BiGSa} can be used to produce Green currents for the cycles $Z(s) \subset X$ defined above.  
We need a slight strengthening of those results.

\begin{proposition}\label{prop:general green}
If $s$ is regular, the integral
\begin{equation}\label{green def}
\mathfrak{g}^\circ( s) = \left(\frac{-1}{2\pi i} \right)^{d-1} \int_1^\infty  \nu^\circ \big( \sqrt{u} \cdot s \big)_{[2d-2]}  \, \frac{du}{u}
\end{equation}
defines a smooth form on $X\smallsetminus Z(s)$ with
\begin{equation}\label{green singularities}
\mathfrak{g}^\circ(s) = 
 \frac{  a(s) }{ h(s)^{d-1}  }+ b(s) \cdot \log( h(s)) 
\end{equation}
for some $a(s), b(s)  \in E_X^{d-1,d-1}$.
If $s$ is smooth,   then
\begin{equation}\label{green is log}
\mathfrak{g}^\circ(s) \in E_X^{d-1,d-1}( \log Z(s) ) 
\end{equation}
with respect to the resolution of singularities of $(X, Z(s))$ obtained by  blowing up  along $Z(s) \subset X$, and 
the associated current  (Remark \ref{rem:log currents}) satisfies the Green equation
\[
dd^c [\mathfrak{g}^\circ( s)   ] + \delta_{Z(s)} =  [   \omega^\circ(s) ] .
\]
\end{proposition}

\begin{proof}
First assume  $d=1$, so that $s$ is a nonzero section of $L^\vee$.
Plugging  \eqref{nu2} into \eqref{green def} yields
\begin{equation}\label{simplest green}
\mathfrak{g}^\circ( s) 
 =  \int_1^\infty  e^{-2\pi u h(s)}  \, \frac{du}{u}    = E_1( 2\pi h(s) ), 
\end{equation}
where
\begin{equation}\label{Ei}
E_1(x) = \int_1^\infty e^{-xu}  \, \frac{du}{u}
= -\log( x) -\gamma - \sum_{k=1}^\infty \frac{(-x)^k}{k\cdot k!}.
\end{equation}
This gives a more precise version of \eqref{green singularities}, which will be essential later.
 
 Now suppose $d>1$.
 For each $1\le j\le d$  abbreviate
\begin{equation}\label{basic eta}
\eta( s_j )   =  -i   \cdot  \frac{ \partial h(s_j) \wedge \overline{\partial} h(s_j) }{ h( s_j ) }    
\in E_X^{1,1},
\end{equation}
so that  \eqref{phi2} becomes
\[
\varphi^\circ(s_j)_{[2]}  =  2 \pi i  e^{-2\pi h(s_j)}  \left(     \chern(L)    + \eta( s_j ) \right), 
\]
and  \eqref{phid} and \eqref{nud} imply
\begin{align*}
\nu^\circ(s)_{[2d-2]} 
&= 
 (2 \pi i)^{d-1} e^{-2\pi h(s)}  
 \sum_{j=1}^d  \underbrace{ (      \chern(L) +  \eta(s_1 ) ) \wedge \cdots \wedge     (   \chern(L)   +  \eta(s_d) )}_{\mathrm{omit\ } j^\mathrm{th} \mathrm{\ factor}} .
\end{align*}
Expanding out the wedge products in each term, we  rewrite this as
\begin{equation}\label{homogeneous nu}
\nu^\circ(s)_{[2d-2]}  = e^{-2\pi h(s)}   \sum_{k=0}^{d-1}  \eta_k( s )  ,
\end{equation}
in which  each 
$
\eta_k(s) \in E_X^{d-1,d-1}
$
 is (up to multiplication by a constant) the wedge product of $\chern(L)^{d-k-1}$ with a sum of $k$-fold wedges of $\eta(s_1),\ldots, \eta(s_d)$.
For  any $t\in \C$ we have $\eta(t s_j ) = |t|^2 \eta(s_j)$, and hence 
\begin{equation}\label{eta homo}
\eta_k (t s) = |t|^{2k} \eta_k( s ).
\end{equation}

Plugging \eqref{homogeneous nu} into \eqref{green def} results in
\begin{align*}
( -2\pi i )^{d-1}    \mathfrak{g}^\circ( s )  
   & =   
   \sum_{k=0}^{d-1}   \eta_k(   s )  \int_1^\infty     u^k     e^{-2\pi  u  h( s )}       \, \frac{du}{u} .
  \end{align*}
If $k>0$, a calculus exercise  shows that  
\begin{equation}\label{exp integral}
   \int_1^\infty    u^k   e^{-  u  x }       \,  \frac{du}{u}    = 
 \frac{    e^{-x}  \cdot (k-1)!   }{x^k}   \cdot 
   \sum_{i=0}^{k-1}  \frac{    x ^i } {  i!}  .
\end{equation}
Rewriting this as $e^{-x} x^{-k}P_k(x)$  for some polynomial $P_k(x)$, we obtain
\begin{equation}\label{green sing}
( -2\pi i )^{d-1}  \mathfrak{g}^\circ( s )  
  = 
\eta_0(s)  \cdot  E_1( 2\pi h(s) )   +   e^{- 2\pi h(s) }  \sum_{k=1}^{d-1}   \frac{  \eta_k(s)  }{ h(s)^k }  \cdot P_k( h(s) )    .
  \end{equation}
The equality \eqref{green singularities}  follows immediately by putting all terms in the sum over the common denominator $h(s)^{d-1}$, and using \eqref{Ei}.

Assuming now that $s$ is smooth, we establish \eqref{green is log}.
Near any point  $x \in Z(s)$ we may choose an open neighborhood $U$ over which the line bundle $L^\vee$ admits a trivializing section $\sigma$.  Each component of $s=(s_1,\ldots, s_d)$ then has the form 
\[
s_i = z_i \cdot \sigma
\]
for some  holomorphic function $z_i$, and $h(s_i)  = f \cdot  | z_i|^2$ where $f = \| \sigma\|^2$ is a smooth function on $U$ valued in the positive real numbers.  

The smoothness of the section $s$ implies that $z_1,\ldots, z_d$ can be completed to a system of local coordinates $z_1,\ldots, z_{\dim(X)}$ on (a possibly smaller)  $U$.
In these  coordinates the cycle $Z(s)\cap U$ is defined by $z_1=\cdots=z_d =0$.  Moreover, 
\begin{equation}\label{h coords}
h(s)|_U=  f  \cdot( | z_1|^2+\cdots + |z_d|^2 ), 
\end{equation}
and the $(1,1)$-form \eqref{basic eta}  can be expressed as 
\begin{align}
\eta( s_j )|_U   & = |z_j |^2   \wedge\mathrm{smooth}   +z_j \, d \overline{z}_j  \wedge\mathrm{smooth}  \label{eta coords}  \\
& \quad + \overline{z}_j \, d z_j  \wedge\mathrm{smooth} +  d z_j  \wedge d \overline{z}_j  \wedge\mathrm{smooth} , \nonumber
\end{align}
where each ``smooth" is some smooth form of the appropriate bi-degree.

Now consider the pullback of \eqref{green sing} to the  blow-up of $U$ along $Z(s) \cap U$.
This blow-up is  isomorphic to  the  submanifold 
\[
V  \subset U \times \mathbb{P}^{d-1}
\]
defined by $\dot{w}_j z_i = \dot{w}_i z_j$ for all $ 1 \le i,j \le d$, where $\dot{w}_1,\ldots, \dot{w}_d$ are the homogeneous coordinates on $\mathbb{P}^{d-1}$.  It is covered by  open subsets $V_1 ,\ldots, V_d$, with   $V_i  \subset V$  defined by the condition $\dot{w}_i \neq 0$.

For ease of notation, let's work on the open subset $V_1 \subset V$ where $\dot{w}_1\neq 0$, and denote by $\pi_1 : V_1 \to U$ the projection.   On  $V_1$ we have coordinates 
\[
z_1,w_2,\ldots, w_d , z_{d+1},\ldots, z_{\dim(X)},
\]
and the functions $z_2,\ldots, z_d$ are expressed in these coordinates as
\begin{equation}\label{blow-up swap}
z_j = z_1 w_j.
\end{equation}
In particular,  the preimage of $Z(s) \cap U$ under $\pi_1 : V_1 \to U$ is defined by the single equation $z_1=0$.

Plugging \eqref{blow-up swap} into \eqref{h coords} and \eqref{eta coords}, we find that 
\[
 \pi_1^* h(s)   =  \phi \cdot  |z_1|^2  
\]
for $\phi$ a smooth function on $V_1$ valued in the positive real numbers, and
\begin{align*}
\pi_1^* \eta( s_j )   & = |z_1 |^2   \wedge\mathrm{smooth}   +z_1 \, d \overline{z}_1  \wedge\mathrm{smooth}    \\
& \quad + \overline{z}_1 \, d z_1  \wedge\mathrm{smooth} +  d z_1  \wedge d \overline{z}_1  \wedge\mathrm{smooth} .
\end{align*}
Recalling the discussion surrounding \eqref{divisor logforms}, it follows that  the pullback of $\eta(s_j)  h(s)^{-1} $    has logarithmic growth along  $\pi_1^* Z(s) \subset V_1$, for every $1\le j \le d$.

The pullback  to $V_1$ of each   $\eta_k(s)  h(s)^{-k}$ appearing in  \eqref{green sing}   has logarithmic growth along 
$\pi_1^* Z(s)$, because each is a sum of  wedge products of smooth forms and the $\eta(s_j)   h(s)^{-1}$ just analyzed.
Similarly, \eqref{Ei} implies that singularities  of   $\eta_0(s) E_1(2\pi h(s))$ are the same as those of 
$\log h(s)$, and so the pullback of this form also has logarithmic growth along $\pi_1^* Z(s)$.

Of course the same analysis applies on each of the open subsets $V_i \subset V$, proving that the pullback of \eqref{green sing} via the  blowup morphism $V\to U$ has logarithmic singularities along the preimage of $Z(s)\cap U$.   This completes the proof of \eqref{green is log}.

For the  Green equation, see Proposition 2.2 of \cite{GS}.
\end{proof}


\subsection{The star product formula}


Suppose $G_1$ and $G_2$ are currents on $X$ satisfying the Green equations
\[
dd^c G_i + \delta_{Z_i} = [ \omega_i ] 
\]
for analytic cycles $Z_1$ and $Z_2$ of codimensions $d_1$ and $d_2$  intersecting properly.
Suppose also that  $G_2=[g_2]$ is the current defined by a smooth form $g_2$ on $X \smallsetminus Z_2$,  locally integrable on $X$.  
The form $g_2$ is then uniquely determined by $G_2$, and  we define
\[
G_1 \star G_2=  \delta_{Z_1} \wedge G_2 + G_1 \wedge \omega_2 \in D_X^{d_1+d_2-1,d_1+d_2-1},
\]
\emph{provided that} the integral
\[
( \delta_{Z_1} \wedge G_2)(\phi)  = \int_{Z_1}  g_2\wedge \phi 
\]
converges for all $\phi \in {}_c E_X^\bullet$ of the appropriate degree.

\begin{remark}
Note that we understand the star product to be a current on $X$, not an element of the space of currents modulo currents of the form $\partial a+\overline{\partial} b$.  Because of this, the star product is neither commutative nor associative, and in fact it may be that $G_1\star G_2$ is defined while $G_2\star G_1$ is not.
\end{remark}

\begin{remark}\label{rem:star switch}
Keeping the previous remark in mind, we caution the reader that we are using the convention for star products opposite to \cite{SouleBook} and \cite{GS}: 
our $G_1 \star G_2$ is their $G_2 \star G_1$.  
\end{remark}

\begin{remark}
The expression  $G_1\star ( G_2 \star G_3)$ does not make sense, as $G_2\star G_3$ is not represented by a locally integrable form (even if $G_2$ and $G_3$ are).  
We therefore understand 
\begin{align*}
G_1\star G_2 \star G_3 & = (G_1\star G_2) \star G_3 \\
G_1\star G_2 \star G_3 \star G_4 &= ((G_1\star G_2) \star G_3) \star G_4 \\
& \vdots 
\end{align*}
provided that each star product on the right  is defined.
\end{remark}

Fix a smooth tuple $s=(s_1,\ldots, s_d)$ with $s_i \in H^0(X,L^\vee)$.
If we write   $d=k+\ell$ with $k,\ell >0$,  and  express $s=(p,q)$ as the concatenation of the smooth tuples
\[
p=(s_1,\ldots, s_k ) \quad \mbox{and} \quad q = ( s_{k+1}, \ldots, s_d ) ,
\]
then  $Z(s) = Z(p) \times_X Z(q)$ as analytic spaces.

\begin{lemma}\label{lem:stars make sense}
If $G(p) \in D_X^{  k-1,k-1 }$ is any Green current   for $Z(p)$,  the star product $G(p) \star \mathfrak{g}^\circ(q)$ is defined.
\end{lemma}

\begin{proof}
The pullback of $\mathfrak{g}^\circ(q)$ to  $Z(p)$ is  the  form $\mathfrak{g}^\circ(q|_{Z(p)})$  obtained by applying the construction of Proposition \ref{prop:general green} to the smooth $\ell$-tuple 
\[
q|_{ Z(p) } = ( s_{k+1}|_{Z(p) }, \ldots, s_d |_{ Z(p) } )
\]
 of sections of $L^\vee|_{Z(p)}$ on the complex manifold $Z(p)$.  In particular, this pullback is locally integrable on $Z(p)$.
\end{proof}

In particular, the lemma implies that the star product in the following theorem is defined.

\begin{theorem}[Garcia-Sankaran]\label{thm:star}
We have the equality of currents
\[
\mathfrak{g}^\circ(s) 
=   \mathfrak{g}^\circ(p)  \star  \mathfrak{g}^\circ(q) 
 - \partial [ A(p ; q) ] - \overline{\partial} [ B(p ; q) ]
\]
on $X$, where 
\begin{align*}
A(p;q)
&=  \left( \frac{-1}{2\pi i } \right)^{d-1}  \int_{1 <  v <  u <\infty }
\overline{\partial}  \left(   \nu^\circ (\sqrt{u} p )_{[2k-2]}  \right) \wedge \nu^\circ( \sqrt{v}  q)_{[2 \ell -2]}  
\, \frac{du }{u  }\frac{ dv}{ v } \\
B(p;q)
&= \left(  \frac{-1}{2\pi i } \right)^{d-1}  \int_{1 <  v <  u <\infty  }
   \nu^\circ (\sqrt{u} p )_{[2 k -2]}  \wedge  \partial \left( \nu^\circ( \sqrt{v}  q)_{[2 \ell -2]}  \right)
    \,\frac{du }{u  }\frac{ dv}{ v } 
   \end{align*}
 are smooth forms on $X \smallsetminus (Z(p) \cup Z(q) )$, locally integrable on $X$.
Moreover, there is a resolution of singularities of 
\[
Z(p) \cup Z(q) \subset X
\]
 for which
 \begin{equation}\label{messy resolve}
A(p;q) , B(p;q) \in E_X^\bullet ( \log Z(p) \cup Z(q) ) .
\end{equation}
\end{theorem}

\begin{proof}
Except for the final claim, this is Theorem  2.16 of \cite{GS}, modified as per Remark \ref{rem:star switch}.
For those authors $X$ is a particular hermitian symmetric domain, but the same argument works on any complex manifold.

It remains to prove \eqref{messy resolve}.  Construct resolutions of singularities 
\[
(X',D') \to (Y',E') \map{r'}  (X,Z(p))
\]
and
\[
(X'', D'') \to (Y'',E'') \map{r''} (X,Z(q)) 
\]
by  taking $Y'$ and $Y''$ to be the blowups of $X$ along $Z(p)$ and $Z(q)$, respectively.
Then let $X'$ and $X''$ be the blowups of $Y'$ and $Y''$ along the  preimages of $Z(s)=Z(p) \cap Z(q)$ under $r'$ and $r''$.

Now fix a resolution of singularities   $(X^\dagger, D^\dagger)$  of the analytic subspace 
\[
D' \times_X D'' \subset X'\times_X X''.
\]
The natural map $X^\dagger \to X$ is then  a resolution of singularities 
 \[
 (X^\dagger ,  D^\dagger) \to (X,Z(p) \cup Z(q)) ,
 \]
 and we claim that  \eqref{messy resolve} is satisfied for any such  choice.
 The proof will require the following elementary lemma.

\begin{lemma}\label{lem:easy resolve}
The pullback  of  $h(p)/h(s)$  to 
\[
X' \smallsetminus D'  \iso X \smallsetminus Z(p)
\]
 extends  smoothly to $X'$, and the pullback of $h(q)/h(s)$
 \[
 X'' \smallsetminus D'' \iso X \smallsetminus Z(q) 
 \]
  extends  smoothly to $X''$.
In particular,  both pullbacks  to 
\[
X^\dagger \smallsetminus D^\dagger \iso X \smallsetminus (Z(p) \cup Z(q) )
\]
  extend  smoothly to $X^\dagger$.
\end{lemma}

\begin{proof}
The function $h(p)/h(s)$ is smooth on the open complement of $Z(s) \subset X$, so it suffices to analyze its singularities on an open neighborhood  of a point of $Z(s)$.  

 As in the proof of Proposition \ref{prop:general green}, we use the smooth tuple $s=(p,q)$ to choose local  coordinates $z_1,\ldots, z_{\dim(X)}$    in such a way that 
 \begin{align*}
 h(p) & = f \cdot ( | z_1|^2+ \cdots + | z_k|^2 )  \\
  h(q) & = f \cdot ( | z_{k+1} |^2+ \cdots + | z_d |^2 ),
 \end{align*}
where $f$ is a smooth function valued in the positive real numbers. In particular
\begin{equation}\label{blow fraction}
\frac{h(p)}{h(s)}  = \frac{ | z_1|^2+ \cdots + | z_k|^2}{ | z_1|^2+ \cdots + | z_k|^2 + | z_{k+1}|^2 + \cdots + |z_d|^2 }.
\end{equation}

Using the explicit description of  blow-ups in coordinates,  as in the proof of Proposition \ref{prop:general green}, it is easy to see that if one first blows-up along the  cycle $Z(p)$ defined by $z_1=\cdots = z_k=0$, and then blows-up along the preimage of the cycle $Z(s)$ defined by $z_1=\cdots = z_d=0$, the pullback of \eqref{blow fraction} to this double blow-up has no singularities.
This proves the first claim of the lemma.  

The proof of the second is  identical, and the third claim follows from the first two, as the map $X^\dagger \to X$ factors through both $X'$ and $X''$.
\end{proof}

Continuing with the proof of Theorem \ref{thm:star},  abbreviate $\hbar = 2\pi h$, and  expand 
\[
\nu^\circ( p)_{[2k-2]} 
= 
 e^{- \hbar(p)}   \sum_{a=0}^{k-1}    \eta_a(p) 
 \quad \mbox{and} \quad 
 \nu^\circ( q)_{[2\ell-2]} 
= 
 e^{- \hbar(q)}   \sum_{b=0}^{\ell -1}    \eta_b(q) 
\]
as in \eqref{homogeneous nu}.
Plugging this expansion into the definitions of $A(p;q)$ and $B(p;q)$, and noting that $\eta_0(p)$ and $\eta_0(q)$ are closed, we find that 
\begin{align}
(- 2\pi i )^{d-1}  A(p;q)   
 & = 
  \sum_{\substack{   0 < a < k \\ 0 \le b <  \ell   }  }
F_{a,b} ( \hbar(p) , \hbar(q) )  \cdot   \overline{\partial}  \eta_a(p)   \wedge   \eta_b (q)    \nonumber  \\
&\quad    -       \sum_{   \substack{ 0 \le a < k \\ 0 \le b < \ell }   }  
F_{ a+1 , b } ( \hbar(p) , \hbar(q) )  \cdot   \overline{\partial} \hbar(p) \wedge   \eta_a (p)     \wedge   \eta_b (q)   \label{first A}  \\
(- 2\pi i )^{d-1}   B(p;q) 
  & = 
 \sum_{   \substack{ 0 \le a < k \\ 0 < b < \ell }   } 
F_{a,b}   ( \hbar(p) , \hbar(q) )  \cdot   \eta_a(p) \wedge   \partial \eta_b(q)   \nonumber  \\
&\quad   -          \sum_{   \substack{ 0 \le a < k \\ 0 \le b < \ell }   } 
F_{a,b +1 }   ( \hbar(p) , \hbar(q) )  \cdot   \eta_a(p)   \wedge   \partial \hbar(q) \wedge  \eta_b (q)   ,
\label{first B}
\end{align}
in which we have set
\begin{align*}
F_{a , b }(x,y)   & = \int_{ 1<v <  u  <\infty } u^a v^b e^{-ux} e^{- vy} \, \frac{du}{u} \frac{dv}{v}  \\
  &= \int_1^\infty  v^{a+b} e^{- vy} \left(  \int_1^\infty  u^a e^{-uvx}  \, \frac{du}{u}   \right) \frac{dv}{v} .
\end{align*}

If $a,b >0$, then  \eqref{exp integral} applies to the inner integral, leaving 
\begin{equation}\label{I third}
F_{a,b} (x,y) =    \sum_{i=0}^{a-1}    \frac{   (a-1)!   }{x^{a-i} \cdot i! }   
  \int_1^\infty  v^{b+i} e^{- v ( x+y) }
\frac{dv}{v} .
\end{equation}
Applying  \eqref{exp integral} once again  leaves
\begin{equation}\label{FabEval}
F_{a,b} (x,y)  
  =  e^{-x-y }     \sum_{i=0}^{a-1}  
 \frac{ \mbox{poly}(x,y)  }{ x^{a-i} \cdot (x+y)^{b +i }   } ,
\end{equation}
where in each term  $\mbox{poly}(x,y)$ is some polynomial (depending on $i$) in $x$ and $y$  whose exact value is irrelevant to us.
If $a>0$ and $b=0$ one argues in  the same way, except that the integral appearing in the $i=0$ term of  \eqref{I third} is $E_1(x+y)$.  Thus
 \begin{equation}\label{FabEval_a0}
F_{a,0}(x,y)  
 = 
 \frac{  E_1(x+y)  \cdot   (a-1)!   }{x^a  }  
 +   e^{-x-y } 
\sum_{i=1}^{a-1}    \frac{ \mbox{poly}(x,y) }{ x^{a-i} \cdot (x+y)^{ i }   }  . 
\end{equation} 
If  $a=0$ and $b>0$ then, again using \eqref{exp integral},  rewrite $F_{0,b}(x,y)$ as
\begin{equation}\label{pre-F0b}
 \int_1^\infty   \left(  \int_1^\infty v^b  e^{- v ( y + ux ) }  \,  \frac{dv}{v}   \right)   \frac{du}{u}   
 =   \sum_{i=0}^{b -1}      \frac{  (b -1)!  } {  i!}     \int_1^\infty    \frac{    e^{-( y + ux )}    }{( y + ux )^{b-i }  }  \,    \frac{du}{u}   .
\end{equation}
The  integral  on the right can again  be evaluated using elementary methods:  for any $r\ge 1$ we have
\begin{align*} \lefteqn{
   \int_1^\infty    \frac{    e^{-( y + ux )}    }{( y + ux )^r  }  \,    \frac{du}{u}   } \\
   & = 
     \frac{    e^{-y}    }{ y^r   }  E_1(x)    +       \sum_{j=1}^r    \frac{ (-1)^{j} } { (j-1)! }    \frac{ E_1(x+y)  }{ y^{ b-i-j+1 }   }     +  \sum_{j=2}^r       \frac{  e^{-(x+y)} \cdot \mbox{poly}(x,y)   }{ y^{ r-j+1 }  (x+y) ^{j-1} }  .
\end{align*}
Using this, one sees that \eqref{pre-F0b} has the form
\begin{align*}
F_{0,b}(x,y) & = 
  e^{-y} E_1(x)  \cdot  \frac{    \mbox{poly} (y)  }{y^b}      +  E_1(x+y)   \cdot \frac{   \mbox{poly}(y) }{y^b}     \\
  & \quad +      \sum_{j=1}^{b-1 }      
     \frac{  e^{-(x+y)}  \cdot  \mbox{poly}(x,y)  }{ y^{ b - j }  (x+y) ^j } .
\end{align*}

With these explicit formulas for the $F_{a,b}$ in hand,  let us  consider the behavior singularities of  \eqref{first A} after pullback via  $X^\dagger \to X$.

For the first sum of \eqref{first A}, one can use   \eqref{FabEval} and  \eqref{FabEval_a0} to write each term in the form
\begin{align*}\lefteqn{ 
F_{a,b}   \big( \hbar(p) , \hbar(q) \big)  \cdot   \overline{\partial}  \eta_a(p) \wedge \eta_b(q)    } \\
& =     
  \frac{ \overline{\partial}    \eta_a(p)   }{ \hbar(p)^{a}   } 
 \wedge 
 \frac{    \eta_b(q)   }{  \hbar(q) ^{b  }   }  
   \wedge   \left(   \frac{   \hbar(q) ^{b  }  }{  \hbar(s) ^{b }   }   \sum_{i=0}^{a-1}  
\phi_i  \frac{    \hbar(p)^{i}  } {  \hbar(s) ^{ i }  }  \right) \\
& \quad +   E_1( \hbar(s) )  \wedge   \frac{ \overline{\partial}    \eta_a(p)   }{ \hbar(p)^{a}   }   \wedge \psi .
\end{align*}
Here each $\phi_i$ is a smooth function on $X$, and $\psi$ is a smooth form (in fact, $\psi=0$ except when $b=0$).    The singularities of every form appearing here are understood:
\begin{itemize}
\item
The function in parentheses pulls back to a smooth function on $X^\dagger$, by Lemma \ref{lem:easy resolve}.  
\item
By the  analysis of singularities in the proof of Proposition \ref{prop:general green}, 
the pullback of  $\overline{\partial} \eta_a(p)/ \hbar(p)^a$ to the blowup along $Z(p) \subset X$  has logarithmic growth along the preimage of $Z(p)$, hence its pullback  to $X'$ has logarithmic growth along $D'$.
\item
Again by the proof of Proposition \ref{prop:general green}, the pullback of  $\eta_b(q)  / \hbar(q)^b$  to   the blowup along $Z(q) \subset X$  has logarithmic growth along the preimage of $Z(q)$, hence its pullback to  $X''$ has logarithmic growth along $D''$.
\item
By \eqref{Ei},  the function $E_1(\hbar(s))$    differs from $-\log  \hbar(s) $ by a smooth function.  Using the coordinates from the proof of Lemma \ref{lem:easy resolve}, one sees that $-\log  \hbar(s) $ pulls back to a function on $X'$ with logarithmic growth along $D'$, and also to a function on $X''$ with logarithmic growth along $D''$. 
\end{itemize}
 It follows that every term in the first summation in \eqref{first A} pulls back to a form on $X^\dagger$ with logarithmic growth along $D^\dagger$.

For the second sum of \eqref{first A}, one similarly uses  \eqref{FabEval} and  \eqref{FabEval_a0} to write each term as
\begin{align*}\lefteqn{ 
F_{a+1,b}   \big( \hbar(p) , \hbar(q) \big)  \cdot   \overline{\partial}  \hbar(p) \wedge  \eta_a(p) \wedge \eta_b(q)  } \\
& =     
\frac{  \overline{\partial}  \hbar(p)    }{ \hbar(p)    } \wedge   \frac{    \eta_a(p)   }{ \hbar(p)^{a}   } 
 \wedge 
 \frac{    \eta_b(q)   }{  \hbar(q) ^{b  }   }  
   \wedge   \left(   \frac{   \hbar(q) ^{b  }  }{  \hbar(s) ^{b }   }   \sum_{i=0}^{a}  
\phi_i  \frac{    \hbar(p)^{i}  } {  \hbar(s) ^{ i }  }  \right) \\
& \quad +   E_1( \hbar(s) )  \wedge  \frac{  \overline{\partial}  \hbar(p)    }{ \hbar(p)    }  \wedge   \frac{    \eta_{a }(p)   }{ \hbar(p)^{a}   }   \wedge \psi .
\end{align*}
The only new expression appearing here  is $ \overline{\partial}  \hbar(p) / \hbar(p)$.  As in the proof of Proposition \ref{prop:general green}, one can find local coordinates $z_1,\ldots, z_{\dim(X)}$ near a point of $Z(p) \subset X$ such that 
\[
h(p) = f \cdot ( | z_1|^2+\cdots + |z_k|^2) 
\]
for some smooth function $f$.  In these coordinates
\[
\frac{  \overline{\partial}  \hbar(p)    }{ \hbar(p)    }
= \overline{\partial} f + f \wedge \frac{ z_1 d\overline{z}_1+ \cdots  + z_k d\overline{z}_k }{| z_1|^2+\cdots + |z_k|^2} .
\]
The pullback of this form to the blow-up along $Z(p) \subset X$, which is defined by $z_1=\cdots = z_k =0$, has logarithmic growth along the preimage of $Z(p)$, as one immediately sees from the explicit coordinates on the blow-up given in the proof of Proposition \ref{prop:general green}.  Hence the pullback of $ \overline{\partial}  \hbar(p) / \hbar(p)$ to $X'$ has logarithmic growth along $D'$, hence all terms in the second sum in \eqref{first A} pull back to forms on $X^\dagger$ with logarithmic growth along $D^\dagger$.

This proves that \eqref{first A} satisfies \eqref{messy resolve}, and the argument for \eqref{first B} is entirely similar.
\end{proof}

As a special case of Theorem \ref{thm:star},  
\begin{align*}
\mathfrak{g}^\circ(s_1,\ldots, s_d) 
& =  \mathfrak{g}^\circ(s_1,\ldots, s_{d-1})  \star  \mathfrak{g}^\circ(s_d)   \\
& \quad 
 - \partial [ A(s_1,\ldots, s_{d-1} ; s_d) ] - \overline{\partial}[ B(s_1,\ldots, s_{d-1} ; s_d) ] .
\end{align*}
Repeated application of this results in
\begin{equation}\label{star iterate}
 \mathfrak{g}^\circ(s)      = 
 \mathfrak{g}^\circ(s_1)  \star   \cdots \star  \mathfrak{g}^\circ (s_d)  
  - \partial  [   \mathfrak{a}(s) ]  - \overline{\partial} [  \mathfrak{b}(s) ] 
\end{equation}
for locally integrable forms
\begin{align*}
\mathfrak{a}(s) 
 &  = \sum_{r=2}^d      A(s_1,\ldots ,   s_{r-1} ; s_r) \wedge \omega^\circ(s_{r+1}) \wedge \cdots  \wedge \omega^\circ(s_d)   \\
\mathfrak{b}(s) 
& = 
\sum_{r=2}^d    B(s_1,\ldots ,  s_{r-1}; s_r) \wedge \omega^\circ(s_{r+1}) \wedge \cdots  \wedge \omega^\circ(s_d)  .
\end{align*}

\subsection{Explicit logarithmic expansions}
\label{ss:explicit expansions}


We now return to the setting of \S \ref{ss:hu}, so that $X_0\subset X$ is a closed complex submanifold, but now assume that $X_0$  is presented to us in a particular way:
there is a holomorphic vector bundle $N \to X$ of dimension $\dim(X)-\dim(X_0)$ and a section
\[
u \in H^0(X,N)
\]
such that $X_0 \subset X$ is defined (as an analytic space) by the equation $u=0$.

This presentation of $X_0\subset X$ identifies
\begin{equation}\label{normal identification}
N_{X_0/X}  \iso N|_{X_0}.
\end{equation}
Indeed, if we denote by $\mathcal{I} \subset \co_X$  the ideal sheaf of holomorphic functions vanishing along $X_0$, then evaluation at $u$ defines an isomorphism $N^\vee \iso \mathcal{I}$.  
Restricting this to $X_0$ yields  an  isomorphism $ N|_{X_0}^\vee \iso \mathcal{I}/\mathcal{I}^2$ of vector bundles on $X_0$, and the normal bundle to $X_0 \subset X$ is (by definition) the dual of the right hand side.

Viewing points of the total space $N\to X$ as pairs $(x,v_x)$ consisting of a point $x\in X$ and a vector $v_x \in N_x$ in the fiber at $x$,  the deformation to the normal bundle of $X_0 \subset X$ can be  identified with the subset
\[
\widetilde{X}  \subset N  \times \C
\]
 of triples   $(x,v_x,t)$ consisting of a point $(x,v_x) \in N$, and a scalar $t \in \C$ satisfying  $t \cdot v_x=u_x.$
The morphisms 
\[
\xymatrix{
{X} & { \tilde{X} } \ar[r]^\tau \ar[l]_\pi  & {\C}  
}
\]
 of  \eqref{deformation diagram} are given by
 $\pi(x,v_x,t) =x$ and $\tau(x,v_x,t)=t$.
This is essentially McPherson's description of the deformation to the normal bundle, as in Remark 5.1.1 of \cite{Fulton}.

As in \S \ref{ss:GS basics},  fix a line bundle $L \to X$ with a hermitian metric $h$.   Any morphism of holomorphic vector bundles
$
y : N \to L^\vee ,
$
determines a section
\begin{equation}\label{degen section}
q = y(u) \in H^0(X,L^\vee)
\end{equation}
vanishing along $X_0$.
We call this the \emph{degenerating section} determined by $y$.
 Like any vector bundle,  $\pi_0 : N_{X_0/X} \to X_0$ acquires a tautological section 
\begin{equation}\label{v_0 taut}
v_0 \in H^0(   N_{X_0/X} ,   \pi_0^* N_{X_0/X} ) 
\end{equation}
after pullback via its own bundle map.
Setting $L_0=L|_{X_0}$, we may restrict  $y: N \to L^\vee$ to a morphism
 \[
 N_{X_0/X} \stackrel{\eqref{normal identification}}{=}  N|_{X_0}  \map{y} L^\vee|_{X_0} = L_0^\vee
 \]
of vector bundles on $X_0$, and then  pull back by $\pi_0:N_{X_0/X} \to X_0$.  
Applying this pullback to the tautological section \eqref{v_0 taut} defines the
 \emph{specialization to the normal bundle} of  the degenerating section \eqref{degen section},   denoted
\begin{equation}\label{very degen section}
\sigma_{X_0/X}(q) =    (\pi_0^*y) (v_0)  \in  H^0(  N_{X_0/X} ,   \pi_0^* L_0^\vee ).
\end{equation}

The  degenerating section \eqref{degen section} and its specialization \eqref{very degen section}  satisfy  the informal relation
\[
\sigma_{X_0/X}(q)   =  \frac{ \pi^* q}{\tau} \Big|_{\tau =0} ,
\]
which we formulate more precisely as the following lemma.

\begin{lemma}\label{lem:degenerate section spec}
For any $q=y(u)$ as above, there is a unique section
\[
\tilde{q} \in H^0( \tilde{X} , \pi^* L^\vee)
\]
satisfying $\tau \cdot \tilde{q} = \pi^* q$.  The pullback of $\tilde{q}$ to  $N_{X_0/X} =\tilde{X}_0$ is $\sigma_{X_0/X}(q)$.
\end{lemma} 

\begin{proof}
There is a tautological section 
$
v  \in H^0( \tilde{X} , \pi^*N)
$
 whose fiber at  a point $(x,v_x,t) \in \widetilde{X}$ is  $v_x$.
This section   satisfies  $\tau \cdot v = \pi^* u$,
and its restriction to  $N_{X_0/X}$  is  \eqref{v_0 taut}.
The image of $v$ under the map
\[
H^0( \tilde{X} , \pi^*N)  \map{\pi^* y} H^0( \tilde{X} , \pi^* L^\vee)
\] 
is a section $\tilde{q}$ with the desired properties.
\end{proof}

Now fix a smooth  tuple $s=(s_1,\ldots, s_d)$ with $s_i \in H^0(X,L^\vee)$, and assume 
$s=(p,q)$ is the concatenation of   
\[
p=( p_1,\ldots, p_k ) \quad\mbox{and}\quad q = (q_1 ,\ldots, q_\ell ),
\]
  satisfying the following properties:
\begin{enumerate}
\item
The tuple $p|_{X_0}$ formed from the restrictions 
\[
p_1|_{X_0}  ,\ldots , p_k|_{X_0}  \in H^0(X_0, L_0^\vee) 
\] 
is again smooth; equivalently, the analytic subspace 
\[
Z(p|_{X_0} ) =  Z(p) \times_X X_0 \subset X_0
\]
 is smooth of codimension $k$.
\item
The sections $q_1,\ldots, q_\ell \in H^0(X,L^\vee)$ are the degenerating sections determined by morphisms 
 $ y_1,\ldots, y_\ell : N \to L^\vee$ as above. 
In what follows, we denote by
\[
\sigma_{X_0/X}(q_i)    \in  H^0(  N_{X_0/X} ,   \pi_0^* L_0^\vee )
\]
the section associated to $q_i=y_i(u)$ by \eqref{very degen section}, and by
 \[
\tilde{q}_i \in H^0( \tilde{X} , \pi^* L^\vee)
\]
  the section associated to $q_i=y_i(u)$ by Lemma \ref{lem:degenerate section spec}.
 \end{enumerate}
Our assumptions imply that  $Z(p)$ intersects $X_0$ transversely,  while $X_0 \subset Z(q)$.  
We allow the possibility that $s=p$ or $s=q$.  Note that the tuples $p$ and $q$ are again smooth, by Remark \ref{rem:CM1}. 
We consider the specializations  of $Z(s)$, $Z(p)$, and $Z(q)$ to $N_{X_0/X}$.

 \begin{proposition}\label{prop:intersection cycle spec}
 For $s=(p,q)$ as above, the following properties hold.
 \begin{enumerate}
\item
 We have the equalities
$
\sigma_{X_0/X}( Z(p) )  = \pi_0^* Z( p|_{X_0} ) 
$
and
\[
  \sigma_{X_0/X}( Z( s ) ) = \sigma_{X_0/X}( Z( p ) ) \cdot  \sigma_{X_0/X}( Z( q ) ) ,
\]
of cycles on $N_{X_0/X}$.
 \item
The tuple
$
\tilde{q}   = ( \tilde{q}_1,\ldots, \tilde{q}_\ell )
$
is smooth, and the cycle 
\[
Z(\tilde{q}) \subset \tilde{X}
\]
defined by the  vanishing of its components satisfies the equality
\[
  \sigma_{X_0/X}( Z( q ) )    =Z(\tilde{q}) \cdot N_{X_0/X}  
  \]
  of cycles on $N_{X_0/X}$.
  \item
  The tuple  $\sigma_{X_0/X} ( q )=(\sigma_{X_0/X} ( q_1 ) , \ldots, \sigma_{X_0/X} ( q_\ell ))$   is smooth, and the analytic cycle on $N_{X_0/X}$ defined by the vanishing of its components is equal to $\sigma_{X_0/X}( Z( q ) )$.
\end{enumerate}
All intersections above are understood in the sense of  \cite{Draper}.
  \end{proposition} 
 
 \begin{proof}
  Let $\Delta \subset \C$ be the open unit disk, and abbreviate $n=\dim(X)$ and $m=\dim(X)-\dim(X_0)$.
The smoothness of $s=(p,q)$ implies that we may find a coordinate neighborhood in $X$ near a point of $X_0$ of the form
\[
U\iso \Delta^n=\{ (z_1,\ldots, z_n) : z_i \in \Delta \}
\]
 in such a way that 
\begin{itemize}
\item
the line bundle $L$ is trivial on $U$,
\item
$U_0=X_0\cap U$ is defined by the vanishing of $z_1, \ldots, z_m$,
\item
$p_1=z_{m+1},\ldots, p_k=z_{m+k}$,
\item
$q_1= z_1,\ldots, q_\ell=z_\ell$.
\end{itemize}
The deformation to the normal bundle of $U_0 \subset U$  is identified with 
\[
\tilde{U} = \{ (z_1,\ldots, z_n, w_1,\ldots, w_m,t) \in \Delta^n \times  \C^m \times \C 
 :  z_i = t w_i,  \, \forall 1\le i \le m   \} ,
\]
and $\tilde{q}_i = w_i$ for all $1\le i \le \ell$.  
The normal bundle itself is identified with
\[
N_{U_0/U} =  \{ (0,\ldots, 0, z_{m+1},\ldots, z_n,w_1,\ldots, w_m,0) \in \Delta^n \times  \C^m \times \C   \} ,
\]
and   $\sigma_{X_0/X}(q_i) = w_i$ for all $1\le i \le \ell$.  
The strict transforms of $Z(p)$ and $Z(q)$ are defined by (respectively) the vanishing of  $z_{m+1},\ldots,z_{m+k}$ and the vanishing of $w_1,\ldots,w_\ell$.
Their specializations to $N_{U_0/U}$ are defined by the same equations.
All parts of the proposition follow immediately from computations in these local coordinates.
 \end{proof}

 Now we turn to the  Green current 
 \[
  \mathfrak{g}^\circ(s)  \in E_X^{d-1,d-1}( \log Z(s) ) 
  \]
   of Proposition \ref{prop:general green}, and the similar currents $ \mathfrak{g}^\circ(p)$ and  $\mathfrak{g}^\circ(q)$.
 The following  lemmas are the key to understanding their logarithmic expansions  along $X_0$, and hence their specializations to $N_{X_0/X}$.

 \begin{lemma}\label{lem:nice expansions}
 There are forms  $a ,b ,c  \in E_{\tilde{X}}^{\ell-1,\ell-1}$ such that 
  \begin{equation}\label{degen g pre-expansion}
 \mathfrak{g}^\circ(q) =
   j_t^* \left( \frac{ a  }{ h(\tilde{q})^{\ell-1} }    +   b  \log  h( \tilde{q})    +  c \cdot  \log | \tau |    \right)
\end{equation}
for all $t\in \C\smallsetminus \{0\}$.   If we define currents 
\[
G_0(t) = \left( \frac{ a  }{ h(\tilde{q})^{\ell-1} }    +   b  \log  h( \tilde{q})   \right)  \wedge \delta_{\tilde{X}_t} 
\]
and $G_1(t) = c\wedge \delta_{\tilde{X}_t}$ on $\tilde{X}$,  then 
\[
 j_{t*} [  \mathfrak{g}^\circ(q) ] =  G_0(t) + G_1(t) \log|t|
 \]
 is a logarithmic expansion of $\mathfrak{g}^\circ(q)$ along $X_0$.
 \end{lemma}
 
\begin{proof}
The smoothness of $\tilde{q}$ allows us to apply the constructions of \S \ref{ss:GS basics} to obtain  Green forms  $\mathfrak{g}^\circ(\tilde{q})$ and $\mathfrak{g}^\circ( \tau \tilde{q})$  for the cycles $Z(\tilde{q}) \subset \tilde{X}$ and 
\[
Z(\tilde{q}) \smallsetminus N_{X_0/X}\subset \tilde{X}\smallsetminus N_{X_0/X},
\]
respectively.
Recalling that $\tau \tilde{q} = \pi^* q$, for $t\neq 0$ these are related by
\[
\mathfrak{g}^\circ(q) = j_t^* \pi^* \mathfrak{g}^\circ (q) 
= j_t^*  \mathfrak{g}^\circ(\tau \tilde{q})  .
\]

As in the proof Proposition \ref{prop:general green}, we may write
\begin{equation} \label{degen green deformation}
( -2\pi i )^{\ell-1}   \mathfrak{g}^\circ(   \tilde{q} )    = 
\eta_0(  \tilde{q})  \cdot  E_1( 2\pi h( \tilde{q}) )    +  e^{- 2\pi h(  \tilde{q} ) }   \sum_{j=1}^{\ell-1}    \frac{ \eta_j ( \tilde{q})  }{ h(  \tilde{q})^j }   
 \cdot P_j( h( \tilde{q}) ), 
\end{equation}
where  $P_j$ is a polynomial and $\eta_j(\tilde{q})$ is a smooth form on $\tilde{X}$ satisfying the homogeneity property \eqref{eta homo}.
 If we replace $\tilde{q}$ by $\tau\tilde{q}$ in \eqref{degen green deformation}, pull back by $j_t: X \to \tilde{X}$, and use
\[
j_t^*  \left( \frac{  \eta_j(\tau \tilde{q}) } { h(\tau \tilde{q})^j }  \right) = 
j_t^*  \left( \frac{  \eta_j( \tilde{q}) } { h(\tilde{q})^j  }  \right) ,
\]
 we find that 
 $
 \mathfrak{g}^\circ( q ) = j_t^*  \mathfrak{g}^\circ( \tau \tilde{q} ) = (-2\pi i )^{1-\ell} j_t^* \Psi
 $
  where
\[
\Psi=  \eta_0 ( \tilde{q} )  \cdot  E_1( 2\pi | \tau |^2 h(  \tilde{q}) ) 
 +  e^{- 2\pi |\tau|^2  h( \tilde{q} ) }    \sum_{j=1}^{\ell-1}   \frac{ \eta_j (  \tilde{q})  }{ h( \tilde{q})^j } 
 \cdot P_j ( h( \tau \tilde{q}) ).
\]
The equality \eqref{degen g pre-expansion} follows easily from this and  \eqref{Ei}.

Applying $j_{t*}$ to both sides of \eqref{degen g pre-expansion} yields
\[
j_{t*} [  \mathfrak{g}^\circ(q) ] =
 \left( \frac{ a  }{ h(\tilde{q})^{\ell-1} }    +   b  \log  h( \tilde{q})    +  c \cdot  \log | t |    \right) \wedge \delta_{\tilde{X}_t}.
\]
To show that this is a logarithmic expansion, one must  verify the continuity and H\"older continuity  at $t=0$
of  $G_0(t)(\varphi)$ and $G_1(t)(\varphi)$, respectively, for any smooth compactly supported form $\varphi$ on $\tilde{X}$.  Using a partition of unity argument, we may reduce to the case in which the support of $\varphi$ is contained in a coordinate neighborhood
\begin{align*}
\tilde{U}  & = \{ (z_1,\ldots, z_n, w_1,\ldots, w_m,t) \in \Delta^n \times  \C^m \times \C 
 :  z_i = t w_i,  \, \forall 1\le i \le m   \}  \\
 & \subset \{ (z_{m+1},\ldots, z_n, w_1,\ldots, w_m,t) \in  \C^{n-m} \times  \C^m \times \C    \} 
\end{align*}
chosen as in the proof of Proposition \ref{prop:intersection cycle spec}.
In particular,  $\tilde{q}_i = w_i$ for all $1\le i \le \ell$, and the function $h(\tilde{q})$  has the form
\[
H(z,w,t) = f_1(z,w,t) \cdot | w_1|^2 + \cdots + f_\ell(z,w,t) \cdot |w_\ell |^2
\]  
for smooth compactly supported  $f_1,\ldots, f_\ell : \C^{n-m} \times  \C^m \times \C  \to \R^{>0}$.

The continuity of $G_0(t)(\varphi)$ now amounts to  the continuity in $t$ of 
\[
\int_{\C^{n-m} \times  \C^m} \frac{g(z,w,t)}{ H(z,w,t)^{\ell-1}    } \cdot  \mu
\quad\mbox{and}\quad
\int_{\C^{n-m} \times  \C^m} g(z,w,t) \cdot     \log H(z,w,t)    \cdot  \mu 
\]
 for any smooth compactly supported function $g(z,w,t)$ on $\C^{n-m} \times  \C^m \times \C$,  where
\[
\mu=
dz_{m+1}\wedge  d\bar{z}_{m+1} \wedge \cdots \wedge dz_n\wedge  d\bar{z}_n \wedge
 dw_1\wedge  d\overline{w}_1 \wedge \cdots \wedge dw_m\wedge  d\overline{w}_m .
\]
The smoothness (hence H\"older continuity)  of $G_1(t)(\varphi)$  amounts to the smoothness in $t$ of 
 \[
\int_{\C^{n-m} \times  \C^m} g(z,w,t)  \cdot  \mu .
\]
These are routine calculus exercises, left to the reader.
\end{proof}

For the Green currents  $\mathfrak{g}^\circ(q_i)$ associated to the individual components of $q=(q_1,\ldots, q_\ell)$, 
one has a more precise version of Lemma \ref{lem:nice expansions}.

\begin{lemma}\label{lem:very nice expansions}
 For $1 \le i \le \ell$ there is a  smooth function $f_i$ on $\tilde{X}$ such that 
\[
 \mathfrak{g}^\circ(  q_i)    
=   j_t^* \left(       -\log( 2\pi e^\gamma h(  \tilde{q}_i))   +| \tau |^2    f_i  - 2 \log | \tau |  \right)
\]
 for all $t\in \C\smallsetminus \{0\}$.   If we define currents
 \[
G_0 =       \left(  -\log( 2\pi e^\gamma  h(  \tilde{q}_i)  )  +| \tau |^2    f_i  \right) \wedge \delta_{\tilde{X}_t} 
 \]
and $G_1 =  - 2  \delta_{\tilde{X}_t}$ on $N_{X_0/X}$, then 
 \[
 j_{t*} [    \mathfrak{g}^\circ(  q_i)    ]
=  G_0     + G_1   \cdot  \log | t |  
\]
is a logarithmic expansion of $ \mathfrak{g}^\circ(  q_i)$ along $X_0$.
 \end{lemma}

 \begin{proof}
The proof is the same as that of Lemma \ref{lem:nice expansions}, except that one  replaces \eqref{degen green deformation}  with the simpler  equality
$
\mathfrak{g}^\circ(  \tilde{q}_i)   = E_1( 2\pi h(  \tilde{q}_i) )
$
of  \eqref{simplest green}.
\end{proof}

\begin{proposition}\label{prop:the specializations}
  We have the equality of currents
 \[
  \sigma_{X_0/X} (   \mathfrak{g}^\circ(p ) )  = [ \pi_0^* \mathfrak{g}^\circ( p|_{X_0} )]  \in D^{k-1,k-1}_{N_{X_0/X}} 
\]
 where $\pi_0: N_{X_0/X} \to X_0$ is the bundle map, and there are smooth forms $a_0$ and $b_0$ on $N_{X_0/X}$ such that 
  \[
\sigma_{X_0/X}(  \mathfrak{g}^\circ (q  ) ) 
=    \frac{  a_0  }{ h( \sigma_{X_0/X}(q) )^{\ell-1} }    +   b_0  \log  h( \sigma_{X_0/X}(q)  )   \in D^{\ell-1,\ell-1}_{N_{X_0/X}} .
\]
For the individual components of $q=(q_1,\ldots, q_\ell)$ we have the exact formula
\[
\sigma_{X_0/X}(  \mathfrak{g}^\circ(q_i) ) =    -  \log    h( \sigma_{X_0/X}( q_i )  )    -        \log   ( 2\pi e^\gamma)  
  \in D^{0,0}_{N_{X_0/X}} .
\]
 \end{proposition}

\begin{proof}
For $\mathfrak{g}^\circ(p)$, note that the tuple $p$ remains smooth (as one can check in the local coordinates of Proposition \ref{prop:intersection cycle spec}) after pullback via any arrow in
\[
\xymatrix{
{N_{X_0/ X} } \ar[r]^{ j_0 } \ar[d]_{\pi_0} & { \tilde{X} } \ar[d]^{\pi}  \\
{X_0} \ar[r]_{i_0} & { X. } 
}
\]
Each of these pullbacks has its own Green form $\mathfrak{g}^\circ(\cdot)$ associated to it, and these satisfy obvious functorial properties, e.g. $\pi^* \mathfrak{g}^\circ(p) =  \mathfrak{g}^\circ (\pi^* p)$.
For any $t\neq 0$ we have the (particularly simple)  logarithmic expansion
\[
 j_{t*} [ \mathfrak{g}^\circ(p) ] =  \mathfrak{g}^\circ(\pi^*p) \wedge \delta_{ \tilde{X}_t} ,
\]
 of  $\mathfrak{g}^\circ(p)$ along $X_0 \subset X$.  
 Of course one must check that the family of currents on the right hand side is defined  at $t=0$, and  satisfies the continuity condition of Definition \ref{def:log expansion}; using the analysis of singularities of $\mathfrak{g}^\circ(\pi^*p)$ from \eqref{green singularities},  this is an easy calculation in  local coordinates  as in  the proof of Lemma \ref{lem:nice expansions}.
The  constant term at $t=0$ of this expansion is
\[
\mathfrak{g}^\circ(\pi^*p) \wedge \delta_{ N_{X_0/X} } 
= j_{0*} [  \mathfrak{g}^\circ (\pi^* p)|_{N_{X_0/X}} ]
=j_{0*} [  \pi_0^* \mathfrak{g}^\circ( p|_{X_0} ) ] ,
\]
proving the  first claim.

The claims about $\mathfrak{g}^\circ(q)$ and $\mathfrak{g}^\circ(q_i)$ follow by taking $t=0$ in the logarithmic expansions of Lemmas \ref{lem:nice expansions} and \ref{lem:very nice expansions}, and recalling from Lemma \ref{lem:degenerate section spec} that the restriction  of $\tilde{q}$ to the fiber $N_{X_0/X} = \tilde{X}_0$ is $\sigma_{X_0/X}(q)$.
\end{proof}

\begin{remark}
Using  Proposition \ref{prop:general green} and \eqref{simplest green},
each section  $\sigma_{X_0/X} (q_i)$  of the hermitian line bundle $\pi_0^*L_0^\vee$ on  $N_{X_0/X}$ determines a Green function 
\[
\mathfrak{g}^\circ( \sigma_{X_0/X} (q_i)  ) = E_1(2\pi  h(\sigma_{X_0/X} (q_i) )  ) 
\]
for the divisor $\sigma_{X_0/X} (q_i) =0$ on $N_{X_0/X}$. 
By  the third claim of Proposition \ref{prop:intersection cycle spec}, this divisor is none other than the specialization of $Z(q_i) \subset X$ to the normal bundle, which also admits the Green function $\sigma_{X_0/X}   ( \mathfrak{g}^\circ( q_i)  )$ obtained by specializing $\mathfrak{g}^\circ( q_i)$.
 Proposition \ref{prop:the specializations} shows that
 \[
 \mathfrak{g}^\circ( \sigma_{X_0/X} (q_i)  )  \neq  \sigma_{X_0/X}   ( \mathfrak{g}^\circ( q_i)  ) .
 \]
 This should not cause confusion, as the Green function on the left hand side plays no role in our arguments, and will never appear again.  
 \end{remark}

  \begin{proposition}\label{prop:special star}
  The  specializations of    $\mathfrak{g}^\circ(s)$, $\mathfrak{g}^\circ(p)$, and $\mathfrak{g}^\circ(q)$ to $N_{X_0/X}$ are related by 
   \begin{align*}
\sigma_{X_0/X} ( \mathfrak{g}^\circ(s)  ) 
 & =   \sigma_{X_0/X} (   \mathfrak{g}^\circ(p ) )  \star   \sigma_{X_0/X}( \mathfrak{g}^\circ( q )  )   \\
& \quad  - \partial \sigma_{X_0/X}(  A(p ; q ) )  - \overline{\partial}  \sigma_{X_0/X} ( B(p ; q ) ) ,
\end{align*}
where  $A(p ; q)$ and $B(p ; q)$ are the currents of Theorem \ref{thm:star}.  
 Moreover,
\begin{align*}
\sigma_{X_0/X} ( \mathfrak{g}^\circ(q) )   & =  \sigma_{X_0/X}(  \mathfrak{g}^\circ(q_1) ) \star \cdots \star \sigma_{X_0/X}(  \mathfrak{g}^\circ(q_\ell) )  \\
& \quad -  \partial \sigma_{X_0/X} ( \mathfrak{a}(q) )  -   \overline{\partial} \sigma_{X_0/X} ( \mathfrak{b}(q) )  
\end{align*}
where  $\mathfrak{a}(q)$ and  $\mathfrak{b}(q)$  are the currents of \eqref{star iterate}.
In particular, all currents on $X$ appearing in these formulas admit logarithmic expansions along $X_0$, and the star products in both formulas are defined. 
\end{proposition}
 
 \begin{proof}
 The core of the proof is the following lemma.

 \begin{lemma}\label{lem:star spec alt}
Suppose $G\in D_X^{k-1,k-1}$ is any Green current for $Z(p)$.
If $G$ admits a logarithmic expansion along $X_0 \subset X$, then so does   $G \star \mathfrak{g}^\circ(q)$, and its specialization to the normal bundle satisfies
\[
 \sigma_{X_0/X} (   G \star  \mathfrak{g}^\circ( q )  ) 
 = \sigma_{X_0/X} (  G)  \star   \sigma_{X_0/X}( \mathfrak{g}^\circ( q )  ) .
 \]
 In particular, the star product on the right is  defined.
\end{lemma}

\begin{proof}
Abbreviate $Z=Z(p)$, and recall from Lemma \ref{lem:stars make sense} that the star product 
\[
 G \star    \mathfrak{g}^\circ ( q ) 
 = \delta_Z \wedge   \mathfrak{g}^\circ ( q)   + G \wedge \omega^\circ( q ) 
\]
is defined.   Applying $j_{t^*}$ to both sides   results in
\[
j_{t*} [ G \star    \mathfrak{g}^\circ ( q )  ]  (\varphi)
=  \int_Z  \mathfrak{g}^\circ ( q ) \wedge j_t^* \varphi
+ ( j_{t*}G) ( \pi^* \omega^\circ( q ) \wedge  \varphi) ,
\]
 for any smooth compactly supported form $\varphi$ on $\tilde{X}$, and any $t\neq 0$.

 Using $j_t : Z \iso \tilde{Z}_t$ and the equality 
 \[
 \mathfrak{g}^\circ(q) =
   j_t^* \left( \frac{ a  }{ h(\tilde{q})^{\ell-1} }    +   b  \log  h( \tilde{q})    +  c \cdot  \log | \tau |    \right)
\]
of Lemma \ref{lem:nice expansions}, the integral on the right becomes
\[
 \int_Z  \mathfrak{g}^\circ ( q ) \wedge j_t^* \varphi
 =
 \int_{\tilde{Z}_t }    \left( \frac{ a  }{ h(\tilde{q})^{\ell-1} }    +   b  \log  h( \tilde{q})    +  c \cdot  \log | \tau |    \right) \wedge  \varphi .
\]
 Fixing a logarithmic expansion $j_{t*} G    =    \sum_{i \ge 0}   G_i  (t)     ( \log |t| )^i$, we obtain
 \[
j_{t*} [  G \star    \mathfrak{g}^\circ ( q )  ]
=  \sum_{i\ge 0} C_i(t)  \cdot ( \log|t|)^i 
\]
in which 
\begin{align*}
C_0(t) & =     \delta_{\widetilde{Z}_t}     \wedge  \left( \frac{ a  }{ h(\tilde{q})^{\ell-1} }    +   b  \log  h( \tilde{q})    \right) 
   + G_0(t) \wedge  \pi^* \omega^\circ( q )   \\
C_1(t) & =  c \wedge  \delta_{\widetilde{Z}_t}    + G_1  (t)  \wedge  \pi^* \omega^\circ(q)  \\
C_i(t)  & =  G_i  (t)  \wedge  \pi^* \omega^\circ(q) \quad \mbox{for } i > 1.
\end{align*}
To see that this is a logarithmic expansion of $G \star    \mathfrak{g}^\circ ( q )$   one must check that the terms involving $\delta_{\tilde{Z}_t}$ are well-defined currents (including at $t=0$) that satisfy the continuity conditions of Definition \ref{def:log expansion};
this is easily  verified in the local coordinates of the proof of Proposition \ref{prop:intersection cycle spec}.

The current $C_0(0)$ is the pushforward via $j_0 : N_{X_0/X} \to \tilde{X}$ of 
\[
  \delta_{ \sigma_{X_0/X}(Z) }  \wedge   \sigma_{X_0/X}( \mathfrak{g}^\circ(q) )        +  \sigma_{X_0/X}(G) \wedge  \pi_0^* i_0^* \omega^\circ( q ) 
\]
which agrees with $\sigma_{X_0/X} (  G)  \star   \sigma_{X_0/X}( \mathfrak{g}^\circ( q )  )$ by Proposition \ref{prop:green specialization}.
\end{proof}

Recall the equality
\[
\mathfrak{g}^\circ(s) 
=   \mathfrak{g}^\circ(p)  \star  \mathfrak{g}^\circ(q) 
 - \partial [ A(p ; q) ] - \overline{\partial} [ B(p ; q) ]
 \]
 of Theorem \ref{thm:star}.  The currents $A(p,q)$ and $B(p,q)$ admit logarithmic expansions along $X_0$ by Theorem \ref{thm:hu} and the final claim of Theorem \ref{thm:star}. 
  The star product admits a logarithmic expansion by Lemma \ref{lem:star spec alt}.
The Green current on the left admits a logarithmic expansion by Theorem \ref{thm:hu} and \eqref{green is log}, and also because the right hand side does.
Specializing both sides to $N_{X_0/X}$ and using  Remark  \ref{rem:d specialization} and Lemma \ref{lem:star spec alt} proves the first claim of Proposition \ref{prop:special star}.

 For the second claim we use the following lemma.
 
 \begin{lemma}\label{lem:degen iterate}
Fix $1\le r <\ell$, and let  $G\in D_X^{r-1,r-1}$ be any Green current for $Z( q_1,\ldots, q_r )$.  
If $G$ admits a logarithmic expansion along $X_0 \subset X$, then so does $G\star \mathfrak{g}^\circ(q_{r+1})$, and 
\[
\sigma_{X_0/X}( G\star \mathfrak{g}^\circ( q_{r+1} )) = 
\sigma_{X_0/X}( G ) \star  \sigma_{X_0/X}(  \mathfrak{g}^\circ(q_{r+1}) ) .
\]
In particular, the star product on the right is defined.
\end{lemma}

\begin{proof} 
The proof is virtually identical to that of Lemma \ref{lem:star spec alt}, using Lemma \ref{lem:very nice expansions} in place of Lemma \ref{lem:nice expansions}.
\end{proof}

To complete the proof of the second claim of Proposition \ref{prop:special star}, we begin with the equality
 \[
 \mathfrak{g}^\circ(q)      = 
 \mathfrak{g}^\circ(q_1 )  \star   \cdots \star  \mathfrak{g}^\circ (q_\ell)  
  - \partial  [   \mathfrak{a}(q) ]  - \overline{\partial} [  \mathfrak{b}(q) ] 
\]
of \eqref{star iterate}. 
Applying Lemma \ref{lem:degen iterate} inductively allows us to specialize both sides to $N_{X_0/X}$, and also shows that
\begin{align*}
\sigma_{X_0/X} ( \mathfrak{g}^\circ(q_1) \star \cdots \star \mathfrak{g}^\circ(q_r) )  
 = \sigma_{X_0/X} ( \mathfrak{g}^\circ(q_1) ) \star \cdots  \star \sigma_{X_0/X} (  \mathfrak{g}^\circ(q_r) ) . 
\end{align*}
Recalling  Remark  \ref{rem:d specialization}, we obtain the desired formula.
 \end{proof}


\section{Orthogonal Shimura varieties}
\label{s:ortho}


We now apply the general theory of the previous subsections to the special case in which $X$ is either the hermitian symmetric domain $\DD$ associated to an orthogonal group over a totally real field, or the complex  Shimura variety $M(\C)$ determined by such a group.  This allows us to prove our main result: a description of the behavior of special arithmetic cycles on the canonical model $M$ under pullback via the inclusion $M_0 \to M$ of a smaller orthogonal Shimura variety.


\subsection{The hermitian symmetric domain}
\label{ss:ortho hermitian}


Let $(V,Q)$  be a quadratic space of dimension $n+2\ge 3$ over a totally real number field $F$.
Assume there is one embedding $\sigma : F \to \R$  for which the real quadratic space
\[
V_\sigma=V\otimes_{ F,\sigma } \R
\]
  has signature  $(n,2)$, while $V_\tau=V\otimes_{F,\tau}\R$ is positive definite for all embeddings $\tau\neq \sigma$.   Denote by 
 \begin{equation}\label{bilinear}
 [x,y] = Q(x+y) - Q(x) -Q(y)
 \end{equation}
the associated  $F$-bilinear form on $V$. 
Extend it $\R$-bilinearly to $V_\sigma$, and $\C$-bilinearly to $V_\sigma \otimes_\R\C$.

The data $(V,Q)$ determines a hermitian symmetric domain
\[
\DD =  \{ z\in V_\sigma \otimes_\R \C : [z,z] =0,\, [z,\bar{z}] <0 \} / \C^\times \subset \mathbb{P}( V_\sigma \otimes_\R \C  ) .
\]
Denote by 
\[
V_\DD = V_\sigma \otimes_\R \co_\DD
\]
 the constant vector bundle on $\DD$  whose fiber at every point is $V_\sigma$.   
It comes equipped with a  symmetric bilinear pairing 
\begin{equation}\label{V_D pairing}
 [ \cdot , \cdot ]   : V_\DD  \times V_\DD  \to \co_\DD,
\end{equation}
which on fibers is just the $\C$-bilinear pairing  induced by \eqref{bilinear}.

The  vector bundle $V_\DD$ is equipped with a   filtration by $\co_\DD$-module local direct summands
\begin{equation}\label{D filtered}
L_\DD \subset L_\DD^\perp \subset V_\DD,
\end{equation}
whose fibers at any point $z\in \DD$  are identified with the subspaces
\begin{equation*}
\C z \subset (\C z)^\perp \subset V_\sigma \otimes_\R\C.
\end{equation*}
In particular $L_\DD$  is  isotropic under the pairing \eqref{V_D pairing}, which induces an isomorphism
\begin{equation}\label{dR duality}
V_\DD  / L_\DD ^\perp   \iso L_\DD^\vee.
 \end{equation}

At each  $z\in \DD$ we endow the  isotropic line 
\[
L_{\DD,z} =\C z \subset V_\sigma \otimes_\R \C
\]  
 with the positive definite hermitian form $h$ determined by
  \begin{equation}\label{natural metric}
h(z,z) = - \frac{ [ z,\bar{z} ] }{2} .
\end{equation}
This makes $L_\DD$ into a hermitian line bundle.

Using \eqref{dR duality},   any $x \in V_\sigma$ determines first a global section of $V_\DD$, and then a global section 
\begin{equation}\label{domain section}
s(x)  \in H^0(\DD, L_\DD^\vee)
\end{equation}
with zero locus the smooth analytic divisor
\[
Z_\DD(x) =  \{ z \in\DD : [ z ,  x ]=0  \}.
\]

More generally,  any  tuple $x=(x_1,\ldots, x_d) \in V_\sigma^d$ determines a tuple 
$
s(x) =(s(x_1) ,\ldots, s(x_d) )
$
 of sections, and we denote by  
 \[
 Z_\DD(x)  \subset \DD
 \]
  the analytic subspace defined by the vanishing of all components.    
  In other words, $Z_\DD(x)$ is those lines  $\C z \subset \DD$ such that $[z, x_i]=0$ for all $1\le i \le d$.  
This is a complex submanifold, which  depends only on  $\mathrm{Span}_\R\{ x_1 , \ldots, x_d\}  \subset V_\sigma$. 
 It  is nonempty precisely when  this subspace is positive definite, in which case it has codimension  $\dim_\R \mathrm{Span}_\R\{ x_1 , \ldots, x_d\}$.
Recalling the notation and terminology of \S \ref{ss:GS basics},  the smoothness and regularity of the tuple $s(x)$ are both equivalent to the linear independence of the vectors $x_1,\ldots, x_d$.

Now fix a  positive definite $v \in \Sym_d(\R)$ and  an   $\alpha \in \GL_d(\R)$ with positive determinant such that  
\[
v =\alpha \cdot{}^t \alpha.
\]
If $x\in V_\sigma^d$ is a tuple with linearly independent components,  we may form a new $d$-tuple $x\alpha \in V_\sigma^d$, and hence a corresponding smooth tuple $s(x \alpha)$ of sections of $L_\DD^\vee$.  
Applying the constructions of \S \ref{ss:GS basics} to this tuple of sections determines  forms
\begin{align}\label{alpha}
\mathfrak{g}_\DD^\circ( x,v ) & =  \mathfrak{g}^\circ( s(x  \alpha)  )  \in E^{d-1,d-1}_{ \DD\smallsetminus Z(x) }  \\
  \omega_\DD^\circ( x, v )  & =\omega^\circ( s(x  \alpha) ) \in E^{d,d}_\DD  \nonumber
\end{align}
related by   the Green equation
 \[
 dd^c[  \mathfrak{g}_\DD^\circ(x,v) ]  + \delta_{Z_\DD(x)} = [ \omega_\DD^\circ (x,v)] .
 \]
These forms  are independent of $\alpha$ by  Proposition 2.6(d) of \cite{GS}.


\subsection{Canonical models}
\label{ss:canonical models}


The quadratic space $(V,Q)$ determines a short exact sequence
\[
1 \to \mathbb{G}_{m /F} \to \mathrm{GSpin}(V) \to \mathrm{SO}(V) \to 1
\]
of reductive groups over $F$.  From now on we denote by $G$ either
\[
\mathrm{Res}_{F/\Q} \, \mathrm{GSpin}(V) \quad \mbox{or} \quad \mathrm{Res}_{F/\Q}\,  \mathrm{SO}(V).
\]
For our purposes these two groups are interchangable.
The group $G(\R)$ acts on $\DD$ via the  projection 
\[
G(\R) \to \prod_{\tau : F\to \R} \mathrm{SO}(V_\tau) \to  \mathrm{SO}(V_\sigma),
\]
 and the pair $(G,\DD)$ is a  Shimura datum.
A choice of sufficiently small  compact open subgroup $K \subset G(\A_f)$  determines a smooth quasi-projective variety $M$ over 
the reflex field $F\iso \sigma(F) \subset \C$ with $\C$-points 
\[
M(\C) = G(\Q) \backslash \DD \times G(\A_f) / K .
\]
It is projective if and only if $V$ is anisotropic.
Any $g\in G(\A_f)$ determines an open and closed submanifold
\begin{equation}\label{shimura component}
( gKg^{-1} \cap G(\Q)  ) \backslash  \DD \subset  M(\C),
\end{equation}
where  the inclusion is $z\mapsto (z,g)$.

For a point $z\in \DD$, the action of any $\gamma \in G(\R)$ on $V_\sigma \otimes_\R\C$ identifies the fibers of \eqref{D filtered} at $z$ and $\gamma z$.
This allows us to descend the filtered vector bundle \eqref{D filtered} from $\DD$ to every quotient \eqref{shimura component}.    By the theory of canonical models of automorphic vector bundles, there is a canonical filtered vector bundle
\[
L_M \subset L_M^\perp \subset V_M 
\]
on $M$ whose restriction to  \eqref{shimura component} agrees with this descent.

In a similar way, the  pairing \eqref{V_D pairing} descends  to an $\co_M$-bilinear pairing
\begin{equation}\label{canonical bilinear}
[\cdot,\cdot] : V_M \times V_M \to \co_M,
\end{equation}
under which $L_M$ and $L_M^\perp$ are the orthogonal sub-bundles to one another, 
and this pairing induces   a canonical isomorphism
\begin{equation}\label{canonical dR duality}
V_M/L_M^\perp \iso L_M^\vee 
\end{equation}
that agrees with \eqref{dR duality} under the complex uniformization.

The vector bundle $V_M$ is nonconstant, but it is \emph{infinitesimally constant} in the sense that it carries a flat connection
\begin{equation}\label{V_M connection}
\nabla : V_M \to V_M \otimes_{\co_M} \Omega^1_M
\end{equation}
characterized by the property that its pullback to $\DD$ via the uniformization \eqref{shimura component} agrees with the constant connection $\mathrm{id} \otimes d$ on $V_\DD = V_\sigma \otimes_\R \co_\DD$.  This allows us to perform parallel transport through square-zero thickenings:

\begin{proposition}\label{prop:parallel transport}
Suppose $M_0 \subset M$ is a closed subscheme, smooth over $F$.   A flat section of $V_M|_{M_0}$ extends uniquely to a flat section  of $V_M$ over the first order infinitesimal neighborhood  of $M_0$ in $M$.
\end{proposition}

\begin{proof}
This can be extracted from the  arguments of  \S 2 of \cite{BO78}.
In fact, as we are working with smooth schemes in characteristic $0$, the results of \emph{loc.~cit.} can be used to show that a flat section defined over $M_0$  extends uniquely to a flat section over the entire formal completion along $M_0 \subset M$.
We  instead sketch a more direct argument  working  only over the first order infinitesimal neighborhood  $M_0^\square \subset  M$.  
Thus  $M_0^\square$ is the closed subscheme defined    by the square $I^2 \subset \co_M$ of the ideal sheaf $I\subset \co_M$ defining $M_0 \subset M$.

Denote by 
$
U  \subset M\times_F M
$
the first order  infinitesimal neighborhood of the diagonal $M \subset  M \times_F M$, and let 
$p_1,p_2 :  U \to M$ be the projection maps.   
 By Proposition 2.9 of \emph{loc.~cit.}, the connection $\nabla$ determines an isomorphism
\[
p_1^* V_M \iso p_2^* V_M
\]
of vector bundles on $U$, satisfying a cocycle relation encoding the flatness of the connection.
This cocycle condition implies that the above isomorphism is an isomorphism of vector bundles \emph{with connections}, where the left and right hand sides are endowed with the pullbacks of $\nabla$ through $p_1$ and $p_2$, respectively.

The smoothness of $M_0$ implies  that, Zariski  locally on $M_0^\square$,  one can find a retraction $\rho: M_0^\square \to M_0$.  Denoting by $i$ and $i^\square$ the inclusions of $M_0$ and $M_0^\square$  into $M$, the product morphism
\[
(  i^\square  ,  i \circ \rho)   : M_0^\square \to M \times_F M
\]
factors through $U$, and hence the pullbacks of $V_M$ by $i^\square$ and $i \circ \rho$ are isomorphic as vector bundles with connections.
The resulting isomorphism 
\[
V_M|_{M_0^\square} \iso \rho^*( V_M|_{M_0} ),
\]
 induces a homomorphism
\[
 H^0(  M_0 , V_M|_{M_0} )^{\nabla =0}  \map{\rho^*}
    H^0(  M_0^\square ,  \rho^*( V_M|_{M_0})  )^{\nabla =0} 
 \iso 
   H^0(  M_0^\square , V_M|_{M_0^\square} )^{\nabla =0} 
\]
of spaces of flat sections, and it is not difficult to check that the first arrow is an isomorphism.
Note that the composition does not depend on the choice of retraction $\rho$, because this is true of its inverse  ``restrict to $M_0$".

This proves  the  existence and uniqueness of flat extensions of flat sections over  open subsets small enough that the required retractions exist, and the uniqueness allows us to glue the sections together over an open cover.
 \end{proof}


\subsection{Arithmetic cycle classes}
\label{ss:canonical cycle construction}


Fix an integer $d$  with $1\le d\le n+1$.
The group $G(\A_f)$ acts on 
\[
\widehat{V} =V \otimes_\Q \A_f,
\]
  and we fix a  $K$-invariant 
$\Z$-valued Schwartz function $\varphi \in  S( \widehat{V}^d )$.

Any $g\in G(\A_f)$ and $T \in \Sym_d(F)$ determine an analytic cycle 
\begin{equation}\label{component cycle}
Z_\DD(T,\varphi)_g =  \sum_{  \substack{ x \in  V^d  \\  T( x ) = T } }  \varphi( g^{-1} x) \cdot Z_\DD( x ),
\end{equation}
on $\DD$.
Here we denote by 
\begin{equation}\label{moment matrix}
T(x) =  \left(  \frac{ [ x_i , x_j ] }{2}   \right)  \in \Sym_d(F)
\end{equation}
the moment matrix of a tuple $x\in V^d$.
The cycle \eqref{component cycle}  descends to the quotient \eqref{shimura component}, and  varying $g$ yields an analytic cycle $Z_M(T,\varphi)(\C)$ on $M(\C)$. 
Being expressible as a union of smaller Shimura varieties constructed in the same way as $M$, this cycle  is the complexification of an algebraic cycle $Z_M(T,\varphi)$   of codimension $\mathrm{rank}(T)$ on the canonical model $M$.

Fix a positive definite $v \in \Sym_d(\R)$ and assume  $\det(T)\neq 0$.
As in \S 4.3 of \cite{GS}, the sums
   \begin{align} \label{component green}
\mathfrak{g}_\DD^\circ(T,v ,\varphi)_g
&  =  
 \sum_{  \substack{ x \in  V^d  \\  T(x ) = T } }  \varphi (g^{-1} x)  \cdot 
  \mathfrak{g}_\DD^\circ ( x ,v )  \in E^{d-1,d-1}_{ \DD \smallsetminus Z_\DD(T,\varphi)(\C)_g }  \\
\omega_\DD^\circ(T,v,\varphi)_g
& = 
\sum_{  \substack{ x \in  V^d  \\  T(x ) = T }  }  
\varphi ( g^{-1} x) \cdot   \omega_\DD^\circ ( x ,v )   \in E^{d,d}_\DD \nonumber
\end{align}
also descend to the quotient \eqref{shimura component}.
Again by varying $g$, we obtain forms 
\[
\mathfrak{g}_M^\circ(T,v ,\varphi) \in E^{d-1,d-1}_{M(\C) \smallsetminus Z_M(T,\varphi)(\C) }  
\quad \mbox{and}\quad 
\omega_M^\circ(T,v,\varphi) \in E^{d,d}_{M(\C)}
\]
 related by the Green equation
\[
dd^c [ \mathfrak{g}_M^\circ(T,v,\varphi) ] + \delta_{Z_M(T,\varphi)} = [ \omega_M^\circ(T,v,\varphi) ] ,
\]
and an arithmetic cycle class
\begin{equation}\label{naive arith cycle}
\widehat{Z}_M(T,v,\varphi) =
( Z_M(T,\varphi) , \mathfrak{g}_M^\circ(T,v,\varphi) ) \in \widehat{\mathrm{CH}}^d (M) .
\end{equation}
We would like to extend the definition to  include singular $T$.

Recall that we have endowed the tautological line bundle $L_\DD$ on $\DD$ with the hermitian metric $h$ of \eqref{natural metric}, and have endowed $L_\DD^\vee$ with the dual metric.    
These induce metrics on the canonical models $L_M$ and $L_M^\vee$, and so determine arithmetic cycle classes 
\[
L_M, L_M^\vee \in   \widehat{\mathrm{CH}}^1 (M)
\]
using the arithmetic Chern class map from \S III.4.2  of \cite{SouleBook}.
Of course $L_M^\vee = -L_M$.
A distinguished role is played by 
\begin{equation}\label{shifted tautological}
\widehat{\omega}^{-1}= L_M^\vee   + ( 0, - \log(2\pi e^\gamma) ) \in  \widehat{\mathrm{CH}}^1 (M).
\end{equation}
In other words, if  we endow $L_M$ with the rescaled metric  $( 2\pi e^\gamma)^{-1} h$, then \eqref{shifted tautological} is the image of its dual under the arithmetic Chern class map.
Abbreviate
\[
\Omega = \chern(L_M^\vee) \in E^{1,1}_{M(\C)}
\]
for the Chern form of  the dual of $(L_M,h)$, and note that  $\Omega$ is also the Chern form of \eqref{shifted tautological}.

\begin{remark}
Our  $L_M$ agrees with the $\mathcal{E}$ in (5.160) of \cite{GS}, but our $\widehat{\omega}$ differs from theirs by an inverse and a rescaling of metrics.
\end{remark}

\begin{remark}
The  factor of $2\pi e^\gamma$  in \eqref{shifted tautological} is needed to make the arithmetic cycle classes defined below satisfy the pullback formula of Theorem \ref{BigThmA}.  
More precisely, in the proof of  Proposition \ref{prop:key degenerate specialization} this factor will match up with the similar factor   appearing in the logarithmic expansions of Lemma \ref{lem:very nice expansions} and the specializations to the normal bundle of Proposition \ref{prop:the specializations}.   
There are other  reasons why the  particular normalization in \eqref{shifted tautological} is a natural choice, as explained in the introduction to \cite{KRY04}. 
 \end{remark}

\begin{theorem}[Garcia-Sankaran]\label{thm:the classes}
Assume  that $V$ is anisotropic. 
There are arithmetic cycle classes
\[
\widehat{Z}_M(T,v,\varphi)    \in \widehat{\mathrm{CH}}^d(M)
\]
indexed by $T\in \Sym_d(F)$,  positive definite $v\in \Sym_d(\R)$, and $K$-fixed $\Z$-valued $\varphi \in S(\widehat{V})$ satisfying the following properties.
\begin{enumerate}
\item
For fixed $T$ and $v$, the formation of $\widehat{Z}_M(T,v,\varphi)$ is linear in $\varphi$.
\item 
If  $\det(T) \neq 0$  then $\widehat{Z}_M(T,v,\varphi)$ agrees with \eqref{naive arith cycle}.

\item
If  $0_d \in \Sym_d(F)$ denotes the  zero matrix, then 
\[
\widehat{Z}_M(0_d ,v,\varphi) =   \varphi(0)\cdot   
 \underbrace {\,  \widehat{\omega}^{-1} \cdots \widehat{\omega}^{-1} } _{d \mathrm{\ times}} .
\]
\item
Assume that  $T$ and $v$ have the form
\[
T =     \begin{pmatrix}  T_0  \\ & 0_{d-r}   \end{pmatrix}    \quad  \mbox{and} \quad
 v =   {}^t\theta \cdot  \begin{pmatrix}   v_0 \\ & w   \end{pmatrix}  \cdot  \theta
\]
with  $T_0 \in  \Sym_r(F)$  nonsingular,  $v_0\in \Sym_r(\R)$ and $w \in \Sym_{d-r}(\R)$ of positive determinant, and 
\[
\theta = \begin{pmatrix}  1_r & * \\ & 1_{d-r} \end{pmatrix}  \in \GL_d(\R).
\]
If $\varphi = \varphi^{(r)} \otimes  \varphi^{(d-r)} \in S(\widehat{V}^r)\otimes S(\widehat{V}^{d-r})$  is a product of $\Z$-valued $K$-fixed Schwartz functions, then
\[
\widehat{Z}_M (T,v,\varphi ) = \widehat{Z}_M(T_0 , v_0, \varphi^{(r)} )  \cdot   \widehat{Z}_M(0_{d-r} ,w, \varphi^{(d-r)} ) .
\]
\item
For any $a \in \mathrm{GL}_d(F)$ we have 
\begin{align*}
\widehat{Z}_M(T ,v , \varphi ) & = \widehat{Z}_M ( {}^aT    ,  {}^av  , {}^a\varphi )     
\end{align*}
where  
\[
{}^aT =  {}^ta   T  a , \quad {}^av = \sigma(a^{-1})  v \sigma( {}^t a^{-1}) , \quad {}^a\varphi(x) = \varphi(x a^{-1}) .
\]
\end{enumerate}
\end{theorem}

\begin{proof}
This is a minor modification of the construction of \S 5.4 of  \cite{GS}.

We have  defined the forms \eqref{component green} only when $\det(T)\neq 0$.
If we drop this assumption,  the construction of $\omega_M^\circ(T,v,\varphi)$ still makes sense  word-for-word.
The construction of   $\mathfrak{g}_M^\circ(T,v,\varphi)$ does not, because Proposition \ref{prop:general green} only applies to regular tuples, and the tuple  $s(x\alpha)$ appearing in the definition \eqref{alpha} is not regular if $T(x)$ is a singular matrix.  Nevertheless,  Propositions 4.3 and 4.4 of \cite{GS} provide the construction of a current  $ \mathfrak{g}_M^\circ(T,v,\varphi)$ on $M(\C)$ of type $(d-1,d-1)$ satisfying the generalized Green equation
\[
dd^c  \mathfrak{g}_M^\circ(T,v,\varphi)  + \delta_{Z_M(T,\varphi)} \wedge  \Omega^{d-\mathrm{rank}(T) }  
=  \omega_M^\circ(T,v,\varphi)  . 
\]
We remark that when $\mathrm{rank}(T)<d$  the current  $\mathfrak{g}_M^\circ(T,v,\varphi)$  is not represented by a locally integrable form on $M(\C)$.

Now let $\mathfrak{g}$ be any choice of Green current for the  cycle $Z_M(T,\varphi)$ of codimension $r=\mathrm{rank}(T)$.
The  arithmetic cycle class
\[
\widehat{Z}_M(T,v,\varphi) = ( Z_M(T,\varphi),\mathfrak{g}) \cdot ( \widehat{\omega}^{-1} )^{d-r} + ( 0 , \mathfrak{g}_M^\circ(T,v,\varphi) - \mathfrak{g}\wedge \Omega^{d-r} ) 
\]
is easily seen to be independent of $\mathfrak{g}$.  This is the same definition as  (5.158) of  \cite{GS}, except that we have used the class \eqref{shifted tautological}   in place of $L_M^\vee$, and have used the current $\mathfrak{g}_M^\circ(T,v, \varphi)$ instead of the modified version of Definition 4.7 of \cite{GS}.

Properties (1) and (2) are immediate from the definitions.  Property (3) follows from 
$Z_M(0_d,\varphi) = \varphi(0)  M$ and $\mathfrak{g}_M^\circ(0_d,v,\varphi) =0$, as in  (4.43) of \cite{GS}.
Property (4) is a consequence of the relations
\begin{align*}
Z_M (T,\varphi )    & = \varphi^{(d-r)}(0) \cdot  Z_M(T_0 ,  \varphi^{(r)} )   \\
\mathfrak{g}_M^\circ(T,v,\varphi)  
& =\varphi^{(d-r)}(0) \cdot   \mathfrak{g}_M^\circ(T_0 , v_0,  \varphi^{(r)} ) \wedge \Omega^{ d-r}  + \partial A + \overline{\partial}B,
\end{align*}
for  currents $A$ and $B$ on $M(\C)$, as in Examples 2.14 and 4.8 of \cite{GS}.  Property (5) follows from
\begin{align*}
Z_M(T , \varphi ) &= Z_M ( {}^aT   , {}^a\varphi ) \\
\mathfrak{g}_M^\circ (T ,v , \varphi ) &=\mathfrak{g}_M^\circ( {}^aT    ,  {}^av  , {}^a\varphi ) ,
\end{align*}
as in  Remark 4.9 of \cite{GS}.
\end{proof}

\begin{remark}\label{rem:good a}
The arithmetic cycle classes of Theorem \ref{thm:the classes} are uniquely determined by the properties listed there.  The key point is that for any $T$ and $v$ one may find an $a\in \mathrm{GL}_d(F)$ such that the matrices  ${}^aT$ and ${}^a v$ appearing in (5) have the form described in (4).    The classes determined by such matrices are obviously determined by properties (1)-(4).
\end{remark}

We now modify  the arithmetic cycle classes of Theorem \ref{thm:the classes}.
Given data $(T,v,\varphi)$ as in that theorem, choose $a\in \GL_d(F)$ in such a way that 
\[
{}^aT =     \begin{pmatrix}  T_0  \\ & 0_{d-\mathrm{rank}(T)}   \end{pmatrix}    \quad  \mbox{and} \quad
 {}^av =   {}^t\theta \cdot  \begin{pmatrix}   v_0 \\ & w   \end{pmatrix}  \cdot  \theta
\]
have the form described in  part (4), and define 
\begin{align}
\widehat{C}_M (T,v,\varphi)    
& = \widehat{Z}_M  ({}^aT,{}^av,{}^a\varphi)  \nonumber  \\
& \quad +  \big( 0  ,  - \log( \det(w)) \cdot \delta_{Z_M(T,\varphi)} \wedge \Omega^{ d-\mathrm{rank}(T)-1} \big) . \label{best classes}
\end{align}
Note  that if $T$ is not totally positive semi-definite,  then $Z_M(T,\varphi)=0$ and the correction term disappears.
If $\det(T)\neq 0$ we understand $\det(w)=1$, so that the correction term again vanishes, leaving
\[
\widehat{C}_M (T,v,\varphi)   = \widehat{Z}_M  ({}^aT,{}^av,{}^a\varphi)=\widehat{Z}_M  ( T,  v,  \varphi).
\]

\begin{proposition}\label{prop:the mod classes}
The arithmetic cycle class $ \widehat{C}_M (T,v,\varphi) $ does not  depend on the choice of $a \in\GL_d(F)$ used in its construction.  It satisfies all the properties listed in Theorem \ref{thm:the classes}, except that now
\[
\widehat{C}_M (0_d ,v,\varphi) =   \varphi(0)\cdot   \big[ 
 \underbrace {\,  \widehat{\omega}^{-1} \cdots \widehat{\omega}^{-1} } _{d} 
 +   \big( 0 , -\log(\det(v)) \cdot  \Omega^{d-1} \big) \big] .
\]
In particular, if $\varphi=\varphi_1\otimes \cdots \cdots \otimes \varphi_d$ is a pure tensor then
\[
\widehat{C}_M (0_d ,v ,\varphi) =\widehat{C}_M(0 ,v_1,\varphi_1)  \cdots  \widehat{C}_M(0 ,v_d,\varphi_d) ,
\]
where $v_1,\ldots, v_d$ are the eigenvalues of $v$, and 
\[
\widehat{C}_M(0 ,v_i ,\varphi_i) =   
\varphi_i(0) \cdot \big[ 
 \widehat{\omega}^{-1}  +   ( 0 , -\log( v_i )  )]  \in  \widehat{\mathrm{CH}}^1 (M).
\]
\end{proposition}

\begin{proof}
For the independence of the choice of $a$, a linear algebra exercise shows that choosing a different $a$ has the effect of multiplying  both $\det(v_0)$ and  $\det(w)$ by nonzero squares in $\sigma(F)$.
Thus it suffices to show that the arithmetic cycle class
\begin{equation}\label{trivial error}
 ( 0  ,  - \log \sigma(\xi^2) \cdot \delta_{Z_M(T,\varphi)} \wedge \Omega^{ d-\mathrm{rank}(T)-1} ) \in \widehat{\mathrm{CH}}^d(M)
\end{equation}
is trivial for any $\xi \in F^\times$.  If we view  $\xi$ as a (constant) rational function on $Z_M(T,\varphi)$, it determines an arithmetic cycle 
\[
\big( i_* \mathrm{div}(\xi) , i_* [ -\log \sigma(\xi^2)] \big) = \big( 0 , -\log \sigma(\xi ^2)  \wedge \delta_{Z(T,\varphi)} \big) 
\in \widehat{Z}^{ \mathrm{rank}(T)+1 }(M),
\]
 where $i : Z_M(T,\varphi) \to M$ is the inclusion.
As in the discussion leading to Definition 1 in \S III.1.1 of \cite{SouleBook}, this arithmetic cycle  is   trivial in the arithmetic Chow group.  On the other hand, its arithmetic intersection with $d-\mathrm{rank}(T)-1$ copies of $\widehat{\omega}^{-1}$ is \eqref{trivial error}, which is therefore also trivial.
 
The remaining claims follow from  Theorem \ref{thm:the classes} and the definitions.
\end{proof}

\begin{remark}\label{rem:classes rep}
If there is no $x\in V^d$ such that $T(x)=T$, then
\[
\widehat{C}_M(T,v,\varphi ) = \widehat{Z}_M(T,v,\varphi ) =  0 .
\]
If $T$ is nonsingular, this is clear from the definitions.
The general case can be reduced   to the nonsingular case using Remark \ref{rem:good a}.
\end{remark}

\begin{remark}
Our classes \eqref{best classes} agree with those of (5.158) of \cite{GS} when $\det(T) \neq 0$.
For singular matrices they do not quite agree.  
As remarked in the proof of Theorem \ref{thm:the classes}, the classes $\widehat{Z}_M(T,v,\varphi)$ differ from the Garcia-Sankaran classes in two ways: the extra factor of $-\log(2\pi e^\gamma)$ in \eqref{shifted tautological}, and the use of the  current $\mathfrak{g}^\circ_M(T,v,\varphi)$ instead of the modified current of Definition 4.7 of \cite{GS}.  
 Using the vanishing of \eqref{trivial error}, one can see that adding the correction term in \eqref{best classes} eliminates the second of these two differences.  Thus the only difference between our $\widehat{C}_M(T,v,\varphi )$  and the classes of Garcia-Sankaran is the shifted metric in \eqref{shifted tautological}.
 \end{remark}


\subsection{The pullback formula}


We now state our main result.  The proof will occupy the rest of the paper.

Suppose we are given an orthogonal decomposition
\[
V = V_0 \oplus W
\]
with $V_0$ and $W$ of dimensions $n_0+2\ge 3$ and $m$, respectively.
Assume moreover that   $W_\tau = W\otimes_{F,\tau}\R$ is positive definite for every $\tau :F \to \R$.  
The assumptions on $V$ imposed in \S \ref{ss:ortho hermitian}  imply that $V_{0,\tau}=V_0\otimes_{F,\tau}\R$ has signature
\[
\mathrm{sig}( V_{0,\tau} ) = \begin{cases}
(n_0,2) & \mbox{if }\tau=\sigma \\
(n_0+2,0) &\mbox{if }\tau \neq \sigma.
\end{cases}
\]

The quadratic space $V_0$ therefore has its own Shimura datum $(G_0,\DD_0)$, and the the inclusion $V_0\subset V$  
induces a injection  of Shimura data 
\[
i_0 : (G_0,\DD_0) \to (G,\DD)
\]
 realizing  $\DD_0 \subset \DD$ as a codimension $m$ submanifold.    Fix a compact open subgroup $K_0 \subset G_0(\A_f) \cap K$,  and let $M_0$ be the associated Shimura variety over $F=\sigma(F)$ with complex points 
\[
M_0(\C) = G_0(\Q) \backslash \DD_0 \times G_0(\A_f) / K_0.
\]
The induced map $i_0 : M_0 \to M$ is finite and unramified.  
The Shimura variety $M_0$ has its own hermitian line bundle $L_{M_0}$, related to the one on $M$ by 
a canonical isomorphism 
\[
L_{M_0} \iso i_0^*L_M.
\]

 \begin{hypothesis}
 We assume throughout that the compact open subgroups $K_0\subset G_0(\A_f)$ and $K\subset G(\A_f)$ have been chosen so that \[i_0 : M_0 \to M\] is a closed immersion.  This is always possible, by Proposition 1.15 of \cite{Deligne}.
 \end{hypothesis}

\begin{theorem}\label{thm:pullback}
Assume that $V$ is anisotropic.
Fix an integer $1\le d \le n_0+1$ and  a $K$-fixed   Schwartz function
\[
\varphi =  \varphi_{0} \otimes \psi \in S( \widehat{V}_0^d )^{K_0} \otimes S(\widehat{W}^d) \subset S( \widehat{V}^d)
\]
with both factors  $\varphi_0$ and $\psi$ valued in $\Z$.
Recall  that \S \ref{ss:canonical cycle construction} associates to any $T\in \Sym_d(F)$ and any positive definite $v\in \Sym_d(\R)$   arithmetic cycle classes 
\begin{equation}\label{two classes}
\widehat{C}_M(T,v,\varphi) \in \widehat{\mathrm{CH}}^d(M) \qquad \mbox{and} \qquad \widehat{C}_{M_0}(T,v,\varphi_0) \in \widehat{\mathrm{CH}}^d(M_0). 
\end{equation}
The specialization to the normal bundle  
\[
\sigma_{M_0/M} : \widehat{\mathrm{CH}}^d(M) \to \widehat{\mathrm{CH}}^d( N_{M_0/M}) 
\]
of Theorem \ref{thm:arithmetic specialization} satisfies
\[
\sigma_{M_0/M} \big(  \widehat{C}_M(T,v ,\varphi) \big)   =
   \sum_{ \substack{ T_0 \in \Sym_d( F ) \\ y\in W^d  \\  T_0+T(y) = T   }  } 
  \psi (y)    \cdot \pi_0^*  \widehat{C}_{M_0}(T_0,v,\varphi_{0}) ,
\]
where $\pi_0 : N_{M_0/M} \to M_0$ is the bundle map.
\end{theorem}

Theorem \ref{thm:pullback} will be proved below.  
First, we record a corollary explaining the precise connection between the classes of \eqref{two classes}.

 \begin{corollary}\label{cor:pullback}
Keeping the notation and assumptions of Theorem \ref{thm:pullback}, the pullback 
\[
i_0^* : \widehat{\mathrm{CH}}^d(M) \to \widehat{\mathrm{CH}}^d( M_0) 
\]
satisfies
\[
i_0^* \widehat{C}_M(T,v,\varphi) =  
 \sum_{ \substack{ T_0 \in \Sym_d( F ) \\ y\in W^d  \\  T_0+T(y) = T   }  }   \psi(y)   \cdot    \widehat{C}_{M_0}(T_0,v,\varphi_{0}).
\]
 \end{corollary} 
 
 \begin{proof}
 This is immediate from Theorems \ref{thm:arithmetic specialization} and \ref{thm:pullback}, along with the injectivity of $\pi_0^*$ proved in Proposition \ref{prop:injective normal}.
 \end{proof}

 
\subsection{Specialization of degenerate cycles}
 

We now state and prove the key ingredient in the proof of Theorem \ref{thm:pullback}. This is Proposition \ref{prop:key degenerate specialization} below, which  allows us to compute the specializations to the normal bundle $N_{M_0/M}$ of those arithmetic cycles on $M$ that intersect $M_0$ improperly.

The action of $G_0(\R)$ on the pair $\DD_0 \subset \DD$ induces an action on   $N_{\DD_0/\DD}$, and the normal bundle to  $M_0 \to M$ has complex points
\[
N_{M_0/M} (\C) = G_0(\Q) \backslash N_{\DD_0/\DD} \times G_0(\A_f) / K_0 .
\]
As in \eqref{shimura component}, every $g\in G_0(\A_f)$ determines a commutative diagram
\begin{equation}\label{normal uniformization}
\xymatrix{
{  \Gamma_g  \backslash N_{\DD_0/\DD}  }  \ar[d] \ar[rrr]^{ z\mapsto (z,g) } &  & &  { N_{M_0/M} (\C) }  \ar[d]^{\pi_0}  \\
{ \Gamma_g    \backslash \DD_0   } \ar[rrr]^{ z\mapsto (z,g) }    &  & & M_0(\C) 
} 
\end{equation}
in which $\Gamma_g = gK_0g^{-1} \cap G_0(\Q)$, and the horizontal arrows are open and closed immersions.

Define complex manifolds
\[
X_0  = \bigsqcup_g \Gamma_g \backslash \DD_0 
\quad \mbox{and} \quad 
X  = \bigsqcup_g \Gamma_g \backslash \DD  ,
\]
where both unions are taken over a set of representatives  for the double quotient $G_0(\Q) \backslash  G_0(\A_f) / K_0$.   This gives a diagram of complex manifolds
\[
\xymatrix{
{  X_0  } \ar@{=} [d] \ar[r]  &  {  X } \ar[d]   \\
{ M_0(\C) }   \ar[r] & { M(\C), }  
}
\]
in which the horizontal arrows are closed immersions, and the right vertical arrow  is a  holomorphic covering of the union of all  connected components of $M(\C)$ having nonempty intersection with $M_0(\C)$.
There are canonical identifications 
\begin{equation}\label{partial normal}
N_{ M_0 / M } (\C) \iso 
N_{X_0/X} \iso \bigsqcup_g \Gamma_g \backslash N_{\DD_0/\DD}
\end{equation}
of holomorphic vector bundles on $X_0 = M_0(\C)$.  Note that $X$, unlike $X_0$,  is not (in any obvious way) the complex points of an algebraic variety.

 Fix a tuple  $y=(y_1,\ldots, y_d) \in W^d$ with linearly independent components, and a  positive definite $w\in \Sym_d(\R)$.
 As explained in \S \ref{ss:ortho hermitian},  this data determines a pair 
  \begin{equation}\label{r degen}
  ( Z_\DD(y) , \mathfrak{g}_\DD^\circ (y,w) )
   \end{equation}
   consisting of an analytic cycle  $ Z_\DD(y) \subset \DD$
  and a Green current for it, represented by a smooth form on $\DD \smallsetminus Z_\DD(y)$.
  Because the group $G_0(\Q)$ acts trivially on the subspace $W \subset V$, hence fixes $y$ componentwise, this pair   is invariant under the action of each $\Gamma_g$.  Thus it descends to each quotient $\Gamma_g \backslash \DD$, and by varying $g$ we  obtain a pair 
\begin{equation}\label{cone pair}
( Z_X(y) , \mathfrak{g}^\circ_X(y,w) )
\end{equation}
consisting of an analytic cycle $Z_X(y) \subset X$ and a Green current for it, represented by a smooth form on $X \smallsetminus Z_X(y)$.  Alternatively, rather than descending from $\DD$, one could obtain this pair by simply repeating the construction of \eqref{r degen} with $\DD$ replaced by $X$ everywhere.

  Using the constructions of \S \ref{ss:hu}, one can specialize \eqref{cone pair} to a  pair  
\begin{equation}\label{r degen spec}
(Z(\C) ,\mathfrak{g})  \define  \big(  \sigma_{X_0/X}( Z_X(y) )   ,  \sigma_{X_0/X}  ( \mathfrak{g}_X^\circ (y,w)   )  \big)
\end{equation}
on the normal bundle \eqref{partial normal}.  Equivalently, one could specialize \eqref{r degen} to obtain a $G(\Q)$-invariant pair on the normal bundle $N_{\DD_0/\DD}$, pass to the quotient by each $\Gamma_g$ in \eqref{partial normal}, and then vary $g$ to obtain a pair on $N_{X_0/X}$.

\begin{remark}
In specializing $\mathfrak{g}_X^\circ (y,w)$ to $N_{X_0/X}$,  we are using Theorem \ref{thm:hu} and \eqref{green is log} to guarantee the existence of a logarithmic expansion of $\mathfrak{g}_X^\circ (y,w)$ along $X_0 \subset X$.  
Alternatively, we will soon see that $\mathfrak{g}_X^\circ(y,w)$ is a special example of the Green form obtained from a tuple of degenerating  sections in the sense of \S \ref{ss:explicit expansions}, and so it has  logarithmic expansion of the more concrete type described in Lemma \ref{lem:nice expansions}. 
\end{remark}

The pair \eqref{r degen spec} is a subtle  thing to understand, as   the intersection of $X_0$ with $Z_X(y)$ is  improper (in fact  $X_0 \subset Z_X(y)$).
It is not even obvious that the analytically defined cycle $Z(\C)$ on \eqref{partial normal}  is algebraic, let alone that that it is defined over the reflex field.
Nevertheless, the following proposition gives us good control over it.

 \begin{proposition}\label{prop:key degenerate specialization}
 The analytic cycle $ Z(\C)   \subset N_{X_0/X}$  in \eqref{r degen spec} is the complexification of an  algebraic cycle
 $
 Z \subset N_{M_0/M},
 $
 and the equality 
\[
( Z,\mathfrak{g})   = \pi_0^*  \big(
 \underbrace {\,  \widehat{\omega}_0^{-1} \cdots \widehat{\omega}_0^{-1} } _{d} \big) 
 +  \pi_0^* \big( 0 , -\log(\det(w)) \cdot  \Omega_0^{d-1} \big) \big)
\]
holds in the codimension $d$ arithmetic Chow group of $N_{M_0/M}$.
Here $\pi_0 : N_{M_0/M} \to M_0$ is the bundle map, $\widehat{\omega}_0^{-1}$ is the analogue of \eqref{shifted tautological} on $M_0$, and $\Omega_0 \in E^{1,1}_{M_0(\C)}$ is its Chern form.
 \end{proposition}

The proof of Proposition \ref{prop:key degenerate specialization}, which occupies the remainder of this subsection, uses  the  degenerating sections of  \S \ref{ss:explicit expansions} in an essential way.
The closed immersion $\DD_0 \subset \DD$ admits a presentation of the type considered in \S \ref{ss:explicit expansions}.  
 More precisely, if we denote by $W_\DD = W_\sigma \otimes_\R \co_\DD$ the constant vector bundle on $\DD$ with fibers $W_\sigma \otimes_\R \C$,   so that $W_\DD \subset V_\DD$,  the composition
\begin{equation}\label{DNu}
 W_\DD \to  V_\DD / L_\DD^\perp \map{\eqref{dR duality}} L_\DD^\vee
\end{equation}
defines a global section of
\begin{equation}\label{DNbundle}
N_\DD=  \underline{\Hom}(  W_\DD ,  L_\DD^\vee  )
\end{equation}
with  vanishing locus  $\DD_0$.  
In particular,  
\[
 N_{\DD_0/\DD}\stackrel{\eqref{normal identification}} {\iso}
  \underline{\Hom}(  W_{\DD} ,  L_{\DD}^\vee  )|_{\DD_0}
 \iso
   \underline{\Hom}(  W_{\DD_0} ,  L_{\DD_0}^\vee  ).
\]
These isomorphisms are equivariant with respect the natural actions of $G_0(\Q)$, and so, using \eqref{partial normal}, define an isomorphism
\begin{equation}\label{X explicit normal}
N_{X_0/X} \iso   \underline{\Hom}(  W_{X_0} ,  L_{X_0}^\vee  )
\end{equation}
of holomorphic vector bundles on $X_0 = M_0(\C)$.  
Here $W_{X_0}$ and $L_{X_0}^\vee$ have the obvious meanings: they are constructed from the vector bundles $W_{\DD_0}$ and $L_{\DD_0}^\vee$ using \eqref{partial normal}.
 We now explain how  to algebraize \eqref{X explicit normal}.

\begin{lemma}\label{lem:canonical normal}
Let $W_{M_0} = W \otimes_F \co_{M_0}$ be the constant vector bundle.
There is an isomorphism
\[
N_{M_0/M} \iso \underline{\Hom}(W_{M_0} , L_{M_0}^\vee)
\]
of vector bundles on  $M_0$ that agrees, using the first  identification in \eqref{partial normal}, with  \eqref{X explicit normal} on the complex fiber.
\end{lemma}

\begin{proof}
The subspace $W \subset V$ is not stable under $G(\Q)$, so does not determine  a sub-bundle of $V_M$.
However, the decomposition $V =  V_0 \oplus W$ is stable under $G_0(\Q)$, which implies that the pullback of $V_M$  via the inclusion $M_0 \to M$  acquires  a canonical splitting
\[
V_M|_{M_0} =  V_{M_0}  \oplus W_{M_0} .
\]
This splitting is orthogonal with respect to the bilinear form \eqref{canonical bilinear}, and the restriction to $M_0$ of the flat connection \eqref{V_M connection} is identified with the sum of the analogous connection on $V_{M_0}$ and the constant connection on $W_{M_0}$ (for which the constant sections   $W \subset H^0( M_0 , W_{M_0})$  are flat).

In particular, any vector $w\in W$ determines a  flat section 
\[
 f_w \in H^0( M_0 , V_M|_{M_0}).
\]
 This section is orthogonal to the line $L_M|_{M_0} \iso L_{M_0} \subset V_{M_0}$, and so lies in the kernel of 
\[
V_M|_{M_0} \map{ \eqref{canonical dR duality}}  L_M^\vee|_{M_0} . 
\]
By  parallel transport (Proposition \ref{prop:parallel transport})  the section  $f_w$  extends to a flat section 
\[
f_w^\square \in H^0( M^\square_0 , V_M|_{M^\square_0})
\]
 over the first order infinitesimal nieghborhood $M_0^\square \subset M$ of $M_0$, whose ideal sheaf $I^2 \subset \co_M$ is the square of the ideal sheaf  $I \subset \co_M$ defining $M_0$.
 The image of $f_w^\square$ under 
 \[
V_M|_{M_0^\square} \map{ \eqref{canonical dR duality}}  L_M^\vee|_{M_0^\square} 
\]
vanishes identically along $M_0 \subset M_0^\square$, so may be viewed as a section of the coherent $\co_M$-module 
\[
 \frac{  I L_M^\vee }{    I^2 L_M^\vee  }   \iso L_M^\vee \otimes  I/I^2.
 \]

The  construction sending $w \in W$ to this last section defines a morphism of $\co_M$-modules
\[
W \otimes_F \co_M \to  L_M^\vee \otimes  I/I^2 .
\]
Restricting  to $M_0$ yields morphism
\[
W_{M_0} \to L_{M_0}^\vee \otimes N_{M_0/M}^\vee,
\]
which we rewrite as 
\[
N_{M_0/M}  \to \underline{\Hom}(W_{M_0} , L_{M_0}^\vee).
\]
By direct comparison of the constructions, one can see that this agrees with  \eqref{X explicit normal} in the complex fiber,  and hence is an  isomorphism.
\end{proof}

Each component $y_i \in W$ of the tuple $y \in W^d$ determines a global section of the constant vector bundle $W_{M_0}$ on $M_0$.  Using Lemma \ref{lem:canonical normal},
this section determines a morphism
\[
y_i : N_{M_0/M} \to L^\vee_{M_0} ,
\] 
which we pull back via the bundle map $\pi_0 : N_{M_0/M} \to M_0$ to a morphism
\[
\pi_0^* y_i : \pi_0^* N_{M_0/M} \to  \pi_0^* L^\vee_{M_0} .
\]
Now apply this morphism to the tautological section
\[
v_0 \in H^0( N_{M_0/M} , \pi_0^* N_{M_0/M} ),
\]
as in \eqref{v_0 taut}, to obtain a global section
\[
Q_i  =( \pi_0^* y_i)  ( v_0)  \in H^0( N_{M_0/M} ,  \pi_0^*L_{M_0}^\vee) .
\]

The following lemma proves the first claim of Proposition \ref{prop:key degenerate specialization}.

\begin{lemma}\label{lem:Zpresentation}
The cycle $Z(\C)  \subset N_{X_0/X}$ from \eqref{r degen spec} is the complexification of the codimension $d$ cycle  
   \[
Z = \mathrm{div}(Q_1) \cdots \mathrm{div}(Q_d)   \subset  N_{M_0/M}
   \]
   obtained by iterated proper intersection.
\end{lemma}

\begin{proof}
It suffices to prove the stated equality after pullback via each of the uniformization maps $N_{\DD_0 / \DD} \to N_{X_0/X}$ of \eqref{partial normal}, so we work over $N_{\DD_0 / \DD}$.

Each $y_i \in W$  determines a global section of the constant bundle $W_\DD$ on $\DD$, and hence, by the definition \eqref{DNbundle}, a morphism
\[
y_i : N_{\DD} \to L_\DD^\vee.
\]
  As in \eqref{degen section},   this morphism determines a degenerating section 
\[
q_i  = y_i(u)  \in H^0(\DD,L_\DD^\vee),
\]
where $u$ is the section of \eqref{DNbundle} determined by \eqref{DNu}.
On the other hand, directly comparing the constructions shows that
\begin{equation}\label{q is s}
q_i  = s(y_i),
\end{equation} 
where the right hand side is the section of $L_\DD^\vee$ defined by  \eqref{domain section}.
 Setting $q=(q_1,\ldots, q_d)$ gives the equality 
\[
Z_\DD(y)=Z_\DD(q)
\]
of  analytic cycles on $\DD$, and we have now shown that 
\begin{equation}\label{abusive Z}
Z (\C) = \sigma_{\DD_0/\DD} ( Z_\DD(q) ),
\end{equation}
where the left hand side now denotes (by abuse of notation) 
the pullback of $Z(\C)$ via  $N_{\DD_0/\DD} \to N_{X_0/X}$.

The construction  \eqref{very degen section} associates to the degenerating section $q_i$ a  section   
\[
\sigma_{\DD_0/\DD}(q_i)  \in H^0( N_{\DD_0/\DD} , \pi_0^* L_{\DD_0}^\vee ) ,
\]
and by directly comparing  the constructions we have
\begin{equation}\label{abusive Q}
Q_i = \sigma_{\DD_0/\DD}(q_i),
\end{equation}
where the left hand side  denotes (by similar abuse of notation) 
the pullback of the complexification of $Q_i$ via  $N_{\DD_0/\DD} \to N_{X_0/X}\iso N_{M_0/M}(\C)$.

By the third claim  of Proposition  \ref{prop:intersection cycle spec}, the tuple 
\[
 \sigma_{\DD_0/\DD} (q )=
 \big( \sigma_{\DD_0/\DD} (q_1) , \ldots,  \sigma_{\DD_0/\DD} (q_d) \big)
 =(Q_1,\ldots, Q_d) 
\]
is smooth, and \eqref{abusive Z} is defined by the vanishing of its components.  
Thus  $Z(\C)$ is defined by the equations $Q_1=\cdots=Q_d=0$, so is equal to 
 the intersection of the divisors of $Q_1,\ldots, Q_d$ by 
Remark \ref{rem:smooth intersection}.
\end{proof}

 As the cycle $Z \subset N_{M_0/M}$ of Lemma \ref{lem:Zpresentation} is presented to us as the proper intersection of the divisors of sections $Q_i \in H^0( N_{M_0/M} ,  \pi^*L_{M_0}^\vee)$, it is easy to construct a Green current for it.  
 Each divisor $\mathrm{div}(Q_i)$ has  a Green current $-\log (2\pi e^\gamma h(Q_i))$, and  the iterated star product 
\[
G= [ -\log (2\pi e^\gamma h(Q_1)) ] \star \cdots \star [ -\log (2\pi e^\gamma h(Q_d)) ] \in D^{d-1,d-1}_{N_{M_0/M}(\C) }
\]
 is a Green current for $Z$.
 
  This construction can be generalized.  For $\beta \in \GL_d(\R)$,  consider the tuple
  \[
  (Q_1',\ldots, Q_d') = (Q_1,\ldots, Q_d) \cdot \beta  \in H^0( N_{M_0/M}(\C) , \pi_0^*L_{M_0(\C)}^\vee )^d
  \]
of sections defined over  the complex fiber (of course they will not be defined over the reflex field $F=\sigma(F)$ unless $\beta$ is).   
Because 
\[
\mathrm{div}(Q_1')\cdots\mathrm{div}(Q_d')=Z(\C),
\]
  the iterated star product 
\[
G(\beta) = [ -\log (2\pi e^\gamma h(Q'_1)) ] \star \cdots \star [ -\log (2\pi e^\gamma h(Q'_d)) ] \in D^{d-1,d-1}_{N_{M_0/M}(\C) }
\]
is a also Green current for $Z$.

 \begin{lemma}\label{lem:twist taut}
 For any $\beta \in \GL_d(\R)$, the pullback of 
 \[
  \underbrace{ \widehat{\omega}_0^{-1} \cdots  \widehat{\omega}_0^{-1}}_{d \mathrm{\ times} } 
  + (0 , -\log|  \det( \beta) |^2 \cdot \Omega_0^{d-1} )  \in \widehat{ \mathrm{CH} }^d( M_0) 
 \]
 via the bundle map $\pi_0 : N_{M_0/M}  \to M_0$ is represented by the arithmetic cycle
 \[
 (Z,G(\beta)) \in \widehat{Z}^d( N_{M_0/M}).
\]
\end{lemma}

  \begin{proof}
By construction, $(Z,G)$ is the arithmetic intersection of the 
  \[
  ( \mathrm{div}(Q_i) , -\log (h(Q_i))) + (0, -\log(2\pi e^\gamma)) \in \widehat{ \mathrm{CH} }^1( N_{M_0/M} )
  \]
  as $i$ varies over $1\le i \le d$.  Each $Q_i$ is a section of $\pi_0^*L_{M_0}^\vee$, and so, recalling \eqref{shifted tautological}, each of these arithmetic divisors represents 
  \[
  \pi_0^*L_{M_0}^\vee + (0, -\log(2\pi e^\gamma))  = \pi_0^* \widehat{\omega}_0^{-1}.
  \]
Thus 
\[
(Z,G) = \underbrace{ \pi_0^*\widehat{\omega}_0^{-1} \cdots  \pi_0^*\widehat{\omega}_0^{-1}}_{d \mathrm{\ times} } 
 \in \widehat{ \mathrm{CH} }^d( N_{M_0/M} ) ,
\]
and the claim is true when $\beta$ is the identity matrix.

It now suffices to show that for any  $\alpha,\beta \in \GL_d(\R)$, we have
\[
(Z,G( \alpha \beta) ) = (Z,G(\alpha)) + (0 , -\log| \det( \beta) |^2 \cdot \Omega_0^{d-1} )
\]
in the arithmetic Chow group of $N_{M_0/M}$.
If $\beta$ is a permutation matrix, this follows from the usual associativity and commutativity of the star product (modulo currents of the form $\partial a+\overline{\partial}b)$.  The cases
\[
\beta = \begin{pmatrix}  1 & 1 \\  0& 1 \\ & & I_{d-2}   &  \end{pmatrix} 
 \quad \mbox{and} \quad 
 \beta = \begin{pmatrix}  \lambda \\ & I_{d-1}  \end{pmatrix}
\]
with $\lambda \in \R^\times$ follow immediately from the definition of the star product.  
The general case follows by writing $\beta$ as a product of such matrices.
  \end{proof}

\begin{proof}[Proof of Proposition \ref{prop:key degenerate specialization}]
The first claim of Proposition \ref{prop:key degenerate specialization} follows from Lemma \ref{lem:Zpresentation}.
For the second claim, factor $w=\beta\cdot {}^t\beta$ with $\beta\in \GL_d(\R)$ of positive determinant.
By  Lemma \ref{lem:twist taut}, it suffices to  prove the equality 
 \[
( Z,\mathfrak{g})   = (Z , G(\beta) ) 
\]
in the arithmetic Chow group of $N_{M_0/M}$.   
Thus we seek currents $a$ and $b$ on $N_{X_0/X}=N_{M_0/M} (\C) $ satisfying
\[
\mathfrak{g} +  \partial a + \overline{\partial} b = G(\beta) .
\]

To this end, we work with the pullbacks of $\mathfrak{g}$ and $G(\beta)$ via  
\[
N_{\DD_0/\DD}   \to \Gamma_g   \backslash N_{\DD_0/\DD}  \map{ z\mapsto (z,g) }  N_{X_0/X} 
\]
for a fixed  $g\in G_0(\A_f)$, as in  \eqref{normal uniformization}.
Recall  from \eqref{abusive Q} the equality 
\[
Q_i = \sigma_{\DD_0/\DD}(q_i)  \in H^0 ( N_{\DD_0/\DD} , \pi_0^* L^\vee_{\DD_0} ) .
\]
The final claim of Proposition \ref{prop:the specializations} implies that the pullback of $G(\beta)$ to $N_{\DD_0/\DD}$  is equal to
\begin{align*}\lefteqn{
[ -\log (2\pi h( \sigma_{\DD_0/\DD}(q_1')) ] \star \cdots \star [ -\log (2\pi e^\gamma h( \sigma_{\DD_0/\DD}(q_d')) ) ] } 
\hspace{4cm}    \\
& = \sigma_{\DD_0/\DD}( \mathfrak{g}^\circ (q_1')) \star \cdots \star \sigma_{\DD_0/\DD}( \mathfrak{g}^\circ (q_d')) ,
\end{align*}
where the $q'_i \in H^0(\DD, L_\DD^\vee)$ are the components of the tuple $q'= q\beta$.

It now follows from the second claim of Proposition \ref{prop:special star} that 
\begin{align*}
G(\beta) 
& = \sigma_{\DD_0/\DD}( \mathfrak{g}^\circ (q_1')) \star \cdots \star \sigma_{\DD_0/\DD}( \mathfrak{g}^\circ (q_d'))  \\
&= \sigma_{\DD_0/\DD}(\mathfrak{g}^\circ(q')) + \partial a +\overline{\partial} b
\end{align*}
for  currents 
\[
a = \sigma_{\DD_0/\DD} ( \mathfrak{a}(q') ) \quad \mbox{and}\quad b= \sigma_{\DD_0/\DD} ( \mathfrak{b}(q') )
\] 
on $N_{\DD_0/\DD}$. 
Here, by abuse of notation, the left hand side is the pullback of $G(\beta)$ to $N_{\DD_0/\DD}$.
As in \eqref{q is s}, we have the equality 
\[
q'=q\beta = s(y\beta) 
\]
of tuples of sections of $L_\DD^\vee$, which implies
\[
\mathfrak{g}^\circ(q') =  \mathfrak{g}^\circ(s(y\beta) )  \stackrel{ \eqref{alpha} }{=}  \mathfrak{g}_\DD^\circ( y,w).
\]
Putting everything together, and recalling \eqref{r degen spec}, we find
\begin{align*}
G(\beta) 
 & = \sigma_{\DD_0/\DD}(\mathfrak{g}^\circ(q')) + \partial a +\overline{\partial} b \\
& = \sigma_{\DD_0/\DD}(\mathfrak{g}_\DD^\circ(y,w)) + \partial a +\overline{\partial} b \\
& = \mathfrak{g} + \partial a + \overline{\partial} b
\end{align*}
as currents on $N_{\DD_0/\DD}$.

The only thing left to verify is that the currents $a$ and $b$ are $G_0(\Q)$-invariant, so descend to currents on $\Gamma_g \backslash N_{\DD_0/\DD} \subset N_{X_0/X}$.  This follows directly from their construction \eqref{star iterate}, as the components of the tuple $q'$ from which $a$ and $b$ are built are $G_0(\Q)$-invariant sections of the line bundle $L_\DD^\vee$ on $\DD$ (alternatively, one could have carried out the entirety of the proof  with  $\DD_0 \subset \DD$ replaced by $X_0 \subset X$ everywhere).
\end{proof}

 
\subsection{Proof of Theorem \ref{thm:pullback}}
 

 Keep the notation and assumptions of Theorem \ref{thm:pullback}.  
 
Assume for the moment that $\det(T)\neq 0$.
 Using the orthogonal decomposition  \[V = V_0 \oplus W,\]
each $x\in V^d$  decomposes as $x=x_0+y$,   with $x_0\in V_0^d$ and $y\in W^d$  satisfying 
\[
T(x_0) + T(y) = T(x).
\]
For a fixed  $g\in G_0(\A_f)$ we may decompose
\eqref{component cycle} and \eqref{component green} as 
\begin{equation}\label{prep decomp cycle}
Z_\DD(T,\varphi)_g 
 =    \sum_{ \substack{ T_0 \in \Sym_d(F)  \\ y\in W ^d  \\ T_0+T(y) = T  } }  \psi(y) 
 \sum_{ \substack{ x_0\in V_0^d  \\  T(x_0) = T_0   }  }   \varphi_{0}(g^{-1} x_0)    Z_\DD(x_0+y)  
 \end{equation}
and
\begin{equation}\label{prep decomp green}
\mathfrak{g}_\DD^\circ(T,v,\varphi)_g = 
  \sum_{ \substack{  T_0 \in \Sym_d(F) \\  y\in W ^d  \\ T_0+T(y) = T  } }  \psi(y) 
 \sum_{ \substack{ x_0\in V_0^d  \\  T(x_0) = T_0   }  }     \varphi_{0}(g^{-1}x_0)      \mathfrak{g}_\DD^\circ(x_0  +y , v) . 
\end{equation}
To compute their specializations to $N_{\DD_0/\DD}$, it suffices to do so for the inner summations for  fixed $T_0$, $v$,  and $y$.

This is done by reduction to the following special case.
Suppose that for some $1\le r \le d$ we have
\[
T_0   = \begin{pmatrix}  S_0 & 0 \\ 0 & 0 \end{pmatrix}\in \Sym_{d}(F) 
\]
with  $S_0 \in \Sym_r(F)$ nonsingular, and 
\[
v   = \begin{pmatrix}  v_0 & 0 \\ 0 & w \end{pmatrix}\in \Sym_{d}(\R)
\]
with  $v_0 \in \Sym_r(\R)$ and $w \in \Sym_{d-r} (\R)$.   Let  $y\in W^d$ be any tuple such that  
\[
\mathrm{rank}(T_0+T(y))=d,
\]
and write $y=(y',y'')$ as the concatenation of $y' \in W^r$ and $y''\in W^{d-r}$.

\begin{lemma}\label{lem:key calculation}
Assume $V_0$ is anisotropic, and that 
\[
\varphi_0 = \varphi_0^{(r)} \otimes \varphi_0^{ (d-r) } \in S(\widehat{V}_0^r) \otimes S(\widehat{V}_0^{d-r})
\]
with both factors in the tensor product $\Z$-valued and $K_0$-fixed.
For any fixed $g\in G_0(\Q)$, we have the equalities 
\begin{align*}\lefteqn{ 
  \sum_{ \substack{ x_0\in V_0^d  \\  T(x_0) = T_0   }  }     \varphi_{0}(g^{-1} x_0)   \cdot  \sigma_{\DD_0/\DD} ( Z_\DD (x_0  +y ) )  }  \\
&  =
  \varphi_{0}^{(d-r)} ( 0 )  \cdot 
   \pi_0^* Z_{\DD_0} (S_0,\varphi_0 , \varphi_0^{(r)} )_g  \cdot \sigma_{\DD_0/\DD}   (  Z_\DD (y'')  )
\end{align*}
and
\begin{align*} \lefteqn{
  \sum_{ \substack{ x_0\in V_0^d  \\  T(x_0) = T_0   }  }     \varphi_{0}(g^{-1} x_0)   \cdot    \sigma_{\DD_0/D} ( \mathfrak{g}_\DD^\circ( x_0+y , v ) ) } \\
& =   
  \varphi_{0}^{(d-r)} (0)  \cdot    \pi_0^* \mathfrak{g}_{\DD_0}^\circ (S_0 ,v_0 ,\varphi_0^{(r)})_g    \star    \sigma_{\DD_0/\DD} (   \mathfrak{g}_\DD^\circ ( y'' ,w  )  )  -  \partial A_g - \overline{\partial} B_g  
\end{align*}
of cycles and currents on $N_{\DD_0/\DD}$,   for some currents  $A_g$ and $B_g$  invariant under the action of the subgroup  $\Gamma_g \subset G_0(\Q)$ from  \eqref{normal uniformization}.
 On the right hand side 
 \[
 \pi_0 : N_{\DD_0/\DD} \to \DD_0
 \]
  is the bundle map, and 
 $Z_{\DD_0} (S_0,\varphi_0^{(r)})(\C)_g$ and $ \mathfrak{g}_{\DD_0}^\circ (S_0 ,v_0 ,\varphi_0^{(r)})_g$ are the cycle and current on $\DD_0$
 defined in the same way as \eqref{component cycle} and \eqref{component green}, but with the Shimura datum $(G,\DD)$ replaced by $(G_0,\DD_0)$.
\end{lemma}

\begin{proof}
Given $x \in V^d$, abbreviate $x' \in V^r$ and $x''\in V^{d-r}$ for the tuples formed from the first $r$ and final $d-r$ components of $x$.

For any $x_0 \in V_0^d$ satisfying $T(x_0)=T_0$ we have $T(x_0') = S_0$ and $T(x_0'') =0$.  Hence  $x_0'' =0$ by our assumption that $V_0$ is anisotropic, and  $x_0+y\in V^d$ is the concatenation of $x_0'+y'\in V^r$ and $y''  \in W^{d-r} \subset V^{d-r}$.  
As in the discussion surrounding \eqref{domain section}, these tuples determine tuples of sections
\[
p = s(  x_0' +y'  )  \in H^0(\DD , L_\DD^\vee)^r
\quad \mbox{and} \quad 
q= s(  y''  )  \in H^0(\DD , L_\DD^\vee)^{d-r},
\]
whose concatenation is $(p,q)=s( x_0+y)$.
These satisfy the assumptions imposed in \S \ref{ss:explicit expansions}.  More precisely:
\begin{enumerate}
\item
The restriction 
$
p|_{\DD_0} = s(x_0') \in  H^0(\DD_0 , L_{\DD_0}^\vee)^r
$
 is  the tuple of sections  formed from $x_0'\in V_0^r$.  As $T(x_0') =S_0$ is nonsingular, this restriction is again smooth, and 
 \[
 Z_{\DD_0}( p|_{\DD_0} ) = Z_{\DD_0}(x_0').
 \]
 
 \item
 As explained in the discussion surrounding \eqref{q is s},   the components of  $q$ are degenerating sections in the sense of \S \ref{ss:explicit expansions}.
\end{enumerate}

Thus  Proposition \ref{prop:intersection cycle spec} applies, and shows that
\begin{align*}
\sigma_{\DD_0/\DD}(   Z_\DD(x_0+y) )   
& = \pi_0^*Z_{\DD_0}(x_0') \cdot  \sigma_{\DD_0/\DD}(    Z_\DD(y'')  ) .
\end{align*}
Now sum both sides of this last equality over all $x_0\in V_0^d$ with $T(x_0)=T_0$. 
As  $x_0''=0$ for every such $x_0$, that sum can be replaced by the sum over all $x_0' \in V_0^r$ satisfying $T(x_0')=S_0$.  
The result is
\begin{align*}\lefteqn{ 
  \sum_{ \substack{ x_0\in V_0^d  \\  T(x_0) = T_0   }  }     \varphi_{0}(g^{-1} x_0)   \cdot  \sigma_{\DD_0/\DD} ( Z_\DD (x_0  +y ) )  }  \\
 & =   \sum_{ \substack{ x'_0\in V_0^r  \\  T(x_0) = S_0   }  }     \varphi^{(r)}_{0}(g^{-1} x_0') \varphi^{(d-r)}_0(0)   \cdot  \pi_0^*Z_{\DD_0}(x_0') \cdot  \sigma_{\DD_0/\DD}(    Z_\DD(y'')  ),
 \end{align*}
proving the first claim of the proposition.

For the specialization of Green forms, choose the $\alpha \in \GL_d(\R)$ of \eqref{alpha} in the block diagonal form
\[
\alpha = \begin{pmatrix} \alpha_0 \\ &  \beta \end{pmatrix} 
\]
with $\alpha_0\in \Sym_r(\R)$ and $\beta\in \Sym_{d-r}(\R)$ of positive determinant.
 By definition, $\mathfrak{g}_\DD^\circ(x_0+y, v)$ is the Green current associated to the tuple of sections
 \[
 s( x_0\alpha + y\alpha) \in H^0(\DD , L_\DD^\vee)^d .
 \]
Writing this as the concatenation of $p\alpha_0$ and $q\beta$, Propositions \ref{prop:the specializations} and  \ref{prop:special star} imply
\begin{align*}
\sigma_{\DD_0/\DD} ( \mathfrak{g}_\DD^\circ( x_0  +y , v  )  ) 
 & =   \pi_0^*  \mathfrak{g}^\circ_{\DD_0}( x_0' ,v_0 )  \star   \sigma_{\DD_0/\DD}( \mathfrak{g}_\DD^\circ( y'' , w  )  )   \\
& \quad  - \partial   \sigma_{\DD_0/\DD}(  A( p \alpha_0 ; q\beta )  )      - \overline{\partial} \sigma_{\DD_0/\DD} ( B( p\alpha_0  ; q \beta  ) ) .
\end{align*}
As above, summing both sides  over all $x_0 \in V_{0}^d$ with $T(x_0) = T_0$ proves the second claim of the proposition.
\end{proof}

The proof of Theorem \ref{thm:pullback} will now proceed in two steps; first assuming $\det(T)\neq 0$, and then without this assumption.

\begin{proof}[Proof of Theorem \ref{thm:pullback}:  the nonsingular case]
Assume $\det(T)\neq 0$, so that 
\begin{align*}
\widehat{C}_M (T,v,\varphi) =
\widehat{Z}_M (T,v,\varphi) =   \big( Z_M(T,v,\varphi) , \mathfrak{g}_M^\circ(T,v,\varphi) \big)  
\end{align*}

For a given  $T_0 \in \Sym_d(F)$ and  $y\in W^d$ satisfying 
$
T_0+T(y) =T,
$
 abbreviate  
$
r=\mathrm{rank}(T_0),
$
 choose   $a \in \mathrm{GL}_d(F)$ in such a way that the matrices 
\[
{}^aT_0 = {}^ta   T_0   a  ,\qquad  {}^av= \sigma(a^{-1})  v \sigma( {}^t a^{-1} )
\]
 have the form
\begin{equation}\label{good a}
{}^aT_0   = \left( \begin{matrix}  S_0  \\ & 0_{d-r}  \end{matrix} \right) ,\qquad 
{}^av = {}^t \theta \cdot  \left( \begin{matrix}   v_0 \\ & w   \end{matrix} \right)  \cdot \theta
\end{equation}
of part (4) of Theorem \ref{thm:the classes}, with $\det(S_0)\neq 0$.
 Decompose 
\[
{}^a\varphi_0 = \sum_i  \Phi_i^{(r)}\otimes \Phi_i^ {(d-r)}  \in S(\widehat{V}_0^r) \otimes S(\widehat{V}_0^{d-r} ) 
\]
as a sum of pure tensors, with all Schwartz functions appearing here $\Z$-valued and $K_0$-fixed.

Fix a $g\in G_0(\A_f)$, and let $\Gamma_g \subset G_0(\Q)$ be the subgroup from  \eqref{normal uniformization}.
It follows from $T(x_0  a) = {}^a T(x_0)   $ that 
 \[
 \sum_{ \substack{ x_0\in V_0^d  \\  T(x_0) = T_0   }  }   \varphi_{0}(g^{-1} x_0)    Z_\DD(x_0+y)  
 =
  \sum_{ \substack{ x_0\in V_0^d  \\  T(x_0) = {}^aT_0   }  }   {}^a\varphi_{0}(g^{-1} x_0)    Z_\DD(x_0+ya)  
 \]
as $\Gamma_g$-invariant cycles on $\DD$.   Specializing both sides to $N_{\DD_0/\DD}$ and using  Lemma \ref{lem:key calculation} yields the equality
\begin{align}\lefteqn{
 \sum_{ \substack{ x_0\in V_0^d  \\  T(x_0) = T_0   }  }  
   \varphi_{0}(g^{-1} x_0)    \sigma_{\DD_0/\DD}  (   Z_\DD (x_0+y)   )   }  \nonumber \\
& = \sum_i \Phi_i^{(d-r)}( 0 )   \pi_0^* Z_{\DD_0}(     S_0  ,\Phi_i^{(r)}     )_g \cdot  \sigma_{\DD_0/\DD} ( Z_\DD ( \bar{y} ) )   \label{cycle pullback}
\end{align}
of $\Gamma_g$-invariant cycles on $N_{\DD_0/\DD}$, 
where 
\[
\bar{y} =( \bar{y}_1,\ldots, \bar{y}_{d-r} ) \in W^{d-r}
\]
 consists of the final $d-r$ components of $ya$.
The same reasoning shows that 
\begin{align}\lefteqn{
\sum_{ \substack{ x_0\in V_0^d  \\  T(x_0) = T_0   }  }  
   \varphi_{0}(g^{-1} x_0) 
\sigma_{\DD_0/D} ( \mathfrak{g}_\DD^\circ (x_0+y , v ) )  }   \nonumber  \\
& =  \sum_i \Phi_i^{(d-r)}( 0 )     \pi_0^* \mathfrak{g}_{\DD_0}^\circ(S_0 ,v_0, \Phi_i^{(r)} )  
\star  \sigma_{\DD_0/\DD}( \mathfrak{g}_\DD^\circ( \bar{y}  ,w )  )    -  \partial A - \overline{\partial} B , 
\label{green pullback}
\end{align}
where  the $\Gamma_g$-invariant currents $A$ and $B$ on $N_{\DD_0/\DD}$ depend on $T_0$, $y$, and the choice of $a$.

Now sum both sides of the equality  \eqref{cycle pullback} over all $T_0 \in \Sym_d(F)$ and  $y\in W^d$ for which $T_0 + T(y) =T$ to obtain
\begin{align*} \lefteqn{ 
 \sigma_{\DD_0/\DD}  (   Z_\DD(T,\varphi) _g  )  }  \\
 &  \stackrel{ \eqref{prep decomp cycle} }{ = } 
 \sum_{ \substack{ T_0 \in \Sym_d(F)  \\ y\in W ^d  \\ T_0+T(y) = T  } }  \psi(y) 
  \sum_{ \substack{ x_0\in V_0^d  \\  T(x_0) = T_0   }  }  
   \varphi_{0}(g^{-1} x_0)    \sigma_{\DD_0/\DD}  (   Z_\DD (x_0+y)   )        \\
&   \stackrel{  \eqref{cycle pullback}   }{ = }      
 \sum_{ \substack{ T_0 \in \Sym_d(F)  \\ y\in W ^d  \\ T_0+T(y) = T  } }  \psi(y) 
 \sum_i \Phi_i^{(d-r)}( 0 )    \pi_0^* Z_{\DD_0}(     S_0  ,\Phi_i^{(r)}     ) _g \cdot  \sigma_{\DD_0/\DD} ( Z_\DD ( \bar{y} ) )   .
\end{align*}
Note that in the inner sum the data $S_0$,  $\bar{y}$,   $r=\mathrm{rank}(T_0)$, and the Schwartz functions  $\Phi_i$  all    depend on $T_0$, $y$, and a choice of $a \in \mathrm{SL}_d(F)$ as in \eqref{good a}.
These are equalities of $\Gamma_g$-invariant analytic cycles on $N_{\DD_0/\DD}$.
By descending to $\Gamma_g \backslash N_{\DD_0/\DD} \subset N_{M_0/M}(\C)$ and then varying $g \in G_0(\A_f)$, we deduce the analogous equality 
\begin{align*} \lefteqn{ 
 \sigma_{M_0/M}  (   Z_M(T,\varphi)    )    }  \\
 &  =     
 \sum_{ \substack{ T_0 \in \Sym_d(F)  \\ y\in W ^d  \\ T_0+T(y) = T  } }  \psi(y) 
 \sum_i \Phi_i^{(d-r)}( 0 )    \pi_0^* Z_{M_0}(     S_0  ,\Phi_i^{(r)}     )   \cdot  \sigma_{X_0/X} ( Z_X ( \bar{y} ) )   .
\end{align*}
of cycles on $N_{M_0/M}$.  Here $\sigma_{X_0/X} ( Z_X ( \bar{y} ) )$ is the specialization to the normal bundle 
\begin{equation}\label{true to cone}
N_{X_0/X} \iso N_{M_0/M}(\C)
\end{equation}
 of the analytic cycle $Z_X( \bar{y} ) \subset X$ associated to $\bar{y} \in W^{d-r}$ as in  \eqref{cone pair}.
It  is algebraic and defined over the reflex field by Proposition \ref{prop:key degenerate specialization}.

The same reasoning, using \eqref{prep decomp green} and \eqref{green pullback} in place of \eqref{prep decomp cycle} and \eqref{cycle pullback}, gives the equality of currents 
\begin{align*}\lefteqn{
 \sigma_{ M_0 / M} (  \mathfrak{g}_M^\circ(T,v,\varphi)   )   }  \\
   & = 
  \sum_{ \substack{  T_0 \in \Sym_d(F) \\  y\in W ^d  \\ T_0+T(y) = T  } }  \psi(y) 
 \sum_i \Phi_i^{(d-r)}( 0 )     \pi_0^* \mathfrak{g}_{M_0}^\circ(S_0 ,v_0, \Phi_i^{(r)} )  
\star  \sigma_{X_0/X}( \mathfrak{g}_X^\circ( \bar{y}  ,w )  ) 
\end{align*}
on \eqref{true to cone}, modulo currents of the form $\partial A$ and $\overline{\partial} B$.
Here $\sigma_{X_0/X}(  \mathfrak{g}_X^\circ ( \bar{y} , w ) ) $ is the specialization to \eqref{true to cone} of the Green current  
$\mathfrak{g}_X^\circ ( \bar{y} , w )$ associated to $\bar{y} \in W^{d-r}$ and $w \in \Sym^{d-r} (\R)$ as in  \eqref{cone pair}.
As in the previous paragraph,   in the inner sum the data $S_0$,  $\bar{y}$,   $r=\mathrm{rank}(T_0)$, $v_0$, $w$, and the Schwartz functions  $\Phi_i$  all    depend on $T_0$, $y$, and a choice of $a \in \mathrm{SL}_d(F)$ as in \eqref{good a}.

Passing to the arithmetic Chow group of $N_{M_0/M}$, the above equalities show that
\begin{align}  \lefteqn{ 
\sigma_{M_0/M} ( \widehat{Z}_M(T,v,\varphi) )   }   \nonumber  \\
& = 
  \sum_{ \substack{  T_0 \in \Sym_d(F) \\ y\in W^d  \\ T_0 + T(y) = T  }  }  
 \psi(y)   \sum_i 
 \pi_0^* \widehat{Z}_{M_0}(S_0,v_0,\Phi_i^{(r)}) \cdot (Z_i ,\mathfrak{g}_i ),  \label{nondegen completion}
\end{align}
where each arithmetic cycle 
\[
( Z_i ,\mathfrak{g}_i ) = \Phi_i^{(d-r)}( 0 )  \,   \big( \sigma_{X_0/X} ( Z_X ( \bar{y} ) )   ,    \sigma_{X_0/X}( \mathfrak{g}_X^\circ( \bar{y}  ,w )  )    \big) 
\]
in the sum depends on $T_0$, $y$, and a choice of $a \in \mathrm{SL}_d(F)$ as in \eqref{good a}.

Loosely speaking, the above decomposition  \eqref{nondegen completion} separates  the parts of the specialization to the normal bundle that arise from proper intersection betwen $Z_M(T,\varphi)$ and $M_0$ from those parts that arise from improper intersection, with the improper parts corresponding to the  various $(Z_i ,\mathfrak{g}_i)$.

We now come to the central point of the proof:  Proposition \ref{prop:key degenerate specialization} tells us that   each $(Z_i,\mathfrak{g}_i)$ is equal to the pullback via $\pi_0 : N_{M_0/M} \to M_0$ of the arithmetic cycle class
\begin{align*}\lefteqn{
 \widehat{C}_{M_0} ( 0_{d-r} ,w , \Phi_i^{(d-r)} )  } \\ 
 & =  \Phi_i^{(d-r)}(0)  \cdot  \Big[ 
 \underbrace {\,  \widehat{\omega}_0^{-1} \cdots \widehat{\omega}_0^{-1} } _{d-r \mathrm{\ times}} 
 +   \big( 0 , -\log(\det(w)) \cdot  \Omega_0^{d-r-1} \big) \Big]
\end{align*}
of Proposition \ref{prop:the mod classes}, where $r=\mathrm{rank}(T_0)$.   
Hence the inner sum in \eqref{nondegen completion} simplifies to 
   \begin{align*}\lefteqn{
 \sum_i  \pi_0^* \widehat{Z}_{M_0}(S_0 , v_0 , \Phi_i^{(r)} )  \cdot    (Z_i,\mathfrak{g}_i  )  } \\
&    = \sum_i \pi_0^* \widehat{C}_{M_0}(S_0 , v_0 , \Phi_i^{(r)} )  \cdot \pi_0^*  \widehat{C}_{M_0}(0_{d-r} ,w, \Phi_i^ {(d-r)}  )  \\
& =  \pi_0^* \widehat{C}_{M_0}( {}^aT_0 , {}^av , {}^a\varphi_0 ) \\
& =   \pi_0^*\widehat{C}_{M_0}( T_0 , v , \varphi_0 ) .
  \end{align*}
Plugging this back into \eqref{nondegen completion} completes the proof of Theorem \ref{thm:pullback} when $\det(T)\neq 0$.
\end{proof}

\begin{proof}[Proof of Theorem \ref{thm:pullback}: the general case]
Now let $T\in \Sym_d(F)$ be arbitrary, and set $r=\mathrm{rank}(T)$.
Using Remark \ref{rem:good a} and Proposition \ref{prop:the mod classes}, one immediately  reduces to the case in which  
\[
T =     \begin{pmatrix}  S  \\ & 0_{d-r}   \end{pmatrix}    \quad  \mbox{and} \quad
 v =   {}^t\theta \cdot  \begin{pmatrix}   v_0 \\ & w   \end{pmatrix}  \cdot  \theta
\]
 as in part (4) of Theorem \ref{thm:the classes}, with $S\in \Sym_r(F)$  nonsingular.
 We may also assume that the factors in $\varphi = \varphi_0 \otimes \psi$ admit further factorizations
 \begin{align*}
 \varphi_0  & =\varphi^{(r)}_0\otimes \varphi^{(d-r)} _0 \in S(\widehat{V}_0^r )  \otimes S(\widehat{V}_0^{d-r} ) \\
  \psi  & =\psi^{(r)}\otimes \psi^{(d-r)}  \in S(\widehat{W}^r )  \otimes S(\widehat{W}^{d-r} ) ,
 \end{align*}
 so that Proposition \ref{prop:the mod classes} implies
\begin{equation}\label{rank splitting}
\widehat{C}_M (T,v,\varphi) = \widehat{C}_M(S , v_0,\varphi^{(r)} )   \cdot   \widehat{C}_M(0_{d-r} ,w,\varphi^{(d-r)}  ) 
\end{equation}
with $\varphi^{(r)} = \varphi_0^{(r)} \otimes \psi^{(r)}$,  and similarly with $r$ replaced by $d-r$.

It is clear from  from the definitions that   pullback via $i_0 : M_0 \to M$ satisfies
\[
i_0^* \widehat{\omega}^{-1} = \widehat{\omega}_0^{-1},
\]
and so Proposition \ref{prop:the mod classes}  and Theorem \ref{thm:arithmetic specialization} imply
\[
\sigma_{M_0/M} \big(  \widehat{C}_M( 0 _{d-r} ,w,\varphi^{(d-r} ) \big) 
=\psi^{(d-r)} (0) \cdot  \pi_0^* \widehat{C}_{M_0}( 0 _{d-r} ,r,\varphi_0^{(d-r)} ).
\]
We have already proved  Theorem \ref{thm:pullback}  for the nonsingular matrix $S$,  so
\[
\sigma_{M_0/M} \big(    \widehat{C}_M(S , v_0,\varphi^{(r)} )    \big) 
= 
 \sum_{  \substack{ S_0 \in \Sym_r(F)  \\  y \in W^r \\ S_0+T(y) = S  }  }   \psi^{(r)}(y) 
 \cdot 
 \pi_0^*  \widehat{C}_{M_0}(S_0 , v_0, \varphi_0^{(r)} )  .
\]
Specialization to the normal bundle commutes with arithmetic intersection (this is immediate from Theorem \ref{thm:arithmetic specialization} and the fact that pullbacks commute with arithmetic intersection), and so the specialization  of \eqref{rank splitting}  is equal to the pullback via $\pi_0 : N_{M_0/M} \to M_0$ of 
\begin{equation}\label{final comparison 1}
 \sum_{  \substack{ S_0 \in \Sym_r(F)  \\  y \in W^r \\ S_0+T(y) = S  }  }   \psi^{(r)}(y)  \cdot \psi^{(d-r)} (0) 
 \cdot 
  \widehat{C}_{M_0}(S_0 , v_0, \varphi_0^{(r)} )   \cdot     \widehat{C}_{M_0}(0_{d-r} ,w , \varphi_0^{(d-r)} )  .
\end{equation}

To complete the proof, we must show that \eqref{final comparison 1} is equal to
\begin{equation}\label{final comparison 2}
 \sum_{  \substack{ T_0 \in \Sym_d(F)  \\   y \in W^d   \\   T_0+T(y) = T  }    }
\psi(y)   \cdot  \widehat{C}_{M_0}(T_0,v,\varphi_0 ) .
\end{equation}
If  the $(T_0,y)$ term  in \eqref{final comparison 2} is nonzero then,  by Remark \ref{rem:classes rep}, there is an    $x\in V_0^d$  such that $T(x)=T_0$. 
The tuple  $x+y\in  V^d$ then satisfies $T(x+y) = T$, and so its  $i^\mathrm{th}$ component  is isotropic for  $r<  i \le d$.
 As we have assumed that $V$ is anisotropic, we deduce that $y$ has the form
 \[
 y=(y_1,\ldots, y_r,0,\ldots,0).
 \]
 It then follows from $T_0+T(y)=T$ that
 \[
 T_0 =     \begin{pmatrix}  S_0  \\ & 0_{d-r}   \end{pmatrix} 
 \]
 for some $S_0 \in \Sym_r(F)$, and we know from Proposition \ref{prop:the mod classes} that 
 \[
  \widehat{C}_{M_0}(T_0,v,\varphi_0 ) =  \widehat{C}_{M_0}(S_0 , v_0, \varphi_0^{(r)} )   \cdot     \widehat{C}_{M_0}(0_{d-r} ,w , \varphi_0^{(d-r)} ) ) .
 \]
   Thus  in \eqref{final comparison 2} we may replace the sum over $T_0 \in \Sym_d(F)$   with a sum over  $S_0\in \Sym_r(F)$,    replace the sum over  $y\in W^d$ with a sum over $y\in W^r$, and the result is \eqref{final comparison 1}.
 \end{proof}

\bibliographystyle{amsalpha}


\providecommand{\bysame}{\leavevmode\hbox to3em{\hrulefill}\thinspace}
\providecommand{\MR}{\relax\ifhmode\unskip\space\fi MR }
\providecommand{\MRhref}[2]{%
  \href{http://www.ams.org/mathscinet-getitem?mr=#1}{#2}
}
\providecommand{\href}[2]{#2}

\end{document}